\documentclass{amsart}
\usepackage{amsmath ,amssymb ,amsthm, mathrsfs, stmaryrd, dsfont}
\usepackage{bm}
\usepackage[all]{xy}
\usepackage[pdftex]{graphicx} % replace `dvipdfmx' by `pdftex' when updating on arXiv, which is available only on arXiv. 
\usepackage[pdftex ,pdfborder={0 0 0}]{hyperref} % samely
\usepackage{tikz, tikz-cd}
\usepackage{shadethm}
\usepackage{url}
\DeclareMathAlphabet\mathbfcal{OMS}{cmsy}{b}{n}

\theoremstyle{definition}
\newtheorem{defin}{Definition}[section]

\newtheorem{thm}[defin]{Theorem}
\newtheorem{quest}[defin]{Question}
\newtheorem{conj}[defin]{Conjecture}
\newtheorem{lem}[defin]{Lemma}
\newtheorem{prop}[defin]{Proposition}
\newtheorem{cor}[defin]{Corollary}
\newtheorem*{organ}{Organization}
\newtheorem*{acknow}{Acknowledgements}
\newtheorem*{mainthm}{Main Theorem}
\newtheorem*{mainprop}{Key Observation}

\theoremstyle{remark}
\newtheorem{rem}[defin]{Remark}
\newtheorem{exam}[defin]{Example}
\newtheorem{postrem}{Postscript Remark}

\newcommand{\F}{\mathcal{F}}
\newcommand{\G}{\mathcal{G}}
\newcommand{\K}{\mathbfcal{K}}

\newcommand{\U}{\mathcal{U}}

\newcommand{\M}{\mathcal{M}}

\newcommand{\op}{\text{op}}
\newcommand{\id}{\mathrm{id}}
\newcommand{\Sets}{\textbf{Sets}}

\newcommand{\eff}{\text{eff}}
\newcommand{\des}{\text{des}}
\newcommand{\Hom}{\mathrm{Hom}}
\newcommand{\Holo}{\mathrm{Holo}}

\newcommand{\Ric}{\text{Ric}}
\newcommand{\Cpx}{\mathds{C}\mathbf{an}}
\newcommand{\Obj}{\mathrm{Obj}}
\newcommand{\Mor}{\mathrm{Mor}}
\newcommand{\Ham}{\mathrm{Ham}}
\newcommand{\sqddbar}{\sqrt{-1} \partial \bar{\partial}}
\newcommand{\intg}{{\mathrm{int}}}

\title{The moduli space of Fano manifolds with K\"ahler--Ricci solitons}

\author{Eiji Inoue}

\address{Graduate School of Mathematical Sciences, the University of Tokyo \endgraf 3-8-1 Komaba, Meguro, Tokyo 153-8914, Japan. }

\email{eijinoe@ms.u-tokyo.ac.jp}

\begin{document}

\begin{abstract}
We construct a canonical Hausdorff complex analytic moduli space of Fano manifolds with K\"ahler--Ricci solitons. 
This enlarges the moduli space of Fano manifolds with K\"ahler--Einstein metrics. 
We discover a moment map picture for K\"ahler--Ricci solitons, and give complex analytic charts on the topological space consisting of K\"ahler--Ricci solitons, by studying differential geometric aspects of this moment map. 
Some stacky words and arguments on Gromov--Hausdorff convergence help to glue them together in the holomorphic manner. 
\end{abstract}

\maketitle

\tableofcontents

\section{Introduction}

In this paper, we give a theoretical framework and method for construction of moduli space of Fano manifolds with K\"ahler--Ricci solitons. 
We begin with backgrounds and motivations for this new-type moduli space. 

In the celebrated paper \cite{FS2}, Fujiki and Schumacher constructed the complex analytic moduli spaces of all compact smooth (polarized) Calabi--Yau ($K_X \equiv 0$) and canonically polarized manifolds ($K_X > 0$), as a higher dimensional analogue of the moduli spaces of Riemann surfaces of genus $g=1$ and $g \ge 2$, respectively. 

In contrast to these cases, it is known that `the moduli space of all Fano manifolds' in a primitive sense behaves pathologically; it does not enjoy the $T_1$-separation axiom, in particular, it does not admit any nice geometric structure like a complex analytic structure. 
Indeed, the separation is obstructed by the existence of iso-trivial degenerations of Fano manifolds: there are (many) families $\mathcal{X} \to \Delta$ of Fano manifolds which is biholomorphically trivial over $\Delta^* = \Delta \setminus \{ 0 \}$ and whose central fibre $\mathcal{X}_0$ is \textit{not} biholomorphic to the general fibres; for example, small deformations of the Mukai--Umemura threefold (cf. \cite[Chapter 7]{TianBook}) and the unique $G_2$-horospherical Fano manifold with Picard number one (cf. \cite{PP} and Example \ref{horo-example}) give such examples. 
In other words, the moduli space of all Fano manifolds does not even exist in the complex analytic framework. 

Still, in view of the work of \cite{FS2}, one can imagine or hope that the existence of some `canonical metrics' on Fano manifolds may play a role to ensure the separation property of the moduli space or stack. 
For Fano manifolds, K\"ahler--Einstein metric is a candidate for such `canonical metrics'. 
However, it is known that Fano manifolds do not always admit K\"ahler--Einstein metrics, while Calabi--Yau and canonically polarized manifolds always admit K\"ahler--Einstein metrics, as observed in \cite{Mat1, Fut}. 
We must exclude `unstable' Fano manifolds in some sense. 

Detecting a geometric condition of Fano manifolds equivalent to the existence of K\"ahler--Einstein metrics was a long standing problem and a conjecture on this concern was settled as the Yau--Tian--Donaldson conjecture for these two decades. 
Recently, Chen--Donaldson--Sun and Tian \cite{CDS, Tian2} broke through this problem (for Fano manifolds): the existence of K\"ahler--Einstein metrics on a Fano manifold is equivalent to the K-stability of the Fano manifold, which is a pure algebro-geometric condition for polarized variety. 

After this breakthrough, in the spirit that the existence of canonical metrics plays a role for the separation, the (algebro-geometric) moduli space of Fano manifolds \textit{with K\"ahler--Einstein metrics} was constructed in \cite{OSS, Oda1, Oda2, LWX1} as an algebraic space within a unified theoretical framework (not defeating one by one). 
Different from the case of \cite{FS2}, even the dimension of the automorphism groups of Fano manifolds may jump along deformation of complex structures. 
So they constructed the moduli space of K\"ahler-Einstein Fano manifolds by rather different new technologies from \cite{FS2}, while sharing the same spirit with \cite{FS2} on the philosophical reason for the separation. 
After the construction of the moduli space, Li--Wang--Xu \cite{LWX2} proves the quasi-projectivity of the moduli space. 
There are also more intensive studies as \cite{SS, LiuXu} on the relation with a literal GIT construction, which a priori depends on some ad hoc data such as an embedding of varieties into a fixed projective space, for some special cases. 

This additional stability assumption `\textit{with K\"ahler--Einstein metrics}' is enjoyable for interest in particular examples: there are various important examples of Fano manifolds admitting K\"ahler--Einstein metrics. 
However, on the other hand, it is also known that there are many examples of Fano manifolds who do not admit any K\"ahler--Einstein metrics; even the one point blowing up of $\mathbb{C}P^n$. 

Philosophically, constructing moduli spaces of varieties in the schematic/complex analytic category within a unified theoretical framework can be regard as giving a (schematic/complex analytic method for) classification of them. 
From a viewpoint of MMP, Odaka and Okada conjectured in \cite{OO} that every \textit{smooth} Fano manifold with Picard number one, which is one of the final outcome of the MMP, is \textit{K-semistable}, so that they are members of the moduli space of K\"ahler--Einstein Fano manifolds and hence are classified. 
However, (infinitely) many counter-examples of this conjecture are discovered by Fujita \cite{Fujit} and Delcroix \cite{Del}. 
We face that the assumption \textit{`with K\"ahler--Einstein metrics'} is restrictive for our interest on this classification concern. 

In this paper, we extend the moduli space of Fano manifolds with K\"ahler--Einstein metrics to \textit{the moduli space of Fano manifolds with K\"ahler--Ricci solitons}. 
K\"ahler--Ricci soliton, which consists of a K\"ahler metric and a holomorphic vector field, is a natural generalization of K\"ahler--Einstein metrics from the viewpoint of K\"ahler--Ricci flow. 
The uniqueness of K\"ahler--Ricci soliton modulo the identity component of the biholomorphism group is known by \cite{TZ2} as in the case of K\"ahler--Einstein metrics \cite{BM}, so that we can regard K\"ahler--Ricci soliton as a kind of `canonical metrics' on Fano manifolds. 
The equivalence with K-stability as developed in \cite{DT, Tian1, Ber} and \cite{CDS, Tian2} for K\"ahler-Einstein metrics are also covered for the case of K\"ahler--Ricci soliton in \cite{Xio, BW} and \cite{DatSze}. 
There are large amount of known examples of Fano manifolds admitting K\"ahler--Ricci solitons; Delcroix's infinite series of counter-examples of Odaka--Okada conjecture admit K\"ahler--Ricci solitons, while they do not admit K\"ahler--Einstein metrics. 

We usually characterize moduli spaces in the schematic/complex analytic framework by a universal property. 
A category consisting of families of which we intend to construct the moduli space, which is usually called the \textit{moduli stack}, is a convenient and essential tool for describing the universal property. 
In our moduli problem, we do not work with the usual moduli stack consisting of the usual families of Fano manifolds. 
In order to ensure the separation, and technically in order to apply GIT method, we instead consider another new moduli stack $\K (n)$ consisting of some families of \textit{pairs} $(X, \xi')$ of $n$-dimensional Fano manifolds and holomorphic vector fields, which is natural in view of the theory of K\"ahler--Ricci soliton. 
The moduli stack $\K (n)$ is furthermore divided into clopen (closed and open, but not necessarily connected) sub-stacks $\K_{T, \chi}$, where the associated holomorphic vector fields are deformed holomorphically. 
As we must review the theory of K\"ahler--Ricci soliton (in section 2) before explaining this unfamiliar moduli stack, here we do not explain the detail and postpone the precise description/definition until section 2 and Definition \ref{moduli stack}. 
The author hopes that Appendix A in this article helps the readers unfamiliar to stacks to grasp some fundamental generalities on stacks over the category of complex spaces. 
Example \ref{horo-example} explains that this formulation of the moduli stack is essential and the reason why the usual stack does not serve our purpose. 
The readers will see in Remark \ref{moduli set} that the change of our moduli stacks does \textit{not affect the sets} of what we intend to parametrize ($X$ or $(X, \xi')$) and these are naturally identified to each other (only) \textit{as sets}. 

Let $\mathcal{KR}_{GH} (n)$ be the set of biholomorphism classes of $n$-dimensional Fano manifolds admitting K\"ahler--Ricci solitons. 
We can endow $\mathcal{KR}_{GH} (n)$ with a natural topology induced by the `complexified' Gromov--Hausdorff convergence (cf. \cite{PSS}). 
Note that the set $\mathcal{K}_{0,GH} (n)$ of biholomorphism classes of $n$-dimensional Fano manifolds admitting K\"ahler--Einstein metrics forms a clopen subset of $\mathcal{KR}_{GH} (n)$. 
Our main theorem is the following. 
\begin{mainthm}[Theorem \ref{main theorem} + Proposition \ref{GH}]
The Hausdorff topological space $\mathcal{KR}_{GH} (n)$ admits a natural complex analytic structure which is uniquely characterized by the following universal property of a natural morphism $\K (n) \to \mathcal{KR}_{GH} (n)$ from the moduli stack: any morphism $\K (n) \to X$ to any complex space $X$ holomorphically and uniquely factors through $\mathcal{KR}_{GH} (n)$. 
\end{mainthm}

In contrast to the current known construction (\cite{Oda1, Oda2, LWX1}) of the moduli space of Fano manifolds with K\"ahler--Einstein metrics, our method for construction actually does not depend on the result in \cite{DatSze}, where they proved modified K-polystable Fano manifolds admit K\"ahler--Ricci solitons. 
Although, as some of the readers might prefer algebro-geometric formulation, we formulate things in terms of modified K-stability, which can be translated into the existence of K\"ahler--Ricci solitons via \cite{DatSze}. 

Our main tool for the construction of complex analytic charts on $\mathcal{KR}_{GH} (n)$ is the following moment map. 
\begin{mainprop}[Proposition \ref{moment} + Proposition \ref{soliton}]
Let $(M, \omega)$ be a $2n$-dimensional $C^\infty$-symplectic manifold underlying a Fano manifold with a Hamiltonian action of a closed real torus $T$. 
For any $\xi \in \mathfrak{t}$, there is a moment map 
\[ \mathcal{S}_\xi : \mathcal{J}_T (M, \omega) \to \mathrm{Lie} (\mathrm{Ham}_T (M, \omega))^\vee \]
on the space $\mathcal{J}_T (M, \omega)$ of $T$-invariant almost complex structures with respect to the modified symplectic structure $\Omega_\xi$ (see subsection 3.1) and the action of $\mathrm{Ham}_T (M, \omega)$. 
Moreover, integrable complex structures in $\mathcal{S}^{-1}_\xi (0)$ correspond to K\"ahler--Ricci solitons. 
\end{mainprop}

We firstly construct charts on the quotient space $(\mathcal{S}^\intg_\xi)^{-1} (0)/\mathrm{Ham}_T (M, \omega)$, where $\mathcal{S}^\intg_\xi$ denotes the restriction of the moment map $\mathcal{S}_\xi$ to the subspace $\mathcal{J}^\intg_T (M, \omega) \subset \mathcal{J}_T (M, \omega)$ consisting of integrable almost complex structures. 
The quotient space reveals to be identified with a clopen subspace of $\mathcal{KR}_{GH} (n)$. 

To compare our constructions with \cite{Oda1, Oda2, LWX1}, we briefly review their methods here. 
They firstly prove the Zariski openness of the set of the K-(semi)stable points in any family of Fano manifolds. 
It follows that the usual moduli stack is Artin algebraic, so that they can apply the established theory of \textit{good moduli spaces} of Artin algebraic stacks. 
Secondly they construct \'etale local charts on this stack of the form $[V/G]$, where each $V$ is an affine scheme and $G$ is a reductive algebraic group. 
Each quotient stack $[V/G]$ has the good moduli space $V \sslash G$. 
We can glue them together, just applying the gluing theory of good moduli spaces developed in \cite{Alp2}. 
Technically, the proofs of the Zariski openness and the existence of the \'etale local charts rely on the argument showing that the set of K-(semi/poly)stable points forms a constructible set of the parameter space in the Zariski topology. 
The CM line bundle, whose GIT weight equals to the Donaldson--Futaki invariant (\cite{PT}), is used to prove the constructibility. 
(Compare \cite{Don3} for another proof of the Zariski openness. )

However, in the case of K\"ahler--Ricci soliton, as there is no candidate for the CM-line bundle because of the irrationality of the modified Donaldson--Futaki invariant, we face a problem with the constructibility. 
So we will work with the real topology, in other words, with Artin \textit{analytic} stacks. 
We can still construct local charts on this Artin analytic stack with good moduli spaces, however, the second nuisance appears when gluing the good moduli spaces together: there is no well-established theory of good moduli spaces for Artin analytic stacks so far. 
(At least to the author, it seems not so easy to show the uniqueness (universal) property of good moduli spaces of Artin analytic stacks, if it exists, which is obviously a key property for the good gluing theory (cf. \cite{Alp1, Alp2}). 
The lack of nice counterpart of `quasi-coherent sheaves' on complex analytic spaces seems critical. (cf. \cite[Section 4]{EP})

Alternatively, we glue our charts by a `cooperation of virtual and real'. 
We construct analytic charts not only on the stack $\K (n)$, but also on the topological spaces $(\mathcal{S}^\intg_\xi)^{-1} (0)/\mathrm{Ham}_T (M, \omega)$, which are related in a canonical way. 
The latter `real side' is studied in section 3 and is used to show that the charts are actually homeomorphisms onto open subsets of $(\mathcal{S}^\intg_\xi)^{-1} (0)/\mathrm{Ham}_T (M, \omega)$. 
This is not treated in \cite{Oda1, Oda2, LWX1} as they could apply Alper's gluing work of good moduli spaces, which works `without reality'. 
The former `virtual side' is studied in section 4 and is used to show that the coordinate changes are holomorphic. 
Finding holomorphic relations between the analytic charts are easier on the stack $\K (n)$ than on the topological spaces $(\mathcal{S}^\intg_\xi)^{-1} (0)/\mathrm{Ham}_T (M, \omega)$. 
These holomorphic relations of stacks descend to the actual holomorphic relations between the analytic charts on $(\mathcal{S}^\intg_\xi)^{-1} (0)/\mathrm{Ham}_T (M, \omega)$ thanks to the universality of the local moduli spaces and the fundamental (2-categorical version of) Yoneda's lemma: the natural fully faithful embedding of the category $\Cpx$ of complex analytic spaces to the 2-category of complex analytic stacks. 
% On the contrary, from this viewpoint, it does not sound realistic to prove the holomorphy `within the real side' as the forgetful functor from $\Cpx$ to the category of topological spaces is not full nor faithful. 

\begin{organ}
The remainder of this paper is organized as follows. 
In section 2, we review some known results on K\"ahler--Ricci soliton and rearrange K-stability notion modified to the soliton setting so that it fits into our moduli problem. 
It is explained that the pair $(X, \xi')$ can be converted into the action $X \curvearrowleft T$, where $T$ is the torus generated by the holomorphic vector field $\xi'$. 
We introduce \textit{gentle Fano $T$-manifolds} as Fano $T$-manifolds inseparable from \textit{smooth} Fano $T$-manifolds with K\"ahler--Ricci solitons, which are expected to be K-semistable. 
They form an adequate moduli stack in our moduli problem. 
Finally, we propose Proposition \ref{gentle degeneration}, which states the uniqueness of the central fiber of gentle degenerations. 
It will be proved after we complete Proposition \ref{etale}, and play an essential role in the proof of Theorem \ref{main theorem} in subsection 4.2. 

In section 3, we construct and study an infinite dimensional moment map $\mathcal{S}_\xi$ whose integrable zero points correspond to K\"ahler--Ricci solitons. 
We describe that local slices $\nu :B \to \mathfrak{k}$ of the moment map actually give charts $\nu^{-1} (0)/K \approx BK^c \sslash K^c$ on the topological space consisting of K\"ahler--Ricci solitons. 
To achieve this, we need to study Banach completions of Fr\'echet manifolds, where we must pay attention to the treatment of the completions of $\Ham_T (M, \omega)$ as they are never Banach Lie groups. 
We also prove that, in any family of Fano $T$-manifolds, the set of gentle Fano $T$-manifolds forms an open subset in the parameter space of the family. 

In section 4, the main theorem is proved. 
We introduce the stack $\K_{T, \chi}$ of gentle Fano $T$-manifolds and show that it is an Artin analytic stack. 
We prove Proposition \ref{gentle degeneration} in subsection 4.4, using the results in the former half of subsection 4.2. 
We use this proposition in the proof of the main theorem. 
In subsection 4.3, we show that our moduli space is related to the topological space $\mathcal{KR}_{GH} (n)$ endowed with the `complexified' Gromov--Hausdorff topology, which is studied in \cite{PSS}. 

In section 5, we review some examples of Fano manifolds with K\"ahler--Ricci solitons and propose some future studies. 
In particular, we find an iso-trivial degeneration of a K\"ahler--Einstein Fano manifold to another Fano manifold with non-Einstein K\"ahler--Ricci soliton, which implies that the usual moduli stack is not sufficiently separated and hence our new formulation of moduli stacks $\K (n)$ and $\K_{T, \chi}$ is essential. 
\end{organ}

\begin{acknow}
I am grateful to my supervisor Prof. Akito Futaki for his deep encouragement, helpful advice and constant support. 
I would like to thank Yuji Odaka for his warmful encouragement and comments in several workshops. 
It's my pleasure to thank Masaki Taniguchi and Hokuto Konno for stimulating daily discussions on gauge theory and Floer theory, which inspire some arguments in this paper. 
I would like to express my gratitude to Thibaut Delcroix, Fabio Podest\`a for helpful conversations, Ruadha\'i Dervan and Philipp Naumann for discussion on their related work. 
Finally, I wish to thank the anonymous reviewers for their careful and patient reading of a coarse version of this article, whose comments are indispensable to improve the quality of this article. 
This work was supported by the program for Leading Graduate Schools, MEXT, Japan. 
\end{acknow}

\section{K\"ahler--Ricci soliton and K-stability}

\subsection{K\"ahler--Ricci soliton}

A K\"ahler metric $g$ on a Fano manifold $X$ is called a \textit{K\"ahler--Ricci soliton} if it satisfies the following equation: 
\[ \Ric (g) - L_{\xi'} g = g \]
for some holomorphic vector field $\xi'$. 
The same term sometimes refers the pair $(g, \xi')$. 

A fundamental feature of a K\"ahler--Ricci soliton $(g, \xi')$ is that it gives an eternal solution of the normalized K\"ahler--Ricci flow: 
\[ \partial_t g (t) = - \Ric (g (t)) + g (t). \]
Namely, for the 1-parameter smooth family $\phi_t :X \xrightarrow{\sim} X$ generated by $\mathrm{Re} (\xi')$, the following holds: 
\[ \partial_t (\phi_t^* g) = -\Ric (\phi_t^* g) + \phi_t^* g. \]
On a Fano manifold admitting K\"ahler--Ricci soliton, it is shown in \cite{TZ3, TZZZ, DerSze} that the normalized K\"ahler--Ricci flow converges to a K\"ahler--Ricci soliton, starting from any K\"ahler metric in $2 \pi c_1 (M)$. 

It is shown in \cite{Zhu} that there is a solution $g$ of the equation
\[ \Ric (g) -L_{\xi'} g = g_0 \]
for any initial K\"ahler metric $g_0$. 
Let us consider the following smooth continuity path for K\"ahler--Ricci soliton: 
\begin{align}\label{continuity path} \Ric (g_t) -L_{\xi'} g_t = t g_t + (1-t) g_0. \end{align}
One can prove that
\[ R_{\xi'} (X) := \sup \{ t \in [0,1] ~|~ \text{a solution } g_t \text{ of } (\ref{continuity path}) \text{ exists. } \} \]
is independent of the choice of the initial metrics $g_0$ and has the equality
\begin{equation}
\label{R}
R_{\xi'} (X) = \sup \{ t \in [0,1] ~|~ \exists g \text{ s.t. } \Ric (g) - L_{\xi'} g > t g \}. 
\end{equation}
The proof of this equality is in \cite{Sze3} for $\xi' = 0$ and in Kazuma Hashimoto's master thesis \cite{Has} for the general case ($\xi' \neq 0$). 
A related invariant is also mentioned in \cite{DGSW}. 

\begin{rem}
Kazuma Hashimoto was a master student of University of Tokyo supervised by Prof. Akito Futaki. 
He did not proceed to doctoral course and quit his research position. 
The proof of the equality (2) in his thesis is an analogy of \cite{Sze3}, using the functionals $\mathcal{M}_\xi := \mu_\omega$ originally defined in \cite{TZ2} and the following $\mathcal{J}_{\alpha, \xi}$ instead of $\mathcal{M}, \mathcal{J}_\alpha$ in \cite{Sze3}: 
\begin{align*}
\mathcal{M}_\xi (\phi) 
&:= - \int_0^1 dt \int_X \dot{\phi}_t \Big{(} s (g_{\phi_t}) - n - \mathrm{tr} (\nabla_{g_{\phi_t}} \xi') + \xi' (h_{g_{\phi_t}} - \theta_\xi' (\phi_t)) \Big{)} e^{\theta_\xi' (\phi_t)} \omega_{phi_t}^n, 
\\
\mathcal{J}_{\alpha, \xi} (\phi)
&:= \int_0^1 dt \int_X \dot{\phi}_t (\mathrm{tr}_{\omega_{\phi_t}} \alpha - n + \xi' \varphi_\alpha) e^{\theta'_\xi (\phi_t)} \omega_{\phi_t}^n, 
\end{align*}
where $\varphi_\alpha$ is a function with $\alpha - \omega = \sqrt{-1} \partial \bar{\partial} \varphi_\alpha$. 
\end{rem} 

The uniqueness and the existence results analogous to those of the K\"ahler--Einstein metrics \cite{BM, CDS, Tian2} (and ) hold also for K\"ahler--Ricci solitons. 

\begin{thm}[Uniqueness, \cite{TZ1, TZ2} (and \cite{BW} for $\mathbb{Q}$-Fano variety with $t=1$)]
If $(g_1, \xi'_1)$ and $(g_2, \xi'_2)$ are two K\"ahler--Ricci solitons on a Fano manifold $X$, then there is an element $\phi \in \mathrm{Aut}^0 (X)$ such that 
\[ g_2 = \phi^* g_1,\quad \xi'_2 = \phi^{-1}_* \xi'_1, \]
where $\mathrm{Aut}^0 (X)$ is the identity component of the group $\mathrm{Aut} (X)$ of biholomorphisms of $X$. 
Moreover, a solution $g_t$ of the equation (\ref{continuity path}) is absolutely unique for any initial metric $g_0$ and $t \in [0, 1)$. 
\end{thm}

\begin{thm}[Existence, \cite{DatSze, CSW}]
$R_{\xi'} (X) =1$ for any K-semistable pair $(X, \xi')$. 
If in addition $(X, \xi')$ is K-polystable, there is a K\"ahler--Ricci soliton on $X$ with respect to $\xi'$. 
\end{thm}

We will see the definition of the K-stability of pairs $(X, \xi')$ in the next subsection. 
The above claim on K-semistability is also covered in \cite{C. Li} for the K\"ahler--Einstein case, using \cite{CDS, Tian2}. 
The opposite implication for K-polystablity is proved in \cite{Ber, BW} including the $\mathbb{Q}$-Fano case as follows. 

\begin{thm}[\cite{Ber, BW}]
Let $X$ be a $\mathbb{Q}$-Fano variety. 
If $X$ admits a K\"ahler--Ricci soliton $(g, \xi')$, then $(X, \xi')$ is K-polystable. 
\end{thm}

In the K\"ahler--Einstein case (i.e. $\xi' = 0$), \cite{Der, C. Li} shows that $X$ is K-semistable if $R (X) = 1$. 
So we can summarize as follows. 
\begin{itemize}
\item $X$ is K-polystable $\iff$ $X$ admits a K\"ahler--Einstein metric. 

\item $X$ is K-semistable $\iff$ $R (X) = 1$. 
\end{itemize}
Only the implication from the right to the left-hand side of the second item is still open for general $(X, \xi')$. 

There is a version of Futaki invariant suitable for K\"ahler--Ricci soliton defined in \cite{TZ2}. 
Let $H^0 (X, \Theta_X)$ denote the space of holomorphic vector fields on $X$. 
Define a linear map $\mathrm{Fut}_{\xi'} : H^0 (X, \Theta_X) \to \mathbb{C}$ by
\[ \mathrm{Fut}_{\xi'} (v') := \int_X v' (h - \theta_{\xi'}) e^{\theta_{\xi'}} \omega^n, \]
where $\omega \in 2 \pi c_1 (M)$ is a K\"ahler form, $h$ is a real valued function satisfying $\sqrt{-1} \partial \bar{\partial} h = \Ric (\omega) - \omega$ and $\theta_{\xi'}$ is a complex-valued function characterized by
\[ \begin{cases} L_{\xi'} \omega = \sqrt{-1} \partial \bar{\partial} \theta_{\xi'} & \\ \int_X e^{\theta_{\xi'}} \omega^n = \int_X \omega^n. & \end{cases} \]
The function $\theta_{\xi'}$ becomes real-valued when $\xi := \mathrm{Im} \xi'$ is a Killing vector. 
This linear function is independent of the choice of $\omega$, so it gives an invariant depending only on $X$ and $\xi'$, which is now called \textit{the modified Futaki invariant}. 
This invariant obviously vanishes when $X$ admits a K\"ahler--Ricci soliton with respect to the vector field $\xi'$. 

The following is a crucial fact in order to properly formulate our moduli problem. 

\begin{prop}[\cite{TZ2}]
\label{TZ K-optimal}
Let $X$ be a Fano manifold, which does not necessarily have a K\"ahler--Ricci soliton, and $K \subset \mathrm{Aut} (X)$ be a compact subgroup. 
Then there is a unique holomorphic vector field $\xi'$ with $\mathrm{Im} (\xi') \in \mathrm{Lie} (K)$ such that
\[ \mathrm{Fut}_{\xi'} (v') = 0, ~ \forall v' \in \mathrm{Lie} (K^c), \]
where $K^c \subset \mathrm{Aut} (X)$ is the complexification of the group $K$. 
\end{prop}

\begin{rem}
In general, a reductive algebraic group $K^c$ does not uniquely determine its maximal compact subgroup $K$, but only up to conjugate. 
When $K^c$ is an algebraic torus, which is isomorphic to $(\mathbb{C}^*)^k$, its maximal compact subgroup $(U (1))^k$ is uniquely determined. 
This fact allows us to get away from a formulation relying on structures over the field $\mathbb{R}$ as we see in the next subsection and to formulate things over even a field of positive characteristic, which should be preferred by algebraic geometers. 
\end{rem}

The total biholomorphism group of a Fano manifold $X$ admitting a K\"ahler--Ricci soliton $(g, \xi')$ is not necessarily reductive. 
Instead, we have the following. 

\begin{thm}[\cite{TZ2} (\cite{BW} for the $\mathbb{Q}$-Fano case)]
\label{reductive}
Suppose a $\mathbb{Q}$-Fano variety $X$ has a K\"ahler--Ricci soliton $(g, \xi')$, then the subgroup $\mathrm{Aut}^0 (X, \xi') \subset \mathrm{Aut}^0 (X)$ consisting of $\xi'$-preserving biholomorphisms is a maximal reductive subgroup of $\mathrm{Aut}^0 (X)$. 
Moreover, the complexification of the identity component $\mathrm{Isom}^0 (X, \xi')$ of the group of isometries preserving $\xi'$ coincides with the group $\mathrm{Aut}^0 (X, \xi')$. 
\end{thm}

\begin{rem}
\label{moduli set}
The reductivity of the automorphism groups of geometric structures of which we intend to construct a geometric moduli space, is crucial if one expect to apply local or global GIT to its construction and indeed indispensable in the doctrine of Alper's good moduli space (cf. \cite{Alp1, Alp2}). 

The uniqueness of K\"ahler--Ricci soliton implies that the set consisting of the isomorphism classes of the pairs $(X, \xi')$ with K\"ahler--Ricci solitons can be naturally identified with the set consisting of the biholomorphism classes of Fano manifolds $X$ with K\"ahler--Ricci solitons. 
So there is no change in the \textit{support sets} of `the moduli spaces' of the following two moduli stacks: one is the usual moduli stack associated with Fano manifolds $X$ admitting K\"ahler--Ricci solitons, and the other is the moduli stack associated with Fano pairs $(X, \xi')$ admitting K\"ahler--Ricci solitons. 

However, there are nice geometric features in the latter stack compared to the former stack, such as the separation property and the reductivity of the stabilizer groups at K-polystable points, which is appropriate for the local GIT construction of the good moduli space. 

So we will work with the latter stack, and precisely define it in subsection 4.1, replacing the pairs $(X, \xi')$ with the $T_{ \xi'}$-action on $X$. 
This \textit{may change the topology} of the moduli space, but it turns out that the latter stack is correct with regard to the `complexified' Gromov--Hausdorff convergence considered in \cite{PSS}. 
\end{rem}

\subsection{K-stability}

Here we review the definition of K-stability and formulate it as the stability notion of a Fano manifold \textit{with an algebraic torus action}. 
This enables us to introduce an adequate notion of `deformations of Fano manifolds with K\"ahler--Ricci solitons' and leads us to the proper definition of the stack $\K (n)$. 

Recall that a \textit{$\mathbb{Q}$-Fano variety} $X$ is a reduced irreducible normal complex space $X$ with the following property: there is a positive integer $\ell$ such that the sheaf $i_* ((\det \Theta_{X^\mathrm{reg}})^{\otimes \ell})$, which is denoted by $\mathcal{O} (-\ell K_X)$, is isomorphic to the sheaf of sections of an ample line bundle on $X$, and $X$ has only log terminal singularities (see \cite{EGZ}). 
The minimum $\ell$ satisfying this property is called the \textit{$\mathbb{Q}$-Gorenstein index} of $X$. 
Obviously, $\mathbb{Q}$-Fano varieties can be embedded into some $\mathbb{C}P^N$, hence they are also considered as schemes, but we treat them in the category of complex spaces. 

A \textit{$\mathbb{Q}$-Fano $T$-variety} is a $\mathbb{Q}$-Fano variety $X$ with a holomorphic action $\alpha :X \times T \to X$, where we only consider an algebraic torus $T \cong (\mathbb{C}^*)^k$. 
When $X$ has no singularities, we call it \textit{Fano $T$-manifold}. 
We denote by $\mathrm{Aut}_T (X)$ the centralizer of $T \subset \mathrm{Aut} (X)$: 
\[ \mathrm{Aut}_T (X) := \{ g \in \mathrm{Aut} (X) ~|~ g t = t g \text{ for } \forall t \in T \}. \]

Let $T$ be an algebraic torus. 
We denote the character lattice of $T$ by $M := \mathrm{Hom} (T, \mathbb{C}^*)$ and its dual (1-psg) lattice by $N := \mathrm{Hom} (\mathbb{C}^*, T)$. 
Let $X$ be a $\mathbb{Q}$-Fano $T$-variety. 
Its $T$-action canonically lifts to the sheaf $\mathcal{O} (-m \ell K_X)$ and hence there is an action of $T$ on the cohomologies of $\mathcal{O} (-m \ell K_X)$. 
For a character $u \in M$, put 
\[ H^0_u (X, \mathcal{O} (-m \ell K_X)) := \{ \sigma \in H^0 (X, \mathcal{O} (-m \ell K_X)) ~|~ t. \sigma = u (t) \sigma ~\forall t \in T \} \]
and set 
\begin{align*} 
h^i_X (m) 
&:= \dim H^i (X, \mathcal{O} (-m \ell K_X)), 
\\ 
h^i_{X, u} (m) 
&:= \dim H^i_u (X, \mathcal{O} (-m \ell K_X)). 
\end{align*}

We adopt the following pure algebraic definition of modified Futaki invariant exhibited in \cite{BW}, which is shown to coincide with $\mathrm{Fut}_{\xi'}$ in the previous section, up to a uniform positive factor. 
Note that the modified Futaki invariant for special degenerations is firstly introduced in \cite{Xio} and reformulated in \cite{WZZ}. 
(The author thanks the referees for telling the author these important references. )

\begin{defin}[modified algebraic Futaki invariant]
For a $\mathbb{Q}$-Fano $T$-variety $X$ and an element $\xi \in N_\mathbb{R}$, we define the \textit{modified} (sometimes we omit this word in our $T$-equivariant setup) \textit{algebraic Futaki invariant} $\mathit{F}_{X, \xi} :N \to \mathbb{R}$ by
\[ \mathit{F}_{X, \xi} (\lambda) := - \lim_{m \to \infty} \frac{w_{X, \xi} (m; \lambda)}{m h^0_X (m)}, \]
where
\[ w_{X, \xi} (m; \lambda) := \sum_{u \in M} e^{\langle u, \xi \rangle/m} h^0_{X, u} (m) \langle u, \lambda \rangle. \]
\end{defin}

We define the \textit{Hilbert character} $\chi :\mathbb{Z} \to \mathbb{Z} [M]$ of a Fano $T$-manifold $X$ by
\begin{equation}
\label{Hilb char}
\chi_m := \sum_{i=0}^{\dim X} (-1)^i \sum_{u \in M} h^i_{X, u} (m) u \in \mathbb{Z} [M]. 
\end{equation}
We call a function $\chi: \mathbb{Z} \to \mathbb{Z} [M]$ a \textit{Fano character} if there exists a Fano $T$-manifold whose Hilbert character given in (\ref{Hilb char}) is the given $\chi$. 

\begin{prop}[\cite{TZ2}]
For every Fano $T$-manifold $X$, there exists a unique vector $\xi \in N_\mathbb{R}$ such that the modified algebraic Futaki invariant $\mathit{F}_{X, \xi} (\lambda)$ in the above sense vanishes on the 1-psg lattice $N$ of $T$ (equivalently, the modified Futaki invariant $\mathrm{Fut}_{\xi'}|_{\mathfrak{t}}$ restricted to the Lie algebra $\mathfrak{t}$ of $T$ vanishes). 
We call this vector $\xi$ the \textit{K-optimal vector} of $(X, T)$. 
\end{prop}

Obviously from the definition of the modified algebraic Futaki invariant, the K-optimal vector $\xi$ of a Fano $T$-manifold $X$ depends only on the Hilbert character $(T, \chi)$. 
So it also makes sense to say that $\xi \in N_\mathbb{R}$ is the K-optimal vector of a Fano character $(T, \chi)$, which is a $T$-equivariant deformation invariant. 

\begin{prop}[\cite{TZ2}]
If a Fano manifold $X$ has a K\"ahler--Ricci soliton $(g, \xi')$, then the $\xi'$ is the K-optimal vector with respect to any algebraic torus containing the algebraic torus generated by $\xi'$. 
(Note that the closure of the exponential of the imaginary part of the holomorphic vector $\xi'$ associated to soliton gives a closed real torus, and the algebraic torus generated by $\xi'$ is just the complexification of this closed real torus. )
\end{prop}

We define the K-optimality of Fano character (and in particular the K-optimality of a Fano $T$-manifold), not of vector, as follows. 

\begin{defin}[K-optimal character]
We call a Fano character $(T, \chi)$ \textit{K-optimal} if there is no proper sub-lattice $\tilde{N} \subset N$ with $\xi \in \tilde{N}_\mathbb{R}$ for the K-optimal vector $\xi \in N_\mathbb{R}$ of $(T, \chi)$. 

For a Fano manifold $X$, we call an algebraic action of an algebraic torus $T$ on $X$ is \textit{K-optimal} if it is maximal (as actions on $X$) among all K-optimal characters $(\tilde{T}, \tilde{\chi})$ obtained from $\tilde{T}$-actions on $X$. 
\end{defin}

\begin{rem}
It is possible that both characters $(T_1, \chi_1) \subset (T_2, \chi_2)$ are K-optimal, where $\chi_i$ are related by the projection $\mathbb{Z} [M_2] \to \mathbb{Z} [M_1]$. 
Not only that, there is an example of a Fano manifold $X_1$ with a K\"ahler--Einstein metric $(g_1, \xi_1' = 0)$ admitting a deformation to a Fano manifold $X_2$ with a K\"ahler--Ricci soliton $(g_2, \xi_2')$ (Example \ref{horo-example}), which shows that both actions $X_1 \curvearrowleft T_1 = 0$ and $X_2 \curvearrowleft T_2 = T (\xi_2') \neq 0$ are K-optimal with $(T_1, \chi_{X_1}) \subset (T_2, \chi_{X_2})$. 
This example illustrates that the torus equivariant formulation is essential for the separation of the moduli space of Fano manifolds with K\"ahler--Ricci solitons. 

For a fixed Fano manifold $X$, K-optimal action $T \subset \mathrm{Aut} (X)$ is unique up to adjoint thanks to the uniqueness in Proposition \ref{TZ K-optimal}. 
\end{rem}

Let $X$ be a $\mathbb{Q}$-Fano $T$-variety. 
A pair $(\pi :\mathfrak{X} \to \mathbb{C}, \theta)$ consisting of the following data is called a \textit{special degeneration} of $X$. 
\begin{enumerate}
\item $\mathfrak{X}$ is a normal complex space with an action of $T \times \mathbb{C}^*$ and $\pi :\mathfrak{X} \to \mathbb{C}$ is a $T \times \mathbb{C}^*$-equivariant proper flat $\mathbb{Q}$-Gorenstein surjective morphism whose central fiber $\mathfrak{X}_0$ is a $\mathbb{Q}$-Fano variety, where $T \times \mathbb{C}^*$ acts on $\mathbb{C}$ by $z. (t, s) = sz$. 

\item $\theta$ is a $T \times \mathbb{C}^*$-equivariant isomorphism $\theta : X \times \mathbb{C}^* \xrightarrow{\sim} \pi^{-1} (\mathbb{C}^*)$. 
\end{enumerate}
We also assume that there is a holomorphic line bundle $\mathcal{L}$ on $\mathfrak{X}$ with an isomorphism $\theta^* \mathcal{L}|_{\pi^{-1} (\mathbb{C}^*)} \cong p_1^* \mathcal{O} (-\ell K_X)$ for some $\ell$. 
It is shown in \cite[Lemma 2.2]{Ber} that if such $\mathcal{L}$ exists, then $-\ell K_{\mathfrak{X}}$ becomes $\mathbb{Q}$-Cartier and the tensor bundle $\mathcal{L}^{\otimes m}$ is actually isomorphic to $\mathcal{O} (-m\ell K_{\mathfrak{X}/\mathbb{C}})$ for some $m$. 
So we exclude the datum $\mathcal{L}$ from the data of special degeneration. 

\begin{defin}[K-stability]
Let $\xi \in N_\mathbb{R}$ be the K-optimal vector of a $\mathbb{Q}$-Fano $T$-variety $X$. 
Denote the vector $(\xi, 0) \in (N \times \mathbb{Z})_\mathbb{R}$ by the same symbol $\xi$. 
We call the $\mathbb{Q}$-Fano $T$-variety $X$ 
\begin{itemize}
\item \textit{K-semistable} if for any special degeneration $(\pi :\mathfrak{X} \to \mathbb{C}, \theta)$ of $X$, the modified algebraic Futaki invariant $\mathit{F}_{X, \xi} (\pi, \theta) := \mathit{F}_{\mathfrak{X}_0, \xi} (\lambda)$ of the central fiber $\mathfrak{X}_0$ is nonnegative, where $\lambda$ is the one parameter subgroup defined by $\lambda :\mathbb{C}^* \to T \times \mathbb{C}^* :s \mapsto (1, s)$. 

\item \textit{K-polystable} if $X$ is K-semistable and $\mathit{F}_{X, \xi} (\pi, \theta) = 0$ if and only if there exists a one parameter subgroup $\lambda: \mathbb{C}^* \to \mathrm{Aut}_T (X)$ such that $\theta (x \lambda(t)^{-1}, t)$ extends to an isomorphism of the total space $X \times \mathbb{C} \xrightarrow{\sim} \mathfrak{X}$. 

\item \textit{K-stable} if the Fano $T$-variety $X$ is K-polystable and $\mathrm{Aut}_T^0 (X) = T$. 
\end{itemize}
\end{defin}

\begin{rem}
A pair $(X, \xi')$ of a $\mathbb{Q}$-Fano variety $X$ and a holomorphic vector field $\xi'$ is called K-(semi/poly)stable if $\xi := \mathrm{Im} \xi'$ generates a closed real torus $T_\mathbb{R}$ and $(X, T)$ is K-(semi/poly)stable where $T$ denotes the complexification of the closed real torus $T_\mathbb{R}$. 
In this case, the vector $\xi \in \mathrm{Lie} (T_\mathbb{R}) = N_\mathbb{R}$ is of course K-optimal. 

We call $X$ modified K-(semi/poly)stable if there exists a torus action $X \curvearrowleft T$ which makes $X$ K-(semi/poly)stable with respect to the action. 
\end{rem}

\begin{rem}
Note that a K-(semi/poly)stable Fano $T$-manifold is \textit{not} necessarily a K-(semi/poly)stable Fano manifold (with respect to the trivial torus action), but only a \textit{modified} K-(semi/poly)stable Fano manifold. 
However, suppose $X$ is a Fano $T$-manifold, $\tilde{T} \subset T$ is a sub-torus and the K-optimal vector $\tilde{\xi}$ with respect to the $\tilde{T}$-action coincides with the K-optimal vector $\xi$ with respect to the $T$-action (i.e. $\xi \in \mathrm{Lie} (\tilde{T})$), then the Fano $T$-manifold $X$ is K-(semi/poly)stable if and only if the Fano $\tilde{T}$-manifold $X$ is. 
This is proved in \cite{DatSze} and recently proved by purely algebraic method in \cite{LiXu, LWX3} for the KE case ($\xi=0, T = 0$). 
\end{rem}

We introduce a gentle Fano $T$-manifold as a Fano $T$-manifold inseparable from a smooth Fano $T$-manifold admitting K\"ahler--Ricci soliton. 

\begin{defin}[gentle Fano]
\label{gentle Fano}
A Fano $T$-manifold $X$ is called \textit{gentle} if there is a $T$-equivariant deformation $\mathcal{X} \to \Delta$ with an isomorphism $\mathcal{X}|_{\Delta^*} \cong X \times \Delta^*$ such that its central fiber $\mathcal{X}_0$ is a smooth K-polystable Fano $T$-manifold. 
We call $\mathcal{X} \to \Delta$ a \textit{gentle degeneration}. 
\end{defin}

From GIT viewpoint, it is naturally expected that any gentle Fano $T$-manifold is K-semistable. 
In this paper, we do not pursue this expectation as it is not essential for the construction of our moduli space, while their K-semistability might be philosophically important. 
Note that we always have $R_{\xi'} (X) = 1$ for a gentle Fano $T$-manifold $X$ with the K-optimal vector $\xi$, thanks to the equality (\ref{R}). 
This fact helps us to prove the following proposition. 

\begin{prop}
\label{gentle degeneration}
Let $X$ be a gentle Fano $T$-manifold whose torus action is K-optimal. 
Then any two gentle degenerations of $X$ have the $T$-equivariant biholomorphic central fibers. 
\end{prop}

The proposition will be proved at the end of section 4, using Proposition \ref{etale} and a version of Donaldson-Sun's technology on Gromov--Hausdorff limit, and will be applied to the proof of Theorem \ref{main theorem}. 
The logical order of our argument is ``Proposition \ref{etale} $\Rightarrow$ Proposition \ref{gentle degeneration} $\Rightarrow$ Theorem \ref{main theorem}''. 
It seems also possible to show this proposition without using a finiteness from Proposition \ref{etale} as in \cite{LWX1}. 
However, the author thinks the finiteness simplifies our argument. 

\section{Donaldson-Fujiki set-up}

We call a closed $C^\infty$-symplectic manifold $(M, \omega)$ \textit{symplectic Fano} if its cohomology class $[\omega]$ is equal to $2 \pi$ times the first Chern class $c_1 (M, \omega)$ and there exists an $\omega$-compatible almost complex structure $J$ with positive Ricci curvature. 
Note that we have $b^1 (M) = 0$ from familiar Bochner's theorem or Myers' theorem as we have a metric with $\Ric > 0$. 
Throughout this section, $T$ stands for a closed real torus and $(M, \omega)$ for a symplectic Fano manifold with a Hamiltonian effective action by $T$. 

We denote by $\mathrm{Symp} (M, \omega)$ the group of symplectic diffeomorphisms and $\Ham^0 (M, \omega)$ its subgroup generated by Hamiltonian diffeomorphisms. 
Thanks to Banyaga's theorem, in the case $b^1 (M) = 0$, $\Ham^0 (M, \omega)$ actually coincides with $\mathrm{Symp}^0 (M, \omega)$, the identity connected component of $\mathrm{Symp} (M, \omega)$. 
(Even though it is easy to see that both groups have a natural Fr\'echet Lie group structures and their Lie algebras coincide, the coincidence at the level of Fr\'echet Lie group is \textit{not} trivial because the Fr\'echet Lie group structures are not locally exponential. See \cite{Neeb} for the generalities on Fr\'echet Lie groups. )
We must work with the group $\mathrm{Symp} (M, \omega)$ (resp. $\mathrm{Symp}_T (M, \omega)$) so that the complexification of the stabilizer group of cscK structure $J \in \mathcal{J} (M, \omega)$ (resp. K\"ahler--Ricci soliton structure $J \in \mathcal{J}_T (M, \omega)$) coincides with the biholomorphism group $\mathrm{Aut} (M, J)$ of $(M, J)$ (resp. $\mathrm{Aut}_T (M, J)$), not only it just includes the identity component $\mathrm{Aut}^0 (M, J)$ (resp. $\mathrm{Aut}^0_T (M, J)$). 
Keeping Banyaga's theorem in our mind, we prefer using the notation $\Ham (M, \omega) := \mathrm{Symp} (M, \omega)$, which is \textit{not necessarily connected}, as we always identify its Lie algebra with $C^\infty (M)/\mathbb{R}$. 

We consider the space $\mathcal{J}_T (M, \omega)$ of $T$-invariant $\omega$-compatible almost complex structures and denote by $\mathcal{J}_T^\intg (M, \omega)$ the subspace of integrable complex structures. 
It is well known that $\mathcal{J}_T (M, \omega)$ admits a natural Fr\'echet smooth manifold structure, which is identified with the space of $T$-equivariant sections of an associated $Sp (2n)/U (n)$-fibre bundle (see \cite{Pal} for instance). 
The tangent space at $J \in \mathcal{J}_T (M, \omega)$ can be written as follows. 
\[ T_J \mathcal{J}_T (M, \omega) = \{ A \in \Gamma_T^\infty (\mathrm{End} TM) ~|~ AJ + JA = 0, \omega (A \cdot, \cdot) + \omega (\cdot, A \cdot) = 0 \}. \]

Similarly, the group $\Ham_T (M, \omega)$ of $T$-compatible symplectic diffeomorphisms can be endowed with a Fr\'echet smooth Lie group structure, whose Lie algebra can be identified with $C^\infty_T (M)/\mathbb{R}$. 
The left adjoint action is given by
\[ \mathrm{Ham}_T (M, \omega) \times C^\infty_T (M)/\mathbb{R} \to C^\infty_T (M)/\mathbb{R} :(\phi, f) \mapsto f \circ \phi^{-1}. \]
The following right action
\[ \mathcal{J}_T (M, \omega) \times \mathrm{Ham}_T (M, \omega) \to \mathcal{J}_T (M, \omega) : (J, \phi) \mapsto \phi^* J \]
is also smooth and its derivative is given by
\[ C^\infty_T (M)/\mathbb{R} \to T_J \mathcal{J}_T (M, \omega) :f \mapsto L_{X_f} J, \]
where $X_f$ is the Hamiltonian vector field of $f$: $-df = i (X_f) \omega$.

\subsection{The moment map}

For a given $\xi \in \mathfrak{t} = \mathrm{Lie} (T)$, we let $\theta_\xi$ be a real valued function on $M$ given by
\[ -d \theta_\xi = i_\xi \omega \]
with the prescribed normalization
\[ \int_M \theta_\xi e^{-2 \theta_\xi} \omega^n = 0. \]
This function is invariant under the action of $\Ham_T (M, \omega)$. 

Set $\theta'_\xi := -2 \theta_\xi$.
For each $J \in \mathcal{J}_T (M, \omega)$, 
\[ \xi'_J :=J \xi + \sqrt{-1} \xi \in \mathcal{X}^{1,0} (M, J) \]
satisfies
\[ \sqrt{-1} \bar{\partial} \theta'_\xi = \sqrt{-1} (d (-2 \theta_\xi) + \sqrt{-1} J d(-2\theta_\xi))/2 = i_{\xi'_J} \omega. \]

We consider the following Riemannian metric on $\mathcal{J}_T (M, \omega)$, modified by $\xi$ from the usual one (\cite{Don1}), defined as
\[ (A, B)_\xi := \int_M g^{ij}_J g_{J, kl} A^k_i B^l_j ~e^{-2 \theta_\xi} \omega^n \]
for tangent vectors $A, B \in T_J \mathcal{J}_T (M, \omega)$ and set
\[ \Omega_\xi (A, B) := (JA, B)_\xi. \]
It is easy to see that $\Omega_\xi$ defines a non-degenerate closed $2$-form on $\mathcal{J}_T (M, \omega)$. 
We also consider
\[ (f, g)_\xi := \int_M f g ~e^{-2 \theta_\xi} \omega^n \]
for $f, g \in C^\infty_T (M)$, which defines an inner product on the subspace
\[ C^\infty_{T, \xi} (M, \omega) := \{ f \in C^\infty_T (M) ~|~ \int_M f ~e^{-2 \theta_\xi} \omega^n = 0 \} \cong C^\infty_T (M)/\mathbb{R}. \]

Finally, we denote by $s (J)$ the Hermitian scalar curvature of $J$, defined by Donaldson \cite{Don1}. 
We normalize $s (J)$ by a factor so that it is equal to the K\"ahler scalar curvature $-g_J^{i\bar{j}} \partial_{J, i} \partial_{J, \bar{j}} (\log \det g_J)$ for any integrable $J$, which is the half of the Riemannian scalar curvature. 
We denote by $\Delta_{g_J}$ the usual Riemannian Laplacian with positive eigenvalue, which is the twice of the $\bar{\partial}$-Laplacian $\bar{\Box}_J = - g_J^{i \bar{j}} \partial_{J, i} \partial_{J, \bar{j}}$ when $J$ is integrable. 
Here is the moment map for our modified symplectic structure $\Omega_\xi$. 

\begin{prop}
\label{moment}
Fix $\xi, \zeta \in \mathfrak{t}$. 
For each $J \in \mathcal{J}_T (M, \omega)$, we consider \textit{the modified Hermitian scalar curvature} defined as 
\[ s_{\xi, \zeta} (J) := (s (J) - n) + \Delta_{g_J} \theta'_\xi - \xi'_J \theta'_\xi - \theta'_\xi - \theta'_\zeta, \]
where $\theta'_\zeta$ is normalized as $\int_M \theta'_\zeta e^{\theta'_\xi} \omega^n = 0$. 
Then the map
\[ \mathcal{S}_{\xi, \zeta} :\mathcal{J}_T (M, \omega) \to C^\infty_{T, \xi} (M, \omega)^\vee :J \mapsto (4s_{\xi, \zeta}, \cdot)_\xi \]
satisfies the property of the moment map with respect to the symplectic structure $\Omega_\xi$ and the action of $\Ham_T (M, \omega)$ on $\mathcal{J}_T (M, \omega)$. 
That is, $\mathcal{S}_{\xi, \zeta}$ is a $\mathrm{Ham}_T (M, \omega)$-equivariant smooth map satisfying
\[ - \frac{d}{dt}\Big{|}_{t=0} \langle \mathcal{S}_{\xi, \zeta} (J_t) , f \rangle = \Omega_\xi (L_{X_f} J_0, \dot{J}_0)  \]
for any smooth curve $J_t \in \mathcal{J}_T (M, \omega)$ and $f \in C^\infty (M)$. 
\end{prop}

\begin{proof}
The equivariance of the map readily follows because the coadjoint right action is given by $(s, \cdot)_\xi \cdot \phi = (\phi^* s, \cdot)_\xi$ and $\theta_\xi'$ and $ \theta_\zeta'$ are $\Ham_T (M, \omega)$-invariant. 

The modified Hermitian scalar curvature can be divided in two parts as follows. 
\begin{align}
\frac{d}{dt}\Big{|}_{t=0} (4s_{\xi, \zeta} (J_t), f)_\xi & = \int_M \frac{d}{dt}\Big{|}_{t=0} 4s_{\xi, \zeta} (J_t) f ~e^{-2\theta_\xi} \omega^n \notag
\\ & = \int_M \frac{d}{dt}\Big{|}_{t=0} 4s (J_t) f e^{-2 \theta_\xi} ~\omega^n - \frac{d}{dt}\Big{|}_{t=0} \int_M ((-4 \Delta_t + 4 \xi'_t) \theta'_\xi ) f ~e^{-2 \theta_\xi} \omega^n \notag
\\ \label{moment0} &= \frac{d}{dt}\Big{|}_{t=0} (4s (J_t), f e^{-2 \theta_\xi}) - \frac{d}{dt}\Big{|}_{t=0} \int_M ((8 \Delta_t - 8 J_t \xi) \theta_\xi) f ~e^{-2 \theta_\xi} \omega^n 
\end{align}
Now we use the following Donaldson's famous calculation \cite{Don1} on the Hermitian scalar curvature with respect to the usual symplectic structure: 
\[ \frac{d}{dt}\Big{|}_{t=0} (4s (J_t), f) = (L_{X_f} J, JA) \]
for $A = \dot{J}_0$. 
The factor $4$ comes from our convention of the metric $(\cdot,\cdot)_\xi$ (compare \cite[Proposition 2.2.1.]{Sze1}). 
Combined with the following basic identities: (a) $X_{fg} = f X_g + g X_f$, (b) $L_{f X} J = f L_X J - Jdf \otimes X + df \otimes JX$, (c) $L_\xi J =0$, the first term of (\ref{moment0}) can be arranged as follows. 
\begin{align*}
\frac{d}{dt}\Big{|}_{t=0} (4s (J_t), f e^{-2 \theta_\xi}) &= (L_{X_{f \exp (-2 \theta_\xi)}} J, JA) 
\\ &= ( L_{\exp (-2 \theta_\xi) X_f} J, JA) + (-2) (L_{(f \exp (-2 \theta_\xi) )\xi} J, JA) 
\\ &= (e^{-2 \theta_\xi} L_{X_f} J - J d (e^{-2 \theta_\xi}) \otimes X_f + d (e^{-2 \theta_\xi}) \otimes JX_f, JA) 
\\ &\quad\quad + (-2) ((fe^{-2\theta_\xi}) L_\xi J - J d (f e^{-2 \theta_\xi}) \otimes \xi + d (f e^{-2 \theta_\xi}) \otimes J \xi ,JA) 
\\ &= (L_{X_f} J, JA)_\xi 
\\ &\quad\quad- (-2) (J d\theta_\xi \otimes X_f, JA)_\xi + (-2) (d\theta_\xi \otimes JX_f, JA)_\xi 
\\ &\qquad\quad - (-2) (Jdf \otimes \xi, JA)_\xi + (-2) (df \otimes J \xi, JA)_\xi 
\\ &\qquad\qquad - 4 (f Jd \theta_\xi \otimes \xi, JA)_\xi + 4 (f d \theta_\xi \otimes J\xi, JA)_\xi. 
\end{align*}
Now it suffices to show the following equalities. 
\begin{align}
\label{moment1}
\frac{d}{dt}\Big{|}_{t=0} ((\Delta_t +(-2) J_t \xi) \theta_\xi, f)_\xi &= (J d\theta_\xi \otimes X_f, JA)_\xi 
\\ \label{moment2} &= - (d\theta_\xi \otimes JX_f, JA)_\xi 
\\ \label{moment3} &= (J df \otimes \xi, JA)_\xi 
\\ \label{moment4} &= - (df \otimes J \xi, JA)_\xi
\end{align}
and
\begin{align}
\label{moment5}
\frac{d}{dt}\Big{|}_{t=0} (-(J_t \xi) \theta_\xi, f)_\xi &= (f Jd\theta_\xi \otimes \xi, JA)_\xi 
\\ \label{moment6} &= -(f d\theta_\xi \otimes J\xi, JA)_\xi. 
\end{align}
As for (\ref{moment1}), 
\begin{align*}
(Jd\theta_\xi \otimes X_f, JA)_\xi &= \int_M g^{ij} g_{kl} (Jd \theta_\xi \otimes X_f)^k_i (JA)^l_j ~e^{-2 \theta_\xi} \omega^n 
\\ &= \int_M (Jd\theta_\xi \otimes X_f)^k_i (JA)^i_k ~e^{-2 \theta_\xi} \omega^n 
\\ &= \int_M - \theta_{\xi, p} (f_j \omega^{jk}) A^p_k ~e^{-2 \theta_\xi} \omega^n
\\ &= \int_M - \theta_{\xi, p} f_j \omega^{pk} A^j_k ~e^{-2 \theta_\xi} \omega^n
% \\ &=  \frac{d}{dt}\Big{|}_{t=0} \int_M - \omega^* (d \theta_\xi, J_t df) ~e^{-2 \theta_\xi} \omega^n
\\ &= \frac{d}{dt}\Big{|}_{t=0} \int_M g^*_t (d\theta_\xi, df) ~e^{-2 \theta_\xi} \omega^n
\\ &= \frac{d}{dt}\Big{|}_{t=0} ((\Delta_t + (-2) J_t \xi) \theta_\xi, f)_\xi. 
\end{align*}
We obtain (\ref{moment5}) as follows. 
\begin{align*} 
(fJ d\theta_\xi \otimes \xi, JA)_\xi &= \int_M g^{ij} g_{kl} (J d \theta_\xi \otimes \xi)^k_i (JA)^l_j ~f e^{-2 \theta_\xi} \omega^n 
\\ &= \int_M (Jd\theta_\xi \otimes \xi)^k_i (JA)^i_k ~f e^{-2 \theta_\xi} \omega^n 
\\ &= \int_M - \theta_{\xi, p} \xi^k A^p_k ~f e^{-2 \theta_\xi} \omega^n 
\\ &= \frac{d}{dt}\Big{|}_{t=0} (- (J_t \xi) \theta_\xi, f)_\xi. 
\end{align*}
The basic formulas in Appendix B help us to compute the rest of them. 
\end{proof}

Now we observe that our moment map actually corresponds to K\"ahler--Ricci solitons. 

\begin{prop}
\label{soliton}
For simplicity, we let $s_\xi, \mathcal{S}_\xi$ stand for $s_{\xi, 0}, \mathcal{S}_{\xi, 0}$, respectively. 
The following (1)-(3) are equivalent for any integrable $J \in \mathcal{J}^\intg_T (M, \omega)$. 
\begin{enumerate}
\item $(g_J, \xi'_J)$ is a K\"ahler--Ricci soliton on $(M, J)$. 

\item $s_\xi (J) = 0$. 

\item $\mathcal{S}_\xi (J) = 0$. 
\end{enumerate}
\end{prop}

\begin{proof}
Provided that $g_J$ satisfies the K\"ahler--Ricci soliton equation $\Ric (g_J) -L_{\xi'_J}g_J = g_J$. 
The trace of this formula gives 
\begin{equation}
\label{soliton1} s (J) + \bar{\Box} \theta'_\xi = n. 
\end{equation}
Since $\xi'_J$ is holomorphic, the Lie derivative by $\xi'_J$ can be arranged as follows. 
\[ \sqddbar (\bar{\Box} \theta'_\xi - \xi'_J \theta'_\xi) = \sqddbar \theta'_\xi, \]
and hence $\bar{\Box} \theta'_\xi - \xi'_J \theta'_\xi - \theta'_\xi$ is constant. 
Recall that the operator $(\bar{\Box} -\xi'_J)$ is a formally self-adjoint elliptic operator with respect to the inner product $(\cdot, \cdot)_{\theta'_\xi}$ (see for example \cite[Section 2.4]{Fut-Book}, it is also shown in our Appendix B). 
It follows that the equation $(\bar{\Box} - \xi'_J) u = f$ has a solution $u$ if and only if $\int_M f e^{\theta'_\xi} \omega^n = 0$. 
This shows
\begin{equation}
\label{soliton2} \bar{\Box} \theta'_\xi - \xi'_J \theta'_\xi = \theta'_\xi
\end{equation}
under the normalization condition $\int_M \theta'_\xi e^{\theta'_\xi} \omega^n = 0$. 
Substituting (\ref{soliton2}), the equation (\ref{soliton1}) can be reformulated as
\[ (s (J) -n) + 2 \bar{\Box} \theta'_\xi - (\xi'_J \theta'_\xi + \theta'_\xi) = 0. \]
The left hand side of the equation is nothing but $s_\xi (J)$, and we obtain $s_\xi (J) =0$. 

Conversely, assume $\mathcal{S}_\xi (J) = 0$. 
Take a function $h$ so that $\sqddbar h = \Ric (\omega) - \omega$. 
Since $L_{\xi'_J} \omega = \sqddbar \theta'_\xi$, it is enough to show that $h - \theta'_\xi$ is actually constant. 
Similarly as before, the Lie derivative of $\sqddbar h = \Ric (\omega) - \omega$ gives
\[ \sqddbar \xi'_J h = \sqddbar (\bar{\Box} \theta'_\xi - \theta'_\xi) \]
and hence
\[ c_1 := \bar{\Box} \theta'_\xi - \theta'_\xi - \xi'_J h \]
is constant. 
We can rearrange the modified Hermitian scalar curvature as
\begin{align} 
s_\xi (J) &= -\bar{\Box} h + 2 \bar{\Box} \theta'_\xi - \xi'_J \theta'_\xi - \theta'_\xi \notag
\\  &= -\bar{\Box} h + 2 \bar{\Box} \theta'_\xi - \xi'_J \theta'_\xi - (\bar{\Box} \theta'_\xi - \xi'_J h - c_1) \notag
\\ \label{moment7} &= -\bar{\Box} (h - \theta'_\xi) + \xi'_J (h - \theta'_\xi) + c_1 . 
\end{align}
Now the assumption $\mathcal{S}_\xi (J) = 0$ implies that $c_2 := s_\xi (J)$ is a constant. 
Since $((\bar{\Box} -\xi'_J) (h- \theta'_\xi), 1)_{\theta'_\xi} = 0$, the constant $c_1 - c_2 = (\bar{\Box} - \xi'_J) (h- \theta'_\xi)$ has to be zero and we have shown that $h- \theta'_\xi$ is constant. 
\end{proof}

\begin{rem}
As noted in \cite{Don1} for the cscK (or K\"ahler--Einstein) case, our moment map picture enables us to interpret or even reproduce the following known results from more geometric viewpoint, which were originally proved in \cite{TZ2}. (See also \cite{Wang}. )
\begin{itemize}
\item The invariance of the modified Futaki invariant. 

\item The reductivity of $\mathrm{Aut} (X, \xi')$. (cf. \cite{FO}) 

\item The uniqueness of K\"ahler--Ricci soliton. (cf. \cite{BerBer})
\end{itemize}

For instance, substituting (\ref{moment7}), we obtain
\begin{align*}
\langle \mathcal{S}_\xi (J), f \rangle &= \int_M s_\xi (J) f ~e^{\theta'_\xi} \omega^n
\\ &= \int_M 4 (- (\bar{\Box} - \xi') (h - \theta'_\xi) + c_1) f ~e^{\theta'_\xi} \omega^n
\\ &= -4 \int_M (\bar{\partial} (h-\theta'_\xi), \bar{\partial} f) ~e^{\theta'_\xi} \omega^n
\\ &= -4 \int_M X_f' (h- \theta'_\xi) ~e^{\theta'_\xi} \omega^n
\\ &= -4 c \mathrm{Fut}_{\xi'} (X_f') 
\end{align*}
for $X'_f \in \mathrm{Lie} (\mathrm{Stab} (J))$, where $c = \int e^{\theta'_\xi} \omega^n /\int \omega^n$ is independent of $J$. 
Its invariance can be interpreted as coming from a general fact on moment maps: for any $x \in M$ and $v \in \mathrm{Lie} (K_x^c)$, $\langle \mu (xg), g^{-1} v \rangle$ is invariant for $g \in K^c$, where $\mu :M \to \mathfrak{k}^*$ is a moment map. 

% Take another $T$-invariant K\"ahler metric $\omega_\phi = \omega + \sqddbar \phi$ with respect to $J$. 
% There is a path $g_t \in \mathrm{Diff}_T (M)$ such that $i (\frac{d}{dt} g_t) (\omega + t \sqddbar \phi) = \sqrt{-1} \bar{\partial} \phi$ and $g_1^* \omega_\phi = \omega$. 
% Then 
% \begin{align*}
% \mathrm{Fut}_{\phi, \xi'_J} (v'_J) &= \int_M v'_J (h_{\phi, J} - \theta_{\phi, \xi}') ~ e^{\theta_{\phi, \xi}'} (\omega_\phi)^n 
% \\ &= \int_M g^* [ v'_J (h_{\phi, J} - \theta_{\phi, \xi}') ~ e^{\theta_{\phi, \xi}'} ]~ \omega^n
% \\ &= \int_M (g^{-1}_* (v'_J)) (h_{g^* J} - \theta_\xi') ~ e^{\theta_\xi'} \omega^n
% \\ &= \mathrm{Fut}_{\xi'_{g^* J}} ((g^{-1}_* v)'_{g^* J})
% \end{align*}
% for $v = X_f \in \mathfrak{k} = \mathrm{Lie} (\mathrm{Stab} (J))$. 
% It suffices to show the invariance of $\langle \mathcal{S}_\xi (g^*_t J), f \circ g_t \rangle$. 

% Since $(g_t^{-1})_* (J \frac{d}{dt} g_t) = (g_t^* J) ((g_t^{-1})_* \frac{d}{dt} g_t) \in \mathrm{Lie} (\mathrm{Stab} (g_t^* J))$, 
% \begin{align*} 
% \frac{d}{dt} \langle \mathcal{S}_\xi (g^*_t J), g_t^* f \rangle &= - \Omega_\xi (L_{(g_t^{-1})_* v} g^*_t J, \dot{J}_g) + \langle \mathcal{S}_\xi (g^*_t J), -((g_t^* J) (g_t^{-1})_* (J \frac{d}{dt} g_t)) f \rangle 
% \\ &= 0. 
% \end{align*}

Proposition \ref{moment} in particular shows that $\mathrm{Fut}_{\xi'}|_{\mathfrak{t}}$ is invariant under $T$-equivariant complex deformation: it only depends on the $T$-equivariant symplectic structure. 
\end{rem}

We call a symplectic Fano $T$-manifold $(M, \omega, T)$ \textit{K-optimal} if there is a $T$-invariant $\omega$-compatible integrable complex structure $J_0$ (in this case, $(M, J_0)$ is a Fano manifold) and whose Hilbert character $(T_\mathbb{C}, \chi (M, J_0))$ is K-optimal. 
Its Hilbert character $\chi_m$, seen as a (real analytic) function $\mathfrak{t} \to \mathbb{R}$ by $\xi \mapsto \sum_{i} \sum_{u \in M} h^i_{X, u} (m) \langle u, \xi \rangle$, can be computed by the equivariant Hirzebruch--Riemann--Roch formula (\cite{Mei}): 
\[ \chi_m (\xi) = \int_M \mathit{Ch}_{\mathfrak{t}} (- K_{(M, \omega)}, \xi) \mathit{Td}_{\mathfrak{t}} (M, \omega, \xi) \]
near $\xi = 0$. 
Here the equivariant Chern character $\mathit{Ch}_{\mathfrak{t}} (-K_{(M, \omega)}, \cdot)$ and Todd character $\mathit{Td}_{\mathfrak{t}} (M, \omega, \cdot)$ is defined as the equivariant cohomology classes of the following $T$-equivariant forms, which is independent of the choice of $J \in \mathcal{J}_T (M, \omega)$: 
\begin{align*} 
\mathit{Ch}_{\mathfrak{t}} (-K_{(M, \omega)}, \cdot) 
&:= e^{\mathrm{tr} (\frac{\sqrt{-1}}{2 \pi} F_{\mathfrak{t}} (g_J, \cdot))} =  e^{\omega + \langle \mu, \cdot \rangle}, 
\\
\mathit{Td}_{\mathfrak{t}} (M, \omega, \cdot) 
&:= \det \Big{(} \frac{\frac{\sqrt{-1}}{2\pi} F_\mathfrak{t} (g_J, \cdot)}{1 - e^{- \frac{\sqrt{-1}}{2\pi} F_\mathfrak{t} (g_J, \cdot)}} \Big{)}, 
\end{align*}
where $\mathit{Td}_{\mathfrak{t}} (M, \omega, \xi)$ is defined near $\xi = 0$. 
Here the equivariant curvature $F_{\mathfrak{t}} (g_J, \xi)$ is given by 
\[ F_\mathfrak{t} (g_J, \xi) := F_{g_J} + 2 \pi \sqrt{-1} (L_\xi - \nabla^{g_J}_\xi). \]

Although we firstly use the integrable complex structure $J_0$ to define the Hilbert character, its Hilbert character can be computed by the $T$-equivariant characteristic classes associated to the symplectic $T$-manifold $(M, \omega, T)$, which makes sense at least near $\xi = 0$ even when there is no $T$-invariant integrable complex structures. 
In particular, the Fano character $(T_\mathbb{C}, \chi (M, J))$ is independent of the choice of integrable $J \in \mathcal{J}_T (M, \omega)$ (in other words, it is well-defined for $(M, \omega, T)$) and is K-optimal for every $J$ if $(M, \omega, T)$ is K-optimal (i.e. if $(T_\mathbb{C}, \chi (M, J_0))$ is K-optimal for some $J_0$). 
Beware that even when $(M, \omega, T)$ is K-optimal and the action $(M, J_0) \curvearrowleft T$ is K-optimal for some $J_0$, \textit{the action} $(M, J) \curvearrowleft T$ might be not K-optimal for other integrable complex structure $J \in \mathcal{J}_T (M, \omega)$ as the action might be not maximal among actions with K-optimal characters. 

We denote by $\mathcal{S}_\xi^\intg$ the restriction of the moment map $\mathcal{S}_\xi :\mathcal{J}_T (M, \omega) \to C^\infty_{T, \xi} (M, \omega)^*$ to the subspace $\mathcal{J}_T^\intg (M, \omega)$, which consists of integrable complex structures. 

\begin{prop}
\label{isomorphism class}
Assume the action of $T$ on $(M, \omega)$ is K-optimal. 
Then the following two statements are equivalent for any integrable $J, J' \in (\mathcal{S}_\xi^\intg)^{-1} (0)$. 
\begin{enumerate}
\item There is a $T$-equivariant $C^\infty$-diffeomorphism $\phi :M \xrightarrow{\sim} M$ such that $J = \phi^* J'$. 

\item $[J] = [J'] \in (\mathcal{S}_\xi^\intg)^{-1} (0)/\Ham_T (M, \omega)$. 
\end{enumerate}
\end{prop}

\begin{proof}
It follows from the uniqueness of K\"ahler--Ricci soliton and $\mathrm{Aut}_T (X) = \mathrm{Aut} (X, \xi)$ from the K-optimal action. 
\end{proof}

\begin{rem}
The above proposition would hold without K-optimal assumption. 
To see this, it suffices to prove the following uniqueness claim. 

Claim: \textit{If $g_1, g_2$ are two $T$-invariant K\"ahler--Ricci solitons on a Fano $T$-manifold $X$ (in this case, we have the same soliton vectors $\xi_1 = \xi_2 \in \mathfrak{t}$), then there is an element $\phi \in \mathrm{Aut}_T^0 (X)$ such that $g_2 = \phi^* g_1$. }

In general, it seems not so easy to verify the K-optimality of a given torus action on a Fano manifold, especially when the dimension of the center of its maximal reductive subgroup is greater than one. 
From this point, it may be better to consider non K-optimal actions for studying explicit description of the moduli space of Fano manifolds with K\"ahler--Ricci solitons in some special cases. 
Indeed, for instance, the claim holds at least for a maximal torus, as the Weyl group $N_T/T$ can be represented by the elements of any maximal compact $K$ including the maximal compact torus $T_\mathbb{R} \subset T$. 
\end{rem}

It follows that the quotient $(\mathcal{S}^\intg_\xi)^{-1} (0)/\Ham_T (M,\omega)$ can be identified with the set of biholomorphism classes of Fano manifolds admitting K\"ahler--Ricci solitons with the fixed underlying symplectic structure $(M, \omega)$, as sets. 
Therefore, this quotient space must be the support set of our moduli space. 
The quotient topology on this set is Hausdorff (cf. \cite{FS1}). 
We exhibit the proof for the readers' convenience. 

\begin{prop}
\label{Hausdorff quotient}
The action of $\Ham_T (M, \omega)$ on $\mathcal{J}_T (M, \omega)$ is proper. 
In particular, the quotient topological space $(\mathcal{S}^\intg_\xi)^{-1} (0)/\Ham_T (M, \omega)$ is Hausdorff. 
\end{prop}

\begin{proof}
We must show that the map
\[ a :\mathcal{J}_T (M, \omega) \times \mathrm{Ham}_T (M, \omega) \to \mathcal{J}_T (M, \omega) \times \mathcal{J}_T (M, \omega) : (J, \phi) \mapsto (J, \phi^* J) \]
is proper. 
Take a sequence $(J_n, \phi_n)$ so that $J_n, \phi_n^* J_n$ converge to some $J_\infty, J_\infty' \in \mathcal{J}_T (M, \omega)$ in the given order. 
It suffices to show that a subsequence of $\phi_n$ converges to some $\phi_\infty \in \Ham_T (M, \omega)$ satisfying $\phi^*_\infty J_\infty = J_\infty'$. 
Let $g_\infty, g_\infty'$ denote the Riemannian metrics associated to $J_\infty, J_\infty'$, respectively. 

Let us take a dense countable subset $S$ of $M$. 
The diagonal argument shows that we have a subsequence of $\phi_n$ so that $\phi_n (x)$ converges for any $x \in S$. 
We continue to write $\phi_n$ for this subsequence. 
We obtain a distance preserving map $\phi_{S, \infty} :(S, d_{g_\infty'}|_S) \to (M, d_{g_\infty})$ by putting $\phi_{S, \infty} (x) := \lim_{n \to \infty} \phi_n (x)$. 
Then this map can be uniquely extended to a distance preserving map $\phi_\infty :(M, d_{g_\infty'}) \to (M, d_{g_\infty})$. 
Similarly we obtain a distance preserving map $\psi_\infty :(M, d_{g_\infty}) \to (M, d_{g_\infty'})$ as a limit of $\phi_n^{-1}$. 
It follows from \cite[Theorem 1.6.14]{BBI} that the distance preserving endomorphism $\phi_\infty \circ \psi_\infty$ is surjective, and we conclude $\phi_\infty$ is a continuous bijective map. 

Thanks to Myers-Steenrod theorem, we see that $\phi_\infty$ is a $C^\infty$-diffeomorphism with $\phi_\infty^* g_\infty = g_\infty'$. 
Moreover, since $g_n = g_{n, i j} dx^i dx^j, \phi_n^* g_n = g_{n, pq} (d \phi_n)^p_i (d\phi_n)^q_j dx^i dx^j$ respectively converge to $g_\infty, \phi^*_\infty g_\infty$ in $C^\infty$-topology, we see the $C^\infty$-convergence of the coefficients $(d\phi_n)^p_i (d\phi_n)^q_j$ to $(d\phi_\infty)^p_i (d\phi_\infty)^q_j$ with respect to a fixed $C^\infty$-coordinate. 
In particular, we have
\begin{equation}
\label{Hausdorff quotient1} (\partial_k (d\phi_n)^p_i) (d \phi_n)^p_i \xrightarrow{C^\infty} (\partial_k (d\phi_\infty)^p_i) (d\phi_\infty)^p_i
\end{equation}
and $(d\phi_n)^p_i = \sqrt{((d\phi_n)^p_i)^2}$ convereges to $(d\phi_\infty)^p_i$ in $C^0$-topology. 
It follows from (\ref{Hausdorff quotient1}) that the $C^k$-convergence of $d\phi_n$ induces the $C^{k+1}$-convergence of them. 
This shows that $\phi_n$ converges to $\phi_\infty \in \Ham_T (M, \omega)$ in the $C^\infty$-topology and $\phi_\infty^* J_\infty = J'_\infty$. 
\end{proof}

\subsection{Local slice}

The materials in this subsection are parallel to \cite{Sze2}, where the cscK case is treated. 

Let $X$ be a Fano $T$-manifold with a K\"ahler--Ricci soliton $(g, \xi')$, $\phi :(M, J_0) \xrightarrow{\sim} X$ be a biholomorphism, where $M$ is a $C^\infty$-manifold and $J_0$ is a complex structure on $M$. 
Put $\omega := (\phi^* g) (J_0 \cdot, \cdot)$, $K := \{ h \in \mathrm{Ham} (M, \omega) ~|~ h^* J_0 = J_0 \}$, which is a compact Lie group, and $\mathfrak{k} := \{ f \in C^\infty_T (M) ~|~ L_{X_f} J_0 = 0, \int_M f e^{\theta'_\xi} \omega^n = 0 \}$, which can be identified with the Lie algebra of $K$. 
Consider the following $L^2_k$-completion of the moment map in the last subsection
\[ \mathcal{S}_\xi :\mathcal{J}_T (M, \omega)^2_k \to L^2_{k-2, T} (M, \omega)^\vee. \]
We denote by $\Theta$ the holomorphic tangent sheaf of $X$ and by $H^i_T (X, \Theta)$ the $T$-invariant subspace of the $i$-th cohomology $H^i (X, \Theta)$. 
Note that we have $H^i (X, \Theta) = 0$ for every $i \ge 2$ and a smooth Fano manifold $X$, thanks to Serre duality and Kodaira vanishing. 

\begin{prop}
\label{local slice}
There are an open ball $B \subset H^1_T (X, \Theta)$ centered at the origin, a $K$-equivariant holomorphic deformation $\varpi :\mathcal{X} \to B$ of $X$ with a holomorphic morphism $\iota :X \hookrightarrow \mathcal{X}_0$ inducing a biholomorphism to the central fiber, a $K$-equivariant $C^\infty$-smooth map $\mathfrak{J} :B \to \mathcal{J}_T (M, \omega)^2_k$ and a $T_\mathbb{R}$-equivariant $L^2_k$-regular diffeomorphism $\Phi :B \times M \xrightarrow{\sim} \mathcal{X}$ with the following properties. 
\begin{enumerate}
\item The holomorphic family $X \lhook \joinrel \xrightarrow{\iota} \mathcal{X} \xrightarrow{\varpi} (B,0)$ is a semi-universal family of $X$. 

\item For each $b \in B$, $\mathfrak{J} (b)$ is an $L^2_k$-regular integrable complex structure satisfying $s_\xi (\mathfrak{J} (b)) \in \mathfrak{k}$ and $\mathfrak{J} (0) = J_0$. 

\item The diffeomorphism $\Phi$ satisfies $\varpi \circ \Phi^{-1} = p_B$ and $\Phi (0, \cdot) = \phi$, where $p_B: B \times M \to B$ is the projection. 
The restricted map $\Phi (b, \cdot) :(M, \mathfrak{J} (b)) \to \mathcal{X}_b$ is a biholomorphism for each $b \in B$. 
\end{enumerate}
\end{prop}

\begin{proof}
Let $\varpi: \mathcal{X} \to B$ be the Kuranishi family of $T$-equivariant deformation of $X$ (see \cite{Kur1, Kur2, Dou1} for its construction). 
From its construction, we have a holomorphic $K$-action on $\mathcal{X}$ and $B$ so that $\varpi$ is $K$-equivariant and a holomorphic map $\mu :B \to \mathcal{J}_T (M)^2_k$ whose image $\mu (b)$ is a real analytic (with respect to the real analytic structure on $X$) integrable complex structure for each $b \in B$ with a biholomorphism $\mathcal{X}_b \cong (M, \mu (b))$. 
As $- \ell K_{\mathcal{X}/B}$ is relatively very ample for large $\ell \in \mathbb{N}$, all the higher direct images of $\mathcal{O} (-\ell K_{\mathcal{X}/B})$ vanishes and thus $\varpi_* \mathcal{O} (-\ell K_{\mathcal{X}/B})$ is a $K$-equivariant vector bundle on $B$. 
Taking smaller $B$ and using a $K$-equivariant isomorphism $\varpi_* \mathcal{O} (-\ell K_{\mathcal{X}/B}) \cong \underline{H^0 (X, \mathcal{O} (-\ell K_{\mathcal{X}/B}))}_{B, K} = B \times H^0 (X, \mathcal{O} (-\ell K_{\mathcal{X}/B}))$ of vector bundles on $B$ (see Lemma \ref{equivariant trivialization} for the $K$-action), we can embed these Fano manifolds into a uniform projective space $\mathbb{C}P^N = \mathbb{P} ( H^0 (X, \mathcal{O} (-\ell K_{\mathcal{X}/B}))^\vee)$ so that $\mathcal{X}_{s. g} = \mathcal{X}_s. g$, where in the latter we consider the $K$-action on $\mathbb{C}P^N$ induced from the action on $H^0 (X, \mathcal{O} (-\ell K_{\mathcal{X}/B}))$. 
Pulling back the Fubini-Study metric, we obtain a $K$-equivariant smooth family of K\"ahler metrics $\{ \omega_b \}_{b \in B}$, where each $\omega_b$ can be identified with a K\"ahler metric on $(M, \mu (b))$. 
Taking smaller $B$ again, we can assume that closed forms $\omega_{b, t} := \omega_0 + t (\omega_b - \omega_0)$ are non-degenerate for each $b \in B$ and $t \in [0,1]$. 
Then we can find a $K$-equivariant family of diffeomorphisms $\{ f_b \}_{b \in B}$ so that $f_b^* \omega_b = \omega_0$. 
Putting $\mathfrak{J}' (b) := f^*_b \mu (b)$, we obtain a $K$-equivariant smooth map $\mathfrak{J}' :B \to \mathcal{J}_T (M, \omega)^2_k$, whose image $\mathfrak{J}' (b)$ is a smooth complex structure for each $b \in B$. 

It suffices to show that we can find an equivariant perturbation $\mathfrak{J}$ of $\mathfrak{J}'$ so that $\mathfrak{J} (b) = g_b^* \mathfrak{J}' (b)$ for each $b$ and $s_\xi (\mathfrak{J} (b)) \in \mathfrak{k}$. 
Let $U^2_{k+2} \subset L^2_{T, k+2} (M, \omega)$ be a small ball of the origin. 
For each $\phi \in U^2_{k+2}$ and an almost complex structure $J \in \mathcal{J}_T (M, \omega)$, we can find an $L^2_k$-regular vector field $X^{\phi, J}_t$ on $M$ so that $i (X^{\phi, J}_t) (\omega_0 - tdJd \phi) = - Jd \phi$. 
This vector field is actually $L^2_{k+1}$-regular. 
In fact, it is sufficient to show that $[X_f, X^{\phi, J}_t]$ is $L^2_k$-regular for any smooth function $f$. 
For any smooth vector field $Z$, we have 
\begin{align*}
i ([X_f, X^{\phi, J}_t]) \omega_t (Z) 
&= -(L_{X_f} \omega) (X_{\phi, J, t}, Z) + X_f (\omega (X^{\phi, J}_t, Z)) - \omega (X^{\phi, J}_t, [X_f, Z])
\\ 
&= X_f (-Jd\phi (Z)) + J d\phi ([X_f, Z]) \in L^2_k. 
\end{align*}
Thus $[X_f, X^{\phi, J}_t]$ is $L^2_k$, as we expected. 
The flow $f^{\phi, J}_t$ of this time-dependent vector fields is $L^2_{k+1}$-regular and satisfy $(f^{\phi, J}_t)^* (\omega_0 - t dJd \phi) = \omega_0$. 
To see the regularity, it is sufficient to show that $(f^{\phi, J}_t)_* Y$ is a $L^2_k$-regular vector field for each smooth vector field $Y$ on $M$. 
Note that $(d/dt) (f^{\phi, J}_t)_* Y = [X^{\phi, J}_t, Y]$ is $L^2_k$-regular and $(f^{\phi, J}_t)_* Y$ can be written as $\int_0^t [X^{\phi, J}_s, Y] ds$. 
Then for each $l \le k$, we obtain the following estimate, so $f^{\phi, J}_t$ is $L^2_{k+1}$-regular. 
\begin{align*}
\int_M |\nabla^l (f^{\phi, J}_t)_* Y|^2 \omega^n &= \int_M \left|\int_0^t \nabla^l [X^{\phi, J}_s, Y] ds \right|^2 \omega^n
\\ &\le \int_M t \int_0^t |\nabla^l [X^{\phi, J}_s, Y]|^2 ds~ \omega^n 
\\ &\le t \int_0^t \| [X^{\phi, J}_s, Y] \|^2_{L^2_k} ds < \infty.
\end{align*}
It follows that $(f^{\phi, J}_1)^* J \in \mathcal{J}_T (M, \omega)^2_k$. 
Consider the orthogonal decomposition $L^2 = \mathfrak{k} \oplus \mathfrak{k}_\bot$ with respect to $L^2$-norm $(\cdot, \cdot)_\xi$. 
Put $\mathfrak{k}^2_{k, \bot} := L^2_k \cap \mathfrak{k}_\bot$ and let $\Pi_\bot :L^2_{k-2} \to \mathfrak{k}^2_{k-2, \bot}$ be the $L^2$-projection. 
Note that 
\[ (D (\phi \mapsto (f^{\phi, J_0}_1)^* J_0))_0 (\psi) = \frac{d}{dt} f_t^* J_0 = J_0 P (\psi), \]
where $P$ denotes the linear differential operator $P :L^2_{T, k+2} \to T_J \mathcal{J}_T (M, \omega)^2_k: \psi \mapsto L_{X_\psi} J_0$, and $(Ds_\xi)_J (A) = P^* J A$, where $P^*$ is the formal adjoint of $P$ with respect to the norm $(-, -)_\xi$. 
It follows that 
\[ G: B \times U \to \mathfrak{k}^2_{k-2, \bot} : (b, \phi) \mapsto \Pi_\bot s_\xi ((f^{\phi, \mathfrak{J}' (b)}_1)^* \mathfrak{J}' (b)) \]
is a $K$-equivariant smooth map with the derivative
\[ DG_{(0, 0)} (0, \psi) = - P^* P (\psi). \]
Since $P^* P$ is a self-adjoint fourth order elliptic differential operator, it gives the isomorphism $P^* P :\mathfrak{k}^2_{k+2, \bot} \to \mathfrak{k}^2_{k-2, \bot}$. 
Applying the implicit function theorem, we can find a new $K$-equivariant smooth map $\mathfrak{J} :B \to \mathcal{J}_T (M, \omega)^2_k$ so that $\Pi_\bot s_\xi (\mathfrak{J} (b)) = 0$, hence $s_\xi (\mathfrak{J} (b)) \in \mathfrak{k}$, taking smaller $B$ if necessary. 
\end{proof}

Pulling back the symplectic structure $\Omega_\xi$ on $\mathcal{J}_T (M, \omega)^2_k$ by the $K$-equivariant smooth map $\mathfrak{J} :B \to \mathcal{J}_T (M, \omega)^2_k$, we obtain a $K$-equivariant smooth symplectic structure (by taking smaller $B$ if necessary), which we denote by the same notation. 
Then $\nu :B \to \mathfrak{k}^* :b \mapsto \mathcal{S}_\xi (\mathfrak{J} (b))$ is a moment map with respect to this symplectic structure. 

\begin{prop}[See Postscript Remark below.]
\label{GITpolystable}
If $B \subset H^1_T (X, \Theta)$ in the Proposition \ref{local slice} is sufficiently small, then the following two statements are equivalent for any $b \in B$. 
\begin{enumerate}
\item The fiber $\mathcal{X}_b$ of the family $\varpi :\mathcal{X} \to B$ has a K\"ahler--Ricci soliton. 

\item The orbit $b \cdot \mathrm{Aut}_T (X) \subset H^1_T (X, \Theta)$ is closed. That is, $b$ is polystable with respect to the $\mathrm{Aut}_T (X)$-action. 
\end{enumerate}
\end{prop}

\begin{postrem}
While discussing with R. Dervan and P. Naumann, the author realized that there was a gap in the following proof with regards to the implication ``the existence of K\"ahler--Ricci soliton $\Rightarrow$ GIT-polystability''. 
To be precise, what we prove here is that the following are equivalent for $b \in B$: 
\begin{enumerate}
\item $\nu^{-1} (0) \cap B \cap b \cdot G \neq \emptyset$. 

\item $\nu^{-1}_0 (0) \cap B \cap b \cdot G \neq \emptyset$. 

\item The point $b \in H^1_T (X, \Theta)$ is polystable with respect the $\mathrm{Aut}_T (X)$-action. 
\end{enumerate}

Of course, this (1) implies the existence of K\"ahler--Ricci soliton on $\mathcal{X}_b$. 
On the other hand, however, the existence of K\"ahler--Ricci soliton on $\mathcal{X}_b$ only implies that there is a unique orbit $B \cap b_0 \cdot G$ in the closure $B \cap \overline{b \cdot G}$ such that $\nu^{-1} (0) \cap B \cap b_0 \cdot G \neq \emptyset$ and $\mathcal{X}_{b'}$ is isomorphic to $\mathcal{X}_b$ for any $b' \in B \cap b_0 \cdot G$ (thanks to K-polystability). 
This in particular implies that $\nu^{-1} (0)/K \approx BK^c \sslash K^c$ (Corollary \ref{Kempf-Ness} below) can be naturally identified with the isomorphism classes of Fano manifolds admitting K\"ahler--Ricci solitons who appear in the family $\varpi: \mathcal{X} \to B$ (thanks to Corollary \ref{local-global} below). 

The author emphasizes that we do not use the original statement of Proposition \ref{GITpolystable} to prove all the statements in the rest of this paper and we only use the equivalence stated in this remark. 
The original proof is still fine to show this equivalence. 
\end{postrem}

\begin{proof}[Proof of Proposition \ref{GITpolystable}]
Let $\Omega_0$ be the linearization of $\Omega_\xi$ at $0 \in B$, i.e., $\Omega_0 = (d_0 \mathfrak{J} \cdot, J_0 d_0 \mathfrak{J} \cdot)_\xi$ under the identification $T_b B = \mathbb{H}^1_T = T_0 B$. 
Consider the map $\nu_0 :\mathbb{H}^1_T \to \mathfrak{k}^*$ defined by
\[ \langle f, \nu_0 (b) \rangle = \Omega_0 (L_{X_f} b, b) = (L_{X_f} d\mathfrak{J}_0 b, J_0 d \mathfrak{J}_0 b)_\xi. \]
Then $\nu_0$ is a moment map with respect to the symplectic structure $\Omega_0$. 
The Kempf-Ness theorem says that $b \in \mathbb{H}^1_T$ is polystable with respect to $K^c = \mathrm{Aut}_T (X)$ if and only if $b K^c \cap \nu_0^{-1} (0) \neq \emptyset$. 

Since
\[ \begin{split} \frac{d^2}{dt^2} \Big{|}_{t=0} \langle f, \nu (tb) \rangle &= \frac{d^2}{dt^2} \Big{|}_{t=0} (f, s_\xi (\mathfrak{J} (tb)))_\xi \\ &= \frac{d}{dt} \Big{|}_{t=0} (L_{X_f} \mathfrak{J} (tb), \mathfrak{J} (tb) \dot{\mathfrak{J}} (tb))_\xi \\ &= \langle f, \nu_0 (b) \rangle, \end{split} \]
the moment map $\nu :B \to \mathfrak{k}^*$ can be expanded as
\[ \nu (tb) = \nu (0) + t d_0 \nu (b) + t^2 \nu_0 (b)/2 + O (t^3). \]
Since $0 \in B$ corresponds to Fano manifolds with K\"ahler--Ricci soliton $(M, J_0, \omega)$, $\nu (0) = \mathcal{S}_\xi (J_0) = 0$ from Proposition \ref{soliton}. 
Moreover, since $0$ is a fixed point of the $K$-action, we have $d_0 \nu = 0$. 
Therefore we get
\[ \nu (tb) = t^2 \nu_0 (b)/2 + O (t^3). \]
Since the action of $K$ on $\mathbb{H}^1_T$ is linear, the stabilizer group $K_b \subset K$ of $b$ satisfies $K_{tb} = K_b$. 
So we have
\[ \frac{d}{dt} \langle f , \nu (tb) \rangle = \Omega_{tb} (b, \sigma_{tb} (f)) = 0 \]
for any $f \in \mathfrak{k}_b$, where $\sigma_b :\mathfrak{k} \to T_x B$ is the differential of the action.  
Then it follows that $\nu (b) \in \mathfrak{k}_b^\bot$ and $\nu_0 (b) \in \mathfrak{k}_b^\bot$. 

Now we cite the following general lemma from \cite[Proposition 9]{Sze2} and \cite[Proposition 17. ]{Don4}. 
\begin{lem}
Let $(B, \Omega)$ be a symplectic manifolds with a $K$-action,  $\nu :B \to \mathfrak{k}$ be a moment map with respect to the $K$-aciton ($\mathfrak{k}$ is endowed with a inner product). 
Suppose $b \in B$ satisfies $\nu (b) \in \mathfrak{k}_b^\bot$ and $\lambda, \delta > 0$ with $\lambda \| \nu (b) \| < \delta$ satisfies $\| (\sigma^*_{e^{iv} b} \sigma_{e^{iv} b})^{-1} \| \le \lambda$ for any $v \in \mathfrak{k}$ with $\| v \| < \delta$. 
Then there is $v_b \in \mathfrak{k}$ such that $\nu (e^{i v_b} b) = 0$ and $\| v_b \| \le \lambda \| \nu (b) \|$. 
\end{lem}
Fix a small $\delta > 0$ so that there is $C > 0$ such that for any $v \in \mathfrak{k}$ with $\| v \| < \delta$ and any $f \in \mathfrak{k}^\bot_{e^{i v} b}$ 
\[ \| \sigma_{e^{iv} b} (f) \|_{\Omega_0}^2  \ge C \| f \|^2 \]
holds. 
Take smaller $B$ so that $\Omega_\xi \ge \frac{1}{2} \Omega_0$. 
Since $\sigma_{tx} (f) = t \sigma_x (f)$ and 
\[ ( \sigma_{tx}^* \sigma_{tx} (f), f )_\xi = \| \sigma_{tx} (f) \|^2_\Omega \ge \frac{1}{2} C t^2 \| f \|^2, \]
we obtain $\| (\sigma_{tx}^* \sigma_{tx})^{-1} \| \le C' t^{-2}$. 
Replacing $\Omega$ with $\Omega_0$, we obtain the similar estimate for the adjoint of $\sigma$ with respect to $\Omega_0$. 

Suppose $b \in B$ is polystable. 
Then there exists a point $b' \in b K^c \cap \nu^{-1}_0 (0)$.
In regards of the linear symplectic form, $b'$ is given by minimizing the norm of $b'$ in the $K^c$-orbit of $b$, so $b'$ is also in $B$. 
Since the points in the same $K^c$-orbit give the isomorphic complex structures, we can assume $\nu_0 (b) = 0$. 
It follows that $\nu (tb) = O (t^3)$. 
Then we can take $t$ small so that $C' t^{-2} \| \nu (tb) \| < \delta$. 
Applying the above lemma, we find a point $tb' \in B$ in the $K^c$-orbit of $tb$ satisfying $\nu (tb') = 0$. 
It follows that $(M, \mathfrak{J} (tb)) \cong (M, \mathfrak{J} (tb'))$ admits K\"ahler--Ricci soliton. 
Note the polystability of $b$ and $tb$ is equivalent as we consider a linear action. 

Conversely, suppose $(M, \mathfrak{J} (b))$ admits K\"ahler--Ricci soliton. 
Then similarly we can show that there is a point $b' \in bK^c$ satisfying $\nu_0 (b') = 0$. 
This shows $b$ is polystable. 
\end{proof}

The following corollary exhibits one of good features of our $T$-equivariant formulation. 
We use this to show the Artinianity of our moduli stack in the next section. 

\begin{cor}
\label{gentle is open condition}
Any $T$-equivariant small deformation of Fano $T$-manifold with K\"ahler--Ricci soliton is gentle. 
In particular, for any $T$-equivariant family $\mathcal{M} \to S$ of complex manifolds, the following subset  
\[ S^\circ := \{ s \in S ~|~ \mathcal{M}_s \text{ is a gentle Fano manifold } \} \]
is an open subset of $S$ (with respect to the real topology). 
\end{cor}

\begin{proof}
Suppose the Fano manifold $(M, \mathfrak{J} (b))$ does not admit K\"ahler--Ricci soliton for the point $b \in B$. 
From the above proposition, $b \in B$ is not polystable. 
Then we can find a polystable point $b_0 \in B$ in the closure of the orbit $bK^c$ by minimizing the norm $\Omega_0 (-, J_0-)$. 
Since $K^c$ is reductive, we can find a regular morphism $\lambda: \mathbb{C}^* \to \mathbb{H}^1_T$ so that $\lambda (t) \to b_0$. 
We can extend this to a regular morphism $\tilde{\lambda} :\mathbb{C} \to \mathbb{H}^1_T$. 
Pulling back the family $\varpi :\mathcal{X} \to B$, we obtain a $T$-equivariant holomorphic family $\mathcal{M} \to \Delta$ whose central fiber $(M, \mathfrak{J} (b_0))$ has K\"ahler--Ricci soliton because there is some $b_0' \in b_0 K^c$ such that $\nu (b_0) = \mathcal{S}_\xi (\mathfrak{J} (b_0')) = 0$. 
So $\mathcal{M} \to \Delta$ gives a gentle degeneration of $\mathcal{X}_b$, hence $\mathcal{X}_b$ is gentle. 
Since the family $\varpi :\mathcal{X} \to B$ parametrizes all isomorphism classes of complex structures near $\mathcal{X}_b$ for any $b \in B$, we have shown the assertion. 
\end{proof}

\subsection{Completion}

The topological space $\Ham_T (M, \omega)^2_{k+1}$ of $L^2_{k+1}$-regular symplectic diffeomorphisms admits a natural Banach smooth manifold structure (cf. \cite{IKT, KriMic}). 
The compositions and the inverses of morphisms in $\Ham_T (M, \omega)^2_{k+1}$ are again in $\Ham_T (M, \omega)^2_{k+1}$. 
However, the following maps
\begin{align*}
\Ham_T (M, \omega)^2_{k+1} \times \Ham_T (M, \omega)^2_{k+1} \to \Ham_T (M, \omega)^2_{k+1} &:(\phi, \psi) \mapsto \phi \circ \psi 
\\ \Ham_T (M, \omega)^2_{k+1} \to \Ham_T (M, \omega)^2_{k+1} &:\phi \mapsto \phi^{-1} 
\end{align*}
are \textit{not} differentiable with respect to the Banach smooth manifold structure, but are just continuous (see \cite{IKT}). 
Therefore we can not treat $\Ham_T (M, \omega)^2_{k+1}$ as a Banach Lie group. 

Nevertheless, we can consider the following $C^1$-smooth map
\[ \mathcal{H} :B \times^K \Ham_T (M, \omega)^2_{k+1} \to \mathcal{J}_T (M, \omega)^2_k :[b, \phi] \mapsto \phi^* \mathfrak{J} (b) \]
by working with a slightly regular target of $\mathfrak{J}$ in Proposition \ref{local slice}, say, by working with $\mathfrak{J}: B \to \mathcal{J}_T (M, \omega)^2_{k+2}$. 
Note, first of all, the quotient $B \times^K \Ham_T (M, \omega)^2_{k+1} := B \times \Ham_T (M, \omega)^2_{k+1}/K$ is endowed with a unique Banach smooth manifold structure whose quotient map is a submersion, as the finite dimensional compact Lie group $K$ acts freely on $B \times \Ham_T (M, \omega)^2_{k+1}$. 
The $C^1$-smoothness of $\mathcal{H}$ follows from the $C^\infty$-smoothness of $\mathfrak{J} :B \to \mathcal{J}_T (M, \omega)^2_{k+2}$ and the $C^1$-smoothness of
\begin{align*} 
\mathcal{J}_T (M, \omega)^2_{k+2} \times \Ham_T (M, \omega)^2_{k+1} 
&\to \mathcal{J}_T (M, \omega)^2_k 
\\ 
\qquad\qquad (J\quad, \quad\phi) \quad\qquad\qquad
&\mapsto \quad \phi^* J, 
\end{align*}
which follows from the main theorem of \cite{IKT}. 

As Proposition \ref{Hausdorff quotient}, the map 
\begin{align*} 
a^2_k 
&: \mathcal{J}_T (M, \omega)^2_k \times \Ham_T (M, \omega)^2_{k+1} \to \mathcal{J}_T (M, \omega)^2_k \times \mathcal{J}_T (M, \omega)^2_k 
\\ 
& \qquad\qquad (J\quad, \quad\phi) \quad\qquad\qquad\longmapsto \qquad\quad (J\quad, \quad\phi^* J), 
\end{align*}
is proper for any large $k$ ($L^2_k \subset C^2$ is sufficient). 
To see this, take a sequence $(J_n, \phi_n) \in \mathcal{J}_T (M, \omega)^2_k \times \Ham_T (M, \omega)^2_{k+1}$ so that $g_n, \phi^*_n g_n$ converge to $g_\infty, g'_\infty$ in $L^2_k$-topology. 
Construct $\phi_\infty$ as in the proof of Proposition \ref{Hausdorff quotient}. 
Again, thanks to Myers-Steenrod theorem, $\phi_\infty$ is $C^2$-smooth and satisfies $\phi_\infty^* g_\infty = g_\infty'$. 
Then $\phi_\infty$ is a harmonic map between $(M, g_\infty)$ and $(M, g_\infty')$. 
Hence it satisfies the elliptic equation
\[ \Delta_{g_\infty'} \phi_\infty^\alpha - \Gamma^\alpha_{\beta \gamma} \frac{\partial \phi_\infty^\beta}{\partial x^i} \frac{\partial \phi_\infty^\gamma}{\partial x^j} {g'}^{ij}_\infty = 0, \]
where the coefficients of the Laplacian $\Delta_{g_\infty'}$ and the Levi-Civita connection $\Gamma^\alpha_{\beta \gamma}$ are $L^2_{k-1}$-regular. 
It follows that $\phi_\infty$ is $L^2_{k+1}$-regular. 

Let us see that $\phi_n$ converges to $\phi_\infty$ in $L^2_{k+1}$-topology. 
Since $g_n \to g_\infty$ and $g_n' := \phi_n^* g_n \to g_\infty'$ in $L^2_k$-topology, we have $\Gamma^\alpha_{\beta \gamma, n} \to \Gamma^\alpha_{\beta \gamma}$ and $\Delta_{g_n'} \to \Delta_{g_\infty'}$ in $L^2_{k-1}$-topology. 
Now we use the following uniform elliptic estimates for the elliptic operators $\Delta_{g_n'}$ ($n=1, 2, \ldots, \infty$) with $L^2_{k-1}$-bounded coefficients and $0 \le \ell \le k-1$. 
\[ \| u \|_{L^2_{\ell+2} (g_0)} \le C_{k-1} (\| \Delta_{g_n'} u \|_{L^2_\ell (g_0)} + \| u \|_{L^2_\ell (g_0)}), \]
where $C_{k-1}$ is independent of $n=1, 2, \ldots, \infty$ and $g_0$ is a fixed reference smooth metric. 
(Note $L^2_{k-1} \subset C^1$. We used this to the above uniform elliptic estimates. See for instance the proof of the elliptic estimates in the Appendix of \cite{Kod}. Note also Sobolev multiplication works. )
First, the $C^1$-convergence of $\phi_n \to \phi_\infty$ follows by the same argument as before. 
Then we know that $\Delta_{g_n'} \phi_\infty^\alpha = \Gamma^\alpha_{\beta \gamma, n} \partial_i \phi^\beta_n \partial_j \phi^\gamma_n g^{ij}_n$ converges to $\Delta_{g_\infty'} \phi_\infty^\alpha = \Gamma^\alpha_{\beta \gamma} \partial_i \phi^\beta_\infty \partial_j \phi^\gamma_\infty g^{ij}_\infty$ in $L^2$-topology (actually in $C^0$-topology). 
Combined with the $L^2_{k-1}$-convergence of $\Delta_{g_n'} \phi_\infty^\alpha \to \Delta_{g_\infty'} \phi_\infty^\alpha$, we obtain $\| \Delta_{g_n'} (\phi^\alpha_n - \phi^\alpha_\infty) \|_{L^2 (g_0)} \to 0$. 
It follows from the above uniform elliptic estimate that 
\[ \| \phi_n^\alpha - \phi_\infty^\alpha \|_{L^2_2 (g_0)} \le C_{k-1} (\| \Delta_{g_n'} (\phi_n^\alpha - \phi_\infty^\alpha) \|_{L^2 (g_0)} + \| \phi^\alpha_n - \phi^\alpha_\infty \|_{L^2 (g_0)}) \to 0, \]
and we obtain $\phi_n \to \phi_\infty$ in $L^2_2$-topology. 
We can repeat this process until we conclude the $L^2_{k+1}$-convergence of $\phi_n \to \phi_\infty$. 

Now we can prove the following. 

\begin{prop}
The $C^1$-smooth map
\[ \mathcal{H} :B \times^K \Ham_T (M, \omega)^2_{k+1} \to \mathcal{J}_T (M, \omega)^2_k \]
is injective for any sufficiently small neighbourhood $B \subset H^1_T (X, \Theta)$ of the origin. 
\end{prop}

\begin{proof}
The derivative of $\mathcal{H}$ at $[0, \id]$ is given by
\[ \mathbb{H}^1_T \times L^2_{T, k+2} (M)_0/\mathfrak{k} \to \Omega^{0,1}_T (T^{1,0})^2_k : (\rho, f) \mapsto d \mathfrak{J}_0 (\rho) + \bar{\partial} X_f'. \]
It is easy to see that this map is injective and has a closed split range. 
Then the implicit function theorem shows that $\mathcal{H}$ gives an immersion in a neighbourhood of $[0, \id]$. 
In particular, $\mathcal{H}$ is locally injective at $[0, \id]$. 

Suppose $\mathcal{H}$ is not (globally) injective for any sufficiently small $B$. 
Then we can take sequences $b_n, b_n' \to 0 \in B$ and $\phi_n, \phi_n' \in \Ham_T (M, \omega)^2_{k+1}$ satisfying
\[ [b_n, \phi_n] \neq [b_n', \phi_n'] \quad\text{and}\quad \mathcal{H} ([b_n, \phi_n]) = \mathcal{H} ([b_n', \phi_n']). \] 
In particular, we have $\mathfrak{J} (b_n) = (\phi_n' \circ \phi_n^{-1})^* \mathfrak{J} (b_n')$ and both $\mathfrak{J} (b_n), \mathfrak{J} (b_n')$ converge to $\mathfrak{J} (0)$ in $\mathcal{J}_T (M, \omega)^2_k$. 
From the properness of $a^2_k$, we have a subsequence of $\phi_n' \circ \phi_n^{-1}$ which converges to some $\phi_\infty$ in the stabilizer $K$ of $\mathfrak{J} (0)$. 
Hence, after taking a subsequence, both $[b_n, \id]$ and $[b_n', \phi_n' \circ \phi_n^{-1}]$ converge to the same $[0, \id] = [0, \phi_\infty]$ with the same images $\mathcal{H} ([b_n, \id]) = \mathcal{H} ([b_n', \phi_n' \circ \phi_n^{-1}])$. 
Since $\mathcal{H}$ is injective near $[0, \id]$, we conclude $[b_n, \id] = [b_n', \phi_n' \circ \phi_n^{-1}]$ for sufficiently large $n$. 
This contradicts to the choice of the sequences $[b_n, \phi_n] \neq [b_n', \phi_n']$ and we have shown that $\mathcal{H}$ is injective for some (hence any) sufficiently small $B$. 
\end{proof}

The  restriction of the map $\mathfrak{J} :B \to \mathcal{J}_T (M, \omega)^2_k$ gives a continuous map $\nu^{-1} (0) \to (\mathcal{S}^\intg_\xi)^{-1} (0)^2_k$ and induces another continuous map 
\[ \nu^{-1} (0)/K \to (\mathcal{S}^\intg_\xi)^{-1} (0)^2_k /\Ham_T (M, \omega)^2_{k+1}. \]
The following corollaries are essential in the proof of the main theorem. 

\begin{cor}
\label{local-global}
The induced map $\nu^{-1} (0)/K \to (\mathcal{S}^\intg_\xi)^{-1} (0)^2_k /\Ham_T (M, \omega)^2_{k+1}$ is a homeomorphism onto an open subset. 
\end{cor}

\begin{proof}
The injectivity follows from the above Proposition. 
From the Proposition in section 2 of \cite{Kur2}, there is a point $b \in B$ such that $(M, \mathfrak{J} (b)) \cong (M, J)$ for any integrable $L^2_{k}$-regular $J$ sufficiently close to $J_0$ in $L^2_{k}$-topology. 
(Here we can work with $L^2_{l=k}$ rather than $L^2_{l=k+2}$ by taking smaller $B$ if necessary, thanks to the uniqueness of Kuranishi family independent of its construction. )
Furthermore, if $J \in (\mathcal{S}_\xi^\intg)^{-1} (0)^2_k$ we have $(M, \mathfrak{J} (b')) \cong (M, \mathfrak{J} (b))$ for any $b' \in bK^c$ and $b K^c \cap \nu^{-1} (0) \neq \emptyset$, so we can take such $b$ from $\nu^{-1} (0) \subset B$ for any $J \in (\mathcal{S}^\intg_\xi)^{-1} (0)^2_k$. 
Therefore the image of $\nu^{-1} (0)/K$ covers an open neighbourhood of $J_0 \in (\mathcal{S}^\intg_\xi)^{-1} (0)^2_k /\Ham_T (M, \omega)^2_{k+1}$. 
Since $(\mathcal{S}^\intg_\xi)^{-1} (0)^2_k /\Ham_T (M, \omega)^2_{k+1}$ is a Hausdorff space, it follows that the map $\nu^{-1} (0)/K \to (\mathcal{S}^\intg_\xi)^{-1} (0)^2_k /\Ham_T (M, \omega)^2_{k+1}$ becomes a homeomorphism onto an open subset, by taking smaller $B$ if necessary. 
\end{proof}

\begin{cor}
\label{Kempf-Ness}
Suppose the torus action on $(M, \omega)$ is K-optimal. 
The inclusion map $\nu^{-1} (0) \hookrightarrow BK^c$ induces a homeomorphism $\nu^{-1} (0)/K \to BK^c \sslash K^c$. 
\end{cor}

\begin{proof}
The analytic GIT quotient $BK^c \sslash K^c$ is identified as a topological space with the quotient space of $BK^c$ by the equivalence relation $b \sim b' \iff \overline{bK^c} \cap \overline{b'K^c} \neq \emptyset$. 
Take elements $b, b' \in \nu^{-1} (0)$ so that $b \sim b'$. 
Since both $b, b'$ are polystable with respect to the $K^c$-action, their $K^c$-orbits are closed and hence it follows that $bK^c = b'K^c$.
As mentioned before, we obtain $(M, \mathfrak{J} (b)) \cong (M, \mathfrak{J} (b'))$. 
Then it follows from Proposition \ref{isomorphism class} that we get $[\mathfrak{J} (b)] = [\mathfrak{J} (b')] \in (\mathcal{S}^\intg_\xi)^{-1} (0)^2_k /\Ham_T (M, \omega)^2_{k+1}$. 
From the above corollary, we obtain $[b] = [b'] \in \nu^{-1} (0)/K$ and we have shown the map $\nu^{-1} (0)/K \to BK^c \sslash K^c$ is injective. 

We know that $BK^c \sslash K^c$ consists of closed $K^c$-orbits and closed $K^c$-orbit has non-empty intersection with $\nu^{-1} (0)$. 
This shows the surjectivity of $\nu^{-1} (0)/K \to BK^c \sslash K^c$. 
Since both the spaces are locally compact Hausdorff, by taking smaller $B$ if necessary, the map becomes a homeomorphism. 
\end{proof}

\begin{cor}
\label{automorphism}
For any $b \in \nu^{-1} (0)$, $\mathrm{Aut}_T (\mathcal{X}_b)$ can be identified with the complexification of the stabilizer group $K_b$ of the action of $K$. 
\end{cor}

\begin{proof}
From the proof of Theorem \ref{reductive}, we know that $\mathrm{Aut}_T (\mathcal{X}_b) \cong \mathrm{Aut}_T (M, \mathfrak{J} (b))$ is the complexification of the compact group $\mathrm{Stab} (\mathfrak{J} (b)) \subset \Ham_T (M, \omega)^2_{k+1}$. 
Since $\mathfrak{J} :B \to \mathcal{J}_T (M, \omega)^2_{k+2}$ is $K \subset \Ham_T (M, \omega)^2_{k+1}$-equivariant, there is an inclusion $K_b \subset \mathrm{Stab} (\mathfrak{J} (b))$. 
For $\phi \in \mathrm{Stab} (\mathfrak{J} (b))$, we have $\mathcal{H} ([b, \phi]) = \phi^* \mathfrak{J} (b) = \mathfrak{J} (b) = \mathcal{H} ([b, \id])$, then the injectivity of $\mathcal{H}$ shows that $[b, \phi] = [b, \id]$, hence $\phi \in K_b$. 
\end{proof}

\begin{postrem}
Corollary \ref{automorphism} enables us to prove Corollary \ref{expected} under the uniqueness statement of (2) in Conjecture \ref{conj}, which we do not follow in this paper. 
So the last corollary will be never used in any proofs of this paper. 
Recently, R. Dervan and P. Naumann find an another pure analytic approach to construct the moduli space of cscK manifolds that makes use of this corollary. 
\end{postrem}

In the next section, we will construct complex structures on the following Hausdorff topological spaces.  

\begin{defin}
\label{support of moduli}
We set
\begin{align*} 
L^2_k K (M, \omega, T) 
&:= (\mathcal{S}_\xi^\intg)^{-1} (0)^2_k /\Ham_T (M, \omega)^2_{k+1}, 
\\
K (M, \omega, T) 
&:= (\mathcal{S}_\xi^\intg)^{-1} (0)/\Ham_T (M, \omega) 
\end{align*}
and
\[ K_{T, \chi} := \coprod_{\chi (M, \omega, T) = \chi} K (M, \omega, T) \]
for a Fano character $\chi = \{ \chi_m \in \mathbb{Z} [M] \}_{m \in \mathbb{Z}}$ where $(M, \omega, T)$ runs K-optimal symplectic Fano $T$-manifolds whose Hilbert character $\chi (M, \omega, T)$ (see the description before Proposition \ref{isomorphism class}) is equal to the given Fano character $\chi$.  
\end{defin}

Note that the homeomorphism $f^* :K (M', \omega', T) \xrightarrow{\sim} K (M, \omega, T)$ induced by a $T$-equivariant symplectic diffeomorphism $f :(M, \omega) \xrightarrow{\sim} (M', \omega')$ is independent of the choice of $f$, so the space $K_{T, \chi}$ is free from the choice of the representatives $(M, \omega, T)$ in the symplectic diffeomorphism class $[M, \omega, T]$. 

\section{Canonical complex structure}

\subsection{The moduli stack $\K_{T, \chi}$}
\label{The moduli stack}

In this subsection, we prepare some standard terminologies around stack and introduce the moduli stack $\K_{T, \chi}$ and see its Artinianity (Definition \ref{Artin def}). 
Though it is a simple task to check the Artinianity under Corollary \ref{gentle is open condition}, the author believes that knowing its proof must help the readers to properly handle the moduli stack $\K_{T, \chi}$ in the next subsection. 
See Appendix A for generalities on stacks over the category of complex spaces, which we call $\Cpx$-stacks. 

We denote by $\Cpx$ the category of complex analytic spaces which are not assumed to be reduced nor irreducible. 
The set of holomorphic morphisms between complex spaces $U$ and $V$ will be denoted by $\Holo (U, V)$. 
We also denote by $\Cpx_S$ the category of complex spaces over $S$ and by $\Holo_S (U, V)$ its set of morphisms. 

Let $S$ be a complex space. 
A morphism of complex spaces $\pi :\M \to S$ is called a \textit{family of complex manifolds over $S$} if it is surjective, proper, smooth (equivalent to submersive when $\M$ and $S$ are complex manifolds) and has connected fibers. 
Recall that a smooth morphism between singular complex spaces is by definition a morphism of complex spaces $f :X \to Y$ with the following property: 
There are open coverings $\{ V_\alpha \subset Y \}_\alpha$, $\{ U_\alpha \subset X \}_\alpha$ of $Y$ and $X$, respectively, an indexed set of smooth complex manifolds $\{ W_\alpha \}_\alpha$ and an indexed set of biholomorphisms $\{ \phi_\alpha : V_\alpha \times W_\alpha \xrightarrow{\sim} U_\alpha \}_\alpha$ satisfying $f \circ \phi_\alpha = p_1$, where the morphism $p_1$ denotes the projection $V_\alpha \times W_\alpha \to V_\alpha$. 

Let $T \cong (\mathbb{C}^*)^k$ be an algebraic torus. 
A \textit{fibrewise $T$-action on a family $\pi :\M \to S$} is a holomorphic morphism $\alpha :\M \times T \to \M$ which satisfies the following conditions. 
\begin{enumerate}
\item (Fibrewise) The morphism $\alpha$ is an $S$-morphism. Namely, $\pi \circ \alpha = \pi \circ p_1 :\M \times T \to S$. 

\item (Group action) $\alpha \circ (\alpha \times \id_T) = \alpha \circ (\id_\M \times \mu) :\M \times T \times T \to \M$, where $\mu :T \times T \to T$ is the multiplication. 
\end{enumerate}

% The base change of the diagonal morphism $\Delta :\M \to \M \times_S \M$ by the morphism $p_1 \times \alpha :\M \times T \to \M \times_S \M$ is denoted by $\mathrm{Stab}_\M \to \M$. 
% The fiber of $\mathrm{Stab}_\M \to \M$ at $x \in \M$ is nothing but the stabilizer of $x$ with respect to the $T$-action. 
% An action of $T$ on an $S$-family $\pi :\M \to S$ is called \textit{free} if the stabilizer morphism $\mathrm{Stab}_\M \to \M$ is an biholomorphism of complex spaces. 

A fibrewise $T$-action on a family $\pi :\M \to S$ is called \textit{effective} if for every $s \in S$ the induced group morphism $T \to \mathrm{Aut} (\M_s)$ is injective. 
Finally, an \textit{$S$-family of complex $T$-manifolds} is defined to be a family of complex manifolds over $S$ together with an effective fibrewise $T$-action in the above sense. 

Now we introduce the stack $\K_{T, \chi}$. 
A $\Cpx$-stack (Definition \ref{cpx-stack}) is a category $\bm{\F}$ together with a functor $\bm{\F} \to \Cpx$ satisfying some natural geometric axioms. 
Here we give the category of our interest. 

\begin{defin}[category/stack $\K_{T, \chi}$]
\label{moduli stack}
Let $T$ be an algebraic torus and $\chi$ be a Fano character. 
Object in $\K_{T, \chi}$ is a family of complex $T$-manifolds $\pi :\M \to S$ whose fibers are \textit{gentle} (see Definition \ref{gentle Fano}) Fano $T$-manifolds whose Hilbert characters are $\chi$. 

A morphism from $\xi := (\pi :\M \to S, \alpha :\M \times T \to \M)$ to $\xi' := (\pi' :\M' \to S', \alpha' :\M' \times T \to \M')$ is a pair $(f, \phi)$ where $f :S \to S'$ is a morphism of complex spaces and $\phi :\M \to \M'$ is a $T$-equivariant morphism which is compatible with $\pi, \pi', f$ and induces a biholomorphism $\tilde{\phi} : \M \to f^* \M'$, where $f^* \M' := S \times_{f, S', \pi'} \M' \subset S \times \M'$. 
Here, the morphism $\phi$ is said to be $T$-equivariant if $\alpha' \circ (\phi \times \id_T) = \phi \circ \alpha :\M \times T \to \M'$. 
\end{defin}

Note that there may be no gentle Fano manifolds whose K-optimal Hilbert characters coincide with a given K-optimal Fano character $(T, \chi)$. 
In other words, there might be no object in the category $\K_{T, \chi}$ for a K-optimal Fano character $(T, \chi)$. 

We denote by $\K_{T, \chi}^s$ the subcategory of $\K_{T, \chi}$ consisting of families of K-stable Fano $T$-manifolds and by $\K (n)$ the disjoint union of the categories $\K_{T, \chi}$ where $(T, \chi)$ run all the K-optimal Fano characters of $n$-dimensional Fano manifolds. 
Both categories $\K_{T, \chi}$ and $\K^s_{T, \chi}$ are $\Cpx$-stacks with the forgetful functors $\K_{T, \chi}, \K^s_{T, \chi} \to \Cpx$ given by $(\pi: \mathcal{M} \to S) \mapsto S$. 
(See Lemma \ref{moduli category} and \ref{moduli can stack} in Appendix A. )

We denote by $\K_{T, \chi} (S)$ the subcategory consisting of objects $(\pi :\M \to S, \alpha)$ with fixed base $S$ and whose morphisms are pairs $(\id_S, \phi)$. 
For any two objects $\xi = (\M \to S, \alpha), \eta = (\M' \to S, \alpha') \in \K_{T, \chi} (S)$, we define the contravariant functor $\mathit{Isom}_S (\xi, \eta)$ from $\Cpx_S$ to $\Sets$ by mapping an object $f: U \to S$ to the set $\Hom_{\K_{T, \chi}} (f^* \xi, f^* \eta)$ and a morphism $h: (U, f) \to (V, g)$ to the map given by the canonical identifications $f^* \xi \cong h^* g^* \xi$, $f^* \eta \cong h^* g^* \eta$. 

The following definition is just an analogy of Artin algebraic stack. 

\begin{defin}[Artin $\Cpx$-stack]
\label{Artin def}
A $\Cpx$-stack $\mathbfcal{F}$ is called an \textit{Artin stack} if it satisfies the following two conditions. 
\begin{enumerate}
\item The diagonal morphism $\Delta :\mathbfcal{F} \to \mathbfcal{F} \times \mathbfcal{F}$ is representable by complex spaces. 

\item There exists a smooth surjective morphism $U \to \mathbfcal{F}$ from a complex space $U$. 
\end{enumerate}
Or equivalently, 
\begin{enumerate}
\item For every complex space $S$ and any $\xi, \eta \in \Obj (\mathbfcal{F})$, there exists a complex space $S_{\xi, \eta}$ and an isomorphism $\Holo_S (-, S_{\xi, \eta}) \cong \mathit{Isom}_S (\xi, \eta)$ of contravariant functors from $\Cpx_S$ to $\Sets$. 

\item There exists a morphism $U \to \mathbfcal{F}$ of fibred categories from a complex space $U$ such that the 2-fibre product $U \times_{\mathbfcal{F}} V$ is isomorphic as stacks to a complex space $f: V_U \to V$ smooth over $V$ with surjective $f$, for any morphism $V \to \mathbfcal{F}$ from any complex space $V$. 
\end{enumerate}
\end{defin}

The $2$-fibre product $U \times_{\mathbfcal{F}} V$ of stacks is always isomorphic to some complex space $W$ over $V$ and $U$ from the first condition (cf. \cite[Tag 045G]{SPA}). 

We frequently use the following representability result in our analytic category in the rest of this paper. 
The representability of the following moduli functor (and analogical functor in the schematic category) in both analytic/schematic category is well-known and the projectivity is also well-known in the schematic category. 
While the compactness of the Douady space is proved by Fujiki for class $\mathcal{C}$ space, it seems we must see the equivalence of the Douady space and the Hibert scheme, or must imitate the construction of Hilbert scheme in the analytic category, to show the projectivity of the Douady space when $X$ is projective. 
Our concerns here are just whether the analytification of the Hilbert scheme represents the moduli functor of the Douady space, and if we can make things $T$-equivariant. 
Though these concerns might be exhibited somewhere in literatures, we place a proof here since the author could not find an appropriate reference and the author believes it is better for the readers (including the author). 

\begin{prop}[$T$-Hilbert scheme]
Suppose $\pi: \mathcal{X} \to B$ is a holomorphic morphism of complex spaces and $\alpha :\mathcal{X} \times T \to \mathcal{X}$ is a holomorphic action with $f \circ \alpha = f \circ p_1$. 
Consider the functor $\mathit{Hilb}_{T, \pi} :\Cpx_B \to \Sets$ given by 
\[ \mathit{Hilb}_{T, \pi} (S) := \Big{\{} \mathcal{Z} \subset S \times_B \mathcal{X} ~|~ \begin{matrix} \mathcal{Z} \text{ is a  $T$-invariant closed analytic subspace } \\ \text{ and } \mathcal{Z} \to S \text{ is a flat family } \end{matrix} \Big{\}}. \]
Then there exists a Hausdorff complex space $\mathrm{Hilb}_{T, \pi}$ representing the functor $\mathit{Hilb}_{T, \pi}$. 
Moreover, suppose $B = \mathrm{pt}$ and $X = \mathcal{X}$ is projective with an ample line bundle $L$, then the subfunctor $\mathit{Hilb}_{T, \chi, X} \subset \mathit{Hilb}_{T, X}$ consisting of families $\mathcal{Z} \to S$ with a fixed Hilbert polynomial $\chi$ is representable by a projective complex space $\mathrm{Hilb}_{T, \chi, X} \subset \mathbb{C}P^N$. 
\end{prop}

\begin{proof}
When $T$ is trivial, the existence of the Hausdorff complex space $\mathrm{Hilb}_\pi = \mathrm{Hilb}_{T, \pi}$ follows from \cite{Dou2} for $B = \mathrm{pt}$ case and from \cite{Pou} for general base $B$.  
The projectivity follows from Grothendieck's existence theorem of the Hilbert scheme, which represents an analogical functor defined on the category of schemes $\mathbf{Sch}_\mathbb{C}$, and the coincidence of the functors when they restricted to the subcategory $\mathbf{Def}_\mathbb{C}$ of the spectrum of finitely generated Artin algebras over $\mathbb{C}$, which is naturally embedded into both $\Cpx$ and $\mathbf{Sch}_\mathbb{C}$. 
Actually, a morphism $f: X \to Y$ between complex spaces is an isomorphism if and only if it induces an isomorphism of functors $h_X|_{\mathbf{Def}_\mathbb{C}} \to h_Y|_{\mathbf{Def}_\mathbb{C}}$. 

When $T$ is non-trivial, we can consider the action of $T$ on the set $\mathit{Hilb}_\pi (S)$ for each $S \in \Cpx_B$, whose fixed point subset is nothing but the subset $\mathit{Hilb}_{T, \pi} (S) \subset \mathit{Hilb}_\pi (S)$. 
Then the existence in the category of complex spaces follows from the following two general statements. 
\begin{enumerate}
\item Suppose $H$ is a (not necessarily reduced) complex space with $T$-action and $x \in H$ is a $T$-fixed point. Then there is a $T_\mathbb{R}$-invariant open neighbourhood $U \subset H$ of $x$ and a $T_\mathbb{R}$-equivariant closed embedding $\varphi :U \to V$ into an open neighbourhood $V \subset T_x H$ of the origin, where $T_\mathbb{R}$ denotes the maximal compact subgroup of $T$ and $T_x H$ denotes the Zariski tangent space (\cite[Subsection 2.2]{Akh}). 

\item Let $W \subset T_x H$ be the set of $T$-invariant points, which forms a $T$-invariant linear subspace. Then the complex space $U_T := U \times_V (W \cap V) \subset U$, considered as a closed subspace of $U$, enjoys the following universal property: for any holomorphic morphism $f: S \to H$ invariant under the $T$-action on $H$, the restricted holomorphic morphism $f|_{f^{-1} (U)} :f^{-1} (U) \to H$ holomorphically and uniquely factors through $U_T$. 
\end{enumerate}

On the other hand, the existence in the category of schemes follows from \cite{Fog2}. 
The rest of the proof is parallel to the first paragraph. 
\end{proof}

The proof of the next proposition is a routine for the readers familiar with Artin stack. 
We exhibit the proof for the others. 

\begin{prop}
\label{Artin stack}
The $\Cpx$-stacks $\K_{T, \chi}$ is an Artin $\Cpx$-stack. 
If $(T, \chi)$ is K-optimal, then $\K_{T, \chi}^s$ is also Artin and is an open sub-stack of $\K_{T, \chi}$. 
\end{prop}

\begin{proof}
By considering the graphs $\mathcal{M} \times_{\phi, \mathcal{N}, \id} \mathcal{N} \subset \mathcal{M} \times \mathcal{N}$ of morphisms $\phi :\mathcal{M} \to \mathcal{N}$, the functor $\mathit{Isom}_S (\xi, \eta)$ is identified with a subfunctor of $\mathit{Hilb}_{T, \M \times \mathcal{N}/S}$. 
Then it is easy to see that $\mathit{Isom}_S (\xi, \eta)$ is representable by an open subspace of $\mathrm{Hilb}_{T, \M \times \mathcal{N}/S}$ (cf. \cite[5.6.2.]{FGA}). 

Next we construct a smooth surjective morphism $U \to \K_{T, \chi}$. 
Let us consider a uniform $T$-equivariant embedding of Fano manifolds in $\K_{T, \chi}$ into some fixed $\mathbb{C}P^N$. 
Then all Fano manifolds in $\K_{T, \chi}$ emerge in $\mathrm{Hilb}_{T, \mathbb{C}P^N}$. 
From Corollary \ref{gentle is open condition}, there is an open subset $U$ of $\mathrm{Hilb}_{T, \mathbb{C}P^N}$, in the real topology, such that the restricted universal family $\mathcal{U}|_U \to U$ exactly consists of gentle Fano manifolds in $\K_{T, \chi}$. 
Note the family $\mathcal{U}|_U \to U$ naturally carries a $T$-action and hence is considered as an object in $\K_{T, \chi}$. 
So we have the induced morphism $U \to \K_{T, \chi} :(f :S \to U) \mapsto (f^* \mathcal{U}|_U \to S)$, which is readily surjective from its construction. 
Consider a morphism $X \to \K_{T, \chi}$; it is equivalent to give a family of gentle Fano $T$-manifolds $\M \to X$. 
For a sufficiently large $m$, the direct image sheaf $\pi_* (\mathcal{O} (-m K_\M))|_{U_\alpha}$ becomes locally free. 
Take a covering $\U = \{ U_\alpha \}_\alpha$ of $X$ that trivializes the vector bundle $\pi_* (\mathcal{O} (-m K_\M))|_{U_\alpha}$ so that we can consider morphisms $U_\alpha \to U$ corresponding to trivializations of $\pi_* (\mathcal{O} (-m K_\M))|_{U_\alpha}$. 
There is a unique $PGL_T$-equivariant extension $U_\alpha \times PGL_T \to U$ of these morphisms. 
Then from the universality of the $2$-fibre product $U \times_{\K_{T, \chi}} U_\alpha$ , we get morphisms $U_\alpha \times PGL_T \to U \times_{\K_{T, \chi}} U_\alpha$. 
We have the inverse morphisms of these morphisms given as follows. 
Take an object $(S, \xi: S \to U, \eta: S \to U_\alpha, \phi: \xi^* \mathcal{U} \xrightarrow{\sim} \eta^* \M|_{U_\alpha})$ of $U \times_{\K_{T, \chi}} U_\alpha$. 
Since $\eta^* \M|_{U_\alpha}$ can be considered as being embedded in $S \times \mathbb{C}P^N$, the isomorphism $\phi$ corresponds to a morphism $\tilde{\phi}: S \to PGL_T$. 
Then we have a morphism $\eta \times \tilde{\phi}: S \to U_\alpha \times PGL_T$, which gives an object in $\Cpx_{U_\alpha \times PGL_T}$. 
Therefore $U \times_{\K_{T, \chi}} X \to X$ is locally written as $U_\alpha \times PGL_T \to U_\alpha$. 
So it is a smooth morphism. 

It is shown in \cite[Theorem 3.4]{H. Li} that K-stable Fano manifolds form an open subset in the parameter space of any family of complex manifolds, without introducing the K-stability of Fano $T$-manifolds. 
From the exactly same argument as above, we conclude that $\K^s_{T, \chi}$ is Artin and is an open sub-stack of $\K_{T, \chi}$. 
\end{proof}

In the above proof, the only non-routine part is Corollary \ref{gentle is open condition}, i.e. the openness of the subset consisting of gentle Fano manifolds in the parameter space of a family. 
Our method in Corollary \ref{gentle is open condition} cannot be applied to prove the Zariski openness. 
This is the reason why we must work in the category of complex spaces rather than the category of algebraic spaces, so that we cannot apply Alper's theory on good moduli spaces over the category of algebraic spaces, at least so far. 

\subsection{Main construction}

In this subsection, we prove our main theorem. 
First we prepare two general lemmas.

\begin{lem}
\label{lemma}
Let $K$ be a compact Lie group and $K^c$ be its complexification, $V$ be a representation of $K^c$ and $B \subset V$ be a $K$-invariant Stein open neighbourhood of the origin. 
Let $s \times t: R \to B \times B$ be the holomorphic groupoid obtained by pulling back the holomorphic action groupoid $a: V \times K^c \to V \times V :(v, k) \mapsto (v, vk)$ along the inclusion $B \times B \subset V \times V$.  
Then the following holds. 
\begin{enumerate}
\item The stack $[B/R]$ associated to the holomorphic groupoid $(B, R, s, t, c)$ as in Appendix A is isomorphic, induced by the inclusion $B \hookrightarrow B K^c$, to the quotient stack $[BK^c/K^c]$ as $\Cpx$-stacks, where $BK^c$ denotes the $K^c$-orbit $\{ bg \in V ~|~ b \in B, g \in K^c \}$ of $B$, which is $K^c$-invariant open. 

\item There is a natural morphism $[BK^c/K^c] \to BK^c \sslash K^c$ of $\Cpx$-stacks to the analytic GIT quotient $BK^c \sslash K^c$ which enjoys the following universal property: any morphism from the quotient stack $[BK^c/K^c]$ to any complex space $X$ uniquely factors through $BK^c \sslash K^c$. 
\end{enumerate}
\end{lem}

\begin{proof}
We identify $[BK^c /K^c]$ with the $\Cpx$-stack $\llbracket BK^c /K^c \rrbracket$ in Example \ref{quotient stack}. 
Consider a morphism $\sigma$ from the fibred category $[B/R]_p$ to $[BK^c/K^c]$ sending an object $\xi: S \to X$ in $[B/R]_p$ to the object $(S, S \times K^c, a \circ (\xi \times \id))$ in $[BK^c/K^c]$ (cf. the description right after Example \ref{quotient stack}). 
Let $S$ be a complex space and $(S, P, \xi')$ be an object in $[BK^c/K^c] (S)$. 
Take a local trivialization $\{ P \cong U_\alpha \times K^c \}_\alpha$ of the principal $K^c$-bundle $P$ and consider the associated $K^c$-equivariant morphisms $\xi'_\alpha: U_\alpha \times K^c \to BK^c$. 
After taking smaller $U_\alpha$, we can find a holomorphic morphism $\xi_\alpha :U_\alpha \to B$ and a holomorphic morphism $g :U_\alpha \to K^c$ so that $\xi_\alpha (x) g (x) = \xi'_\alpha (x, e)$. 
It follows that the object $(U_\alpha, U_\alpha \times K^c, \xi'_\alpha)$ in $[BK^c/K^c]$ is isomorphic to $\sigma (\xi_\alpha) = (U_\alpha, U_\alpha \times K^c, a \circ ((\xi_\alpha g) \times \id_{K^c}))$. 
Moreover, it is easily seen that $\mathit{Isom}_{[B/R]_p, S} (\xi, \eta) \to \mathit{Isom}_{[BK^c/K^c], S} (\sigma (\xi), \sigma (\eta))$ is a sheafification of the functor $\mathit{Isom}_{[B/R]_p, S} (\xi, \eta): \Cpx_S \to \Sets$. 
It follows that $[BK^c/K^c]$ is a stackification of the fibred category $[B/R]_p$. 
Therefore, it is isomorphic to the stackification $[B/R]$ of $[B/R]_p$. 

Since $B$ is a reduced Stein space, $BK^c$ is also a reduced Stein space and there exists a categorical quotient $BK^c \sslash K^c$, which is also a reduced Stein space (see \cite{Heinzner}, \cite{Snow}). 
Take an object $(S, P, \xi)$ in $[BK^c/K^c]$ and a local trivialization $\{ P \cong U_\alpha \times K^c \}_\alpha$ of $P$. 
Then we have holomorphic morphisms $\tilde{\xi}_\alpha :U_\alpha \to U_\alpha \times K^c \to BK^c \to BK^c \sslash K^c$. 
Since $\xi_\alpha :U_\alpha \times K^c \to BK^c$ agree on the overlaps $U_\alpha \cap U_\beta$ up to the action of $K^c$, and $BK^c \to BK^c \sslash K^c$ is $K^c$-invariant, holomorphic morphisms $\tilde{\xi}_\alpha$ coincide on the overlaps $U_\alpha \cap U_\beta$ and define a holomorphic morphism $S \to BK^c \sslash K^c$, glued together. 
This construction gives the morphism $[BK^c/K^c] \to BK^c \sslash K^c$. 
The universal property follows from the fact that any $K^c$-invariant holomorphic morphism $BK^c \to X$ uniquely factors through $BK^c \sslash K^c$. 
\end{proof}

The complex space $BK^c \sslash K^c$ is moreover normal as it is an open subspace of the algebraic GIT quotient $V \sslash K^c$, which is normal whenever $V$ is normal (cf. \cite[Section 0.2]{MFK}). 

\begin{lem}
\label{equivariant trivialization}
Let $K$ be a compact Lie group, $B$ be a complex manifold with holomorphic $K$-action and $E \to B$ be a $K$-equivariant holomorphic vector bundle. 
Suppose $0 \in B$ is a fixed point of $K$-action. 
Since the fiber $E_0$ can be considered as $K$-representation, we can construct a $K$-equivariant holomorphic vector bundle $\underline{E_0}_{B,K} := B \times E_0$ whose action is given by $(b, v)k := (bk, vk)$. 
Then $E$ is $K$-equivariantly isomorphic to $\underline{E_0}_{B, K}$ on some neighbourhood of $0 \in B$. 
\end{lem}

\begin{proof}
Consider the frame bundle $\pi :P \to B$ of the holomorphic vector bundle $E$ and fix a point $p_0 \in \pi^{-1} (0) \subset P$. 
We have a right holomorphic action of $K$ on $P$ defined by 
\[ pk :\mathbb{C}^r \xrightarrow{p_0} E_0 \xrightarrow{k^{-1}} E_0 \xrightarrow{p_0^{-1}} \mathbb{C}^r \xrightarrow{p} E_b \xrightarrow{k} E_{bk} \]
for $p :\mathbb{C}^r \xrightarrow{\sim} E_b \in P$ and $k \in K$. 
The point $p_0 \in P$ is a fixed point of this action and $\pi :P \to B$ is a $K$-equivariant submersion. 
So we have a $K$-equivariant holomorphic section $\sigma :B \to P$ with $\sigma (0) = p_0$ by taking smaller $B$ if necessary. 
Now the map $B \times GL (r) \to P :(b, g) \mapsto \sigma (b) g$ gives a $K$-equivariant isomorphism of principal $GL (r)$-bundles and hence induces a $K$-equivariant isomorphism of the adjoint bundles $\underline{\mathbb{C}^r}_{B, K} \xrightarrow{\tilde{p}_0} \underline{E_0}_{B, K}$. 
\end{proof}

Let $X$ be a Fano manifold with a K\"ahler--Ricci soliton $(g, \xi')$, $T \cong (\mathbb{C}^*)^{\times k}$ be the algebraic torus generated by $\xi'$, $K := \mathrm{Isom}_{\xi'} (X, g)$ be the (possibly non-connected) isometry group preserving $\xi'$ and $H^1_T (X, \Theta) \subset H^1 (X, \Theta)$ denote the $T$-invariant subspace of the first cohomology of the tangent sheaf. 
Recall in Proposition \ref{local slice} and in Corollary \ref{gentle is open condition} we obtained a family $\varpi: \mathcal{X} \to B$ of gentle Fano $T$-manifolds over a small ball $B \subset H^1_T (X, \Theta)$. 
Moreover, $\mathcal{X}$ admits a holomorphic $K$-action so that $\varpi$ is $K$-equivariant with respect to the linear action on $B$, and a $T$-equivariant biholomorphism $\mathcal{X}_0 \cong X$. 
This in particular defines a morphism $B \to \K_{T, \chi}$ of Artin $\Cpx$-stacks. 
Now we prove the following. 

\begin{prop}
\label{etale}
Let $X$ be a Fano $T$-manifold with K\"ahler--Ricci soliton and the Hilbert character $(T, \chi)$. 
Then by taking smaller $B$ if necessary, the morphism $B \to \K_{T, \chi}$ factors through an \'etale morphism $[B/R] \to \K_{T, \chi}$ with finite fibres, where $R$ is defined as in Lemma \ref{lemma}. 
In other words, for any morphism $S \to \K_{T, \chi}$ of $\Cpx$-stacks, there is an \'etale morphism $S' \to S$ of complex spaces with finite fibres and an $S$-isomorphism of $\Cpx$-stacks from $S'$ to the 2-base change $S \times_{\K_{T, \chi}} [B/R] \to S$. 
\end{prop}

\begin{proof}
The family $\varpi :\mathcal{X} \to B$ in Proposition \ref{local slice} defines a morphism $B \to \K_{T, \chi}$. 
Now we will show that this morphism factors through the quotient morphism $B \to [B/R]$. 
It is equivalent to the existence of a natural $T$-equivariant $R$-biholomorphism $s^* \mathcal{X} \xrightarrow{\sim} t^* \mathcal{X}$. 
We prove this by relating our analytic family to an algebraic family \textit{as groupoids}. 
As a consequence, the induced morphism $[B/R] \to \K_{T, \chi}$ is shown to be \'etale with finite fibres. 

Since $\varpi :\mathcal{X} \to B$ is a $K$-equivariant family and $\mathcal{O} (-K_{\mathcal{X}/B})$ is relatively ample, we can find a large $\ell \in \mathbb{N}$ so that the direct image sheaf $\varpi_* \mathcal{O} (-\ell K_{\mathcal{X}/B})$ is $K$-equivariantly isomorphic to the sheaf of sections of a $K$-equivariant holomorphic vector bundle $E$. 
Lemma \ref{equivariant trivialization} shows that there is a $K$-equivariant isomorphism $\underline{H^0 (X, \mathcal{O} (-\ell K_X))}_{B, K} \cong E$, so we have a $K$-equivariant $B$-embedding $\mathcal{X} \hookrightarrow B \times \mathbb{P}^N$, where we identify $\mathbb{P}^N$ with $\mathbb{P} (H^0 (X, \mathcal{O} (-\ell K_X))^*)$. 
From the universality of $\mathrm{Hilb}_{T, \mathbb{P}^N}$, we obtain a $K$-equivariant holomorphic morphism $h :B \to \mathrm{Hilb}_{T,\mathbb{P}^N}$ together with an isomorphism $h^* \mathcal{U} \cong \mathcal{X}$. 

From the Euler sequence
\[ 0 \to \mathcal{O}_{\mathbb{P}^N} \to \mathcal{O} (1)^{\oplus N+1} \to \Theta_{\mathbb{P}^N} \to 0, \]
we obtain $H^1_T (X, i^* \Theta_{\mathbb{P}^N}) = 0$ and $H^0_T (X, i^* \Theta_{\mathbb{P}^N}) \cong H^0_T (\mathbb{P}^N, \Theta_{\mathbb{P}^N})$. 
Combining with this with the following exact sequence 
\begin{align*}
0 & \to H^0_T (X, \Theta_X) \to H^0_T (X, i^* \Theta_{\mathbb{P}^N}) \to H^0_T (X, N_{X/\mathbb{P}^N}) 
\\ & \to H^1_T (X, \Theta_X) \to H^1_T (X, i^* \Theta_{\mathbb{P}^N}) \to H^1_T (X, N_{X/\mathbb{P}^N}) \to 0
\end{align*}
shows that the sequence 
\[ 0 \to H^0_T (X, \Theta_X) \to H^0_T (\mathbb{P}^N, \Theta_{\mathbb{P}^N}) \to H^0_T (X, N_{X/\mathbb{P}^N}) \to H^1_T (X, \Theta_X) \to 0 \]
is exact and $H^1_T (X, N_{X/\mathbb{P}^N})$ vanishes. 
So we conclude that $\mathrm{Hilb}_{T, \mathbb{P}^N}$ is smooth at $[X] = h (0) \in \mathrm{Hilb}_{T, \mathbb{P}^N}$, whose tangent space is given by $H^0_T (X, N_{X/\mathbb{P}^N})$ (cf. \cite[subsection 6.4.]{FGA}). 

Now we work in the category of algebraic spaces in the blink of an eye. 
Since $\mathrm{Aut}_T (X)$ is reductive, we can apply the \'etale slice theorem \cite[Theorem 2.1]{AHR}, which generalizes the Luna's \'etale slice theorem to non-affine cases, and then obtain the following: a smooth affine $\mathrm{Aut}_T (X)$-variety $(W, w)$, an $\mathrm{Aut}_T (X)$-equivariant morphism $\phi :(W, w) \to (\mathrm{Hilb}_{T, \mathbb{P}^N}, h (0))$ which induces a $PGL_T (N+1)$-equivariant \'etale morphism $W \times^{\mathrm{Aut}_T (X)} PGL_T \to \mathrm{Hilb}_T$, and a $\mathrm{Aut}_T (X)$-equivariant \'etale morphism $(W, w) \to (H^1_T (X, \Theta), 0)$. 
\[ \begin{tikzcd} W \times^{\mathrm{Aut}_T (X)} PGL_T \ar{d}[swap]{\text{\'etale}} \ar[hookleftarrow]{r} & (W, w) \ar{dr}{\text{\'etale}} \ar{dl}[swap]{\phi} & ~ \\ (\mathrm{Hilb}_T, h(0)) & ~ & (H^1_T (X, \Theta_X), 0) \end{tikzcd} \]
Note that the quotient morphism $W \times PGL_T \to W \times^{K^c} PGL_T$ is a $K^c$-equivariant submersion, under the right action of $K^c = \mathrm{Aut}_T (X)$ on $W \times PGL_T$ defined by $(x, g_0)g_1 = (xg_1, g_1^{-1} g_0 g_1)$ and on $W \times^{K^c} PGL_T$ defined by $[x, g_0]g_1 = [x, g_0 g_1]$.
Since the point $(w, e) \in W \times PGL_T$ is fixed by this $K^c$-action, we obtain a $K$-equivariant holomorphic section $\sigma$ from a neighbourhood of $[w, e] \in W \times^{K^c} PGL_T$ with $\sigma ([w, e]) = (w,e)$. 
Therefore, taking smaller $B$ if necessary, we can assume that $h :B \to \mathrm{Hilb}_T$ factors through $W \times PGL_T \to \mathrm{Hilb}_T$. 
We can moreover assume that the composed morphism $(B, 0) \to (W, w)$ of a lifting $B \to W \times PGL_T$ passing through $(w, e)$ with the projection to the first factor is $K$-equivariant holomorphic open embedding. 
Note that we do not know whether this morphism $(B, 0) \to (W, w)$ is a section of the \'etale morphism $(W, w) \to (H^1_T (X, \Theta_X), 0)$. 

Set $\mathrm{Hilb}^\circ_T := h (B) \cdot PGL_T \subset \mathrm{Hilb}_T$. 
Since $\varpi :\mathcal{X} \to B$ is a complete family at any point $b \in B$, $\mathrm{Hilb}^\circ_T$ is an open subset. 
The restricted $PGL_T$-equivariant universal family $\mathcal{U}^\circ \to \mathrm{Hilb}_T^\circ$ parametrizes only gentle Fano $T$-manifolds and hence induces an open embedding $[\mathrm{Hilb}^\circ_T/PGL_T] \to \K_{T, \chi}$. 
We fix this subset $\mathrm{Hilb}^\circ$ while we later take smaller $B$. 

It follows from \cite[Proposition 5.1]{Snow} that we have a $K^c$-invariant open neighbourhood $W^\circ \subset \phi^{-1} (\mathrm{Hilb}_T^\circ)$ of $w$ so that the restriction $W^\circ \to H^1_T (X, \Theta)$ is an $K^c$-invariant open embedding. 
Taking smaller $B$, we have the restricted morphism $B \to W^\circ \times PGL_T$. 
Let $g :B \to PGL_T$ be the composition of this morphism with the projection to the second factor. 
Denote by $h_\circ :B \to \mathrm{Hilb}^\circ$ the composition $\alpha_{\mathrm{Hilb}_T} \circ (h \times g^{-1}) :B \to \mathrm{Hilb}^\circ_T \times PGL_T \to \mathrm{Hilb}^\circ_T$. 
Then the holomorphic morphism $h_\circ$ is $K$-equivariant and factors through the $K^c$-equivariant holomorphic morphism $W^\circ \to \mathrm{Hilb}^\circ_T$. 
Moreover, we have an induced isomorphism $h_\circ^* \mathcal{U} \cong \mathcal{X}$. 

Since the differential of the induced morphism $B \to W^\circ$ at $0 \in B$ is a $K$-equivariant isomorphism, we can assume that $B \to W^\circ$ is a $K$-equivariant open embedding. 
Let us denote by $\beta :B \to H^1_T (X, \Theta)$ the composition of this morphism $B \to W^\circ$ with $W^\circ \to H^1_T (X, \Theta)$. 
Then $\beta$ is also a $K$-equivariant open embedding. 

Note that both $BK^c \subset H^1_T (X, \Theta)$ and $\beta (B) K^c \subset H^1_T (X, \Theta)$ are the complexification, in the sense of \cite{Heinzner}, of $B$ with respect to the action of $K$ . 
From the uniqueness of the complexification, there is a $K^c$-equivariant biholomorphism $\gamma :B K^c \to \beta (B)K^c$ which is compatible with the $K$-equivariant morphisms $B \subset BK^c$ and $\beta :B \to \beta (B) K^c$. 

Now we have the following cartesian diagrams
\[ \begin{tikzcd} R \ar{r} \ar{d}[swap]{s \times t} & B K^c \times K^c \ar{r}{\hookrightarrow \boxtimes \id} \ar{d}{\alpha_{BK^c}} & H^1_T \times K^c \ar{d}{\alpha_{H^1_T}} \\ B \times B \ar{r}{\hookrightarrow^{\boxtimes 2}} & BK^c \times BK^c \ar{r}{\hookrightarrow^{\boxtimes 2}} & H^1_T \times H^1_T \end{tikzcd} \]
\[ \begin{tikzcd} R_W \ar{r} \ar{d}[swap]{s_W \times t_W} & \beta(B) K^c \times K^c \ar{r}{\hookrightarrow \boxtimes \id} \ar{d}{\alpha_{\beta(B)K^c}} & H^1_T \times K^c \ar{d}{\alpha_{H^1_T}} \\ B \times B \ar{r}{\beta^{\boxtimes 2}} & \beta(B)K^c \times \beta(B)K^c \ar{r}{\hookrightarrow^{\boxtimes 2}} & H^1_T \times H^1_T \end{tikzcd} \]
Since $\gamma :BK^c \xrightarrow{\sim} \beta(B) K^c$ is $K^c$-equivariant, it satisfies $\gamma^{\boxtimes 2} \circ \alpha_{BK^c} = \alpha_{\beta (B) K^c} \circ (\gamma \boxtimes \id)$ and hence gives an isomorphism of the groupoids $(p_1 \circ \alpha_{BK^c}, p_2 \circ \alpha_{BK^c}) :BK^c \times K^c \to BK^c$ and $(p_1 \circ \alpha_{\beta (B) K^c}, p_2 \circ \alpha_{\beta (B) K^c}) :\beta (B)K^c \times K^c \to \beta (B) K^c$. 
It follows that there is an isomorphism $(\rho, \id_B): (R, B) \xrightarrow{\sim} (R_W, B)$ of the groupoids $s \times t :R \to B \times B$ and $s_W \times t_W :R_W \to B \times B$. 
Hence there is an isomorphism $[B/R] \cong [B/R_W]$ of the quotient $\Cpx$-stacks. 

On the other hand, since $\beta :B \to H^1_T$ factors through the $K^c$-equivariant open embedding $W^\circ \to H^1_T$, the groupoids $s_W \times t_W :R_W \to B \times B$ also appears in the following cartesian diagram. 
\[ \begin{tikzcd} R_W \ar{r} \ar{d}[swap]{s_W \times t_W} & W^\circ \times K^c \ar{d}{\alpha_{W^\circ}} \\ B \times B \ar{r}{\beta^{\boxtimes 2}} & W^\circ \times W^\circ \end{tikzcd} \]
Therefore we obtain an open embedding of the quotient $\Cpx$-stacks $[B/R_W] \hookrightarrow [W^\circ/K^c]$. 

Moreover, the \'etale finite morphism $W^\circ \times^{K^c} PGL_T \to \mathrm{Hilb}^\circ$ induces an \'etale finite morphism of the quotient $\Cpx$-stacks $[W^\circ/K^c] \cong [W^\circ \times^{K^c} PGL_T/PGL_T] \to [\mathrm{Hilb}_T^\circ/PGL_T]$. 

Now combining all, we obtain an \'etale morphism $[B/R] \to \K_{T, \chi}$ with finite fibers, which obviously commutes with $B \to \K_{T, \chi}$ and $B \to [B/R]$ from its construction. 
\end{proof}

Here is our main theorem. 

\begin{thm}
\label{main theorem}
There exists a Hausdorff complex analytic space $\mathcal{K}_{T, \chi}$, which we call \textit{the moduli space of Fano manifolds with K\"ahler--Ricci solitons}, and a morphism $\K_{T, \chi} \to \mathcal{K}_{T, \chi}$ from the Artin $\Cpx$-stack $\K_{T, \chi}$ such that any morphism from $\K_{T, \chi}$ to any complex space $X$ holomorphically and uniquely factors through $\mathcal{K}_{T, \chi}$. 
Moreover, this moduli space enjoys the following property. 
\begin{enumerate}
\item The complex space $\mathcal{K}_{T, \chi}$ is normal and homeomorphic to the space $K_{T, \chi}$ in Definition \ref{support of moduli}. We will prove this in Proposition \ref{GH} after constructing a moduli space with the following property. 

\item The morphism $\K_{T, \chi} \to \mathcal{K}_{T, \chi}$ induces a bijection $|\K_{T, \chi}|/\sim ~\to \mathcal{K}_{T, \chi}$ where $|\K_{T, \chi}|$ denotes the set of points of the stack $\K_{T, \chi}$, which is canonically identified with the set of the isomorphism classes of gentle Fano manifolds, and $[X] \sim [X']$ if the central fibers of the gentle degenerations of gentle Fano manifolds $X$ and $X'$ coincide. 
\end{enumerate}
\end{thm}

As we have already noted, the logical order of our argument is ``Proposition \ref{etale} $\Rightarrow$ Proposition \ref{gentle degeneration} $\Rightarrow$ Theorem \ref{main theorem}''. 
Here we apply Proposition \ref{gentle degeneration} before we prove it. 
The author recommend the readers who prefer following the proof in the logical order to read section \ref{promised proof} firstly. 

\begin{proof}
The image of the \'etale morphism $[B/ R] \to \K_{T, \chi}$ defines an open sub-stack $\mathrm{Im} [B/R] \subset \K_{T, \chi}$. 
Object in $\mathrm{Im} [B/R]$ is an object $(\pi: \mathcal{M} \to S, \alpha)$ in $\K_{T, \chi}$ whose fibers are gentle Fano $T$-manifolds appearing in the Kuranishi family $\varpi :\mathcal{X} \to B$. 
Firstly, we prove that the morphism $[B/R] \to BK^c \sslash K^c$ in Lemma \ref{lemma} factors through $\mathrm{Im} [B/R]$. 
We construct a morphism $\mathrm{Im} [B/R] \to BK^c \sslash K^c$. 
Take an object $(\pi: \mathcal{M} \to S, \alpha)$ in $\mathrm{Im} [B/R]$ and consider the following cartesian diagram. 
\[ \begin{tikzcd} \tilde{S} \ar{r} \ar{d} & {[}B/R{]} \ar{d} \\ S \ar{r} & \mathrm{Im} [B/R] \end{tikzcd} \]
Since $[B/R] \to \mathrm{Im} [B/R]$ is \'etale, $\tilde{S} \to S$ is also \'etale. 
Then we can take local slices $s_\alpha :U_\alpha \to \tilde{S}$ of $\tilde{S} \to S$ so that $\{ U_\alpha \}_\alpha$ covers $S$ and obtain morphisms $U_\alpha \to [B/R]$, hence also obtain holomorphic morphisms $\phi_\alpha :U_\alpha \to BK^c \sslash K^c$. 
From its construction, we know that $x \in U_\alpha$ maps to the point $\phi_\alpha (x) \in BK^c \sslash K^c = \nu^{-1} (0)/K$ representing the central fiber of some gentle degeneration of $\mathcal{M}_x$, which is unique due to Proposition \ref{gentle degeneration}. 
So if $S$ is reduced, these morphisms $\phi_\alpha :U_\alpha \to BK^c \sslash K^c$ coincide on the overlaps, hence they give a holomorphic morphism $\phi :S \to BK^c \sslash K^c$, glued together. 
When $S$ is not reduced, since any Fano $T$-manifold has reduced semi-universal family, we can locally extend the morphism $S \to \mathrm{Im} [B/R]$ to some $T \to \mathrm{Im} [B/R]$ with reduced $T$. 
Take a point $x \in U_\alpha \cap U_\beta$ and a small neighbourhood $U$ of $x$ so that $U \to \mathrm{Im} [B/R]$ extends to some $T \to \mathrm{Im} [B/R]$ with reduced domain $T \supset U$ containing $U$ as a closed subspace. 
Taking smaller $T$, sections $s_\alpha|_U$, $s_\beta|_U$ extend to some sections $t_\alpha, t_\beta :T \to \tilde{T}$, where $\tilde{T}$ is given similar as $\tilde{S}$. 
Therefore the morphisms $\phi_\alpha|_U, \phi_\beta|_U : U \to BK^c \sslash K^c$ extend to some morphisms $\psi_\alpha, \psi_\beta :T \to BK^c \sslash K^c$. 
As we have already observed that $\psi_\alpha$ and $\psi_\beta$ coincide, the restrictions $\phi_\alpha|_U = \psi_\alpha \circ i_U, \phi_\beta|_U = \psi_\beta \circ i_U : U \to BK^c \sslash K^c$ also coincide as holomorphic morphisms, where $i_U: U \to T$ is the closed immersion. 
Therefore we obtain a morphism $\phi: S \to BK^c \sslash K^c$ by gluing the morphisms $\phi_\alpha :U_\alpha \to BK^c \sslash K^c$. 
It is easy to see that this construction is functorial, so we obtain the expected morphism $\mathrm{Im} [B/R] \to BK^c \sslash K^c$, which inherits the universal property from the morphism $[B/R] \to BK^c \sslash K^c$. 

Now consider two morphisms $[B/R] \to \K_{T, \chi}$ and $[B'/R'] \to \K_{T, \chi}$ with different domains. 
We also consider two maps $i :BK^c \sslash K^c \to L^2_k K (M, \omega, T)$ and $i' :B'{K'}^c \sslash {K'}^c \to L^2_k K (M, \omega, T)$. 
For any point $x \in L^2_k K (M, \omega, T)$ in the overlaps $\mathrm{Im} i \cap \mathrm{Im} i'$, we can find another \'etale morphism $[B''/R''] \to \K_{T, \chi}$ and a map $i'' :B'' {K''}^c \sslash {K''}^c \to L^2_k K (M, \omega, T)$ with $i'' ([0]) = x$ so that $\mathrm{Im} [B''/R''] \subset \mathrm{Im} [B/R] \cap \mathrm{Im} [B'/R'] \subset \K_{T, \chi}$ and $\mathrm{Im} i'' \subset \mathrm{Im} i \cap \mathrm{Im} i'$. 
Especially we have a natural inclusion morphism $\mathrm{Im} [B''/R''] \to \mathrm{Im} [B/R]$ and hence obtain a morphism $\mathrm{Im} [B''/R''] \to BK^c \sslash K^c$. 
From the universality of the morphism $\mathrm{Im} [B''/R''] \to B'' {K''}^c \sslash {K''}^c$, we obtain a holomorphic morphism $B''{K''}^c \sslash {K''}^c \to BK^c \sslash K^c$. 
This holomorphic morphism is clearly compatible with $i''$ and $i$ as maps, so especially it is a homeomorphism onto its open image, after taking smaller $B''$ if necessary. 
Since the analytic GIT quotient spaces $BK^c \sslash K^c$ are normal, this holomorphic homeomorphism is actually a biholomorphism. 
This argument shows that the coordinate change ${i'}^{-1} \circ i$ is biholomorphic. 
Thus we obtain a complex space $\mathcal{K} (M, \omega, T)$ by giving a complex structure on the topological space $L^2_k K (M, \omega, T)$ defined from the above holomorphic charts. 
Set $\mathcal{K}_{T, \chi} := \coprod_{\chi (M, \omega, T) = \chi} \mathcal{K} (M, \omega, T)$. 
Clearly from its construction, there is a morphism $\K_{T, \chi} \to \mathcal{K}_{T, \chi}$ enjoying the universal property. 
It follows from section 3 and Proposition \ref{gentle degeneration} that this morphism induces a bijection $|\K_{T, \chi}| /\sim~ \to \mathcal{K}_{T, \chi}$. 
We prove in the next subsection that the space $\mathcal{K} (M, \omega, T)$, which is homeomorphic to $L^2_k K (M, \omega, T)$, is actually homeomorphic to $K (M, \omega, T)$. 
\end{proof}

\begin{cor}
The $\Cpx$-stack $\K^s_{T, \chi}$ admits a tame moduli space $\K^s_{T, \chi} \to \mathcal{K}^s_{T, \chi}$ (see \cite[Definition 7.1]{Alp1}), with the same universal property as the moduli space $\K_{T, \chi} \to \mathcal{K}_{T, \chi}$. 
Moreover, the complex space $\mathcal{K}^s_{T, \chi}$ is a Hausdorff complex orbifold. 
\end{cor}

This corollary follows from the construction in the proof of the main theorem, the openness property of the K-stable Fano $T$-manifolds and the fact that $[B/R] \to \K_{T, \chi}^s$ is an open embedding in this case, which is an easy consequence of the injectivity of the map $|[B/R]| \to BK^c \sslash K^c \approx \nu^{-1} (0)/K \to K_{T, \chi}$ and the bijection $|\K^s_{T, \chi}| \to K_{T, \chi}$. 
The orbifold coordinates are given by open neighbourhoods of the origin in the spaces $H^1_T (X, \Theta) \sslash (\mathrm{Aut}_T (X)/T)$. 
We can also consider a separated smooth Deligne-Mumford $\Cpx$-stack $\K^*_{T, \chi}$ associated to the $\Cpx$-stack $\K^s_{T, \chi}$.

\subsection{Consistency}

In the previous section, we constructed a complex analytic space structure on the spaces $L^2_k K (M, \omega, T) = (\mathcal{S}^\intg_\xi)^{-1} (0)^2_k /\Ham_T (M, \omega)^2_{k+1}$ and proved that it has a certain universality independent of $k$, which is described in terms of the stack $\K_{T, \chi}$. 
Since the universality determines a complex space uniquely up to biholomorphisms, the complex spaces $(\mathcal{S}^\intg_\xi)^{-1} (0)^2_k /\Ham_T (M, \omega)^2_{k+1}$ are all canonically biholomorphic to each other. 
In particular, we deduce that they are all homeomorphic through the following natural maps
\[ \mathcal{I}_{l, k} :(\mathcal{S}^\intg_\xi)^{-1} (0)^2_l /\mathrm{Ham}_T (M, \omega)^2_{l+1} \to (\mathcal{S}^\intg_\xi)^{-1} (0)^2_k/\mathrm{Ham}_T (M, \omega)^2_{k+1} :[J] \mapsto [J]. \]

Now we show the continuous map 
\[ \mathcal{I}_k :(\mathcal{S}^\intg_\xi)^{-1} (0)/\Ham_T (M, \omega) \to (\mathcal{S}^\intg_\xi)^{-1} (0)^2_k /\Ham_T (M, \omega)^2_{k+1} : [J] \mapsto [J] \]
is also homeomorphic, using that $\mathcal{I}_{l, k}$ is homeomorphic. 

In the proof of the following proposition, we hope to apply an $L^2_k$-version of Newlander--Nirenberg theorem. 
Unfortunately, the author could not find in literatures a precise $L^2_k$-version of Newlander--Nirenberg theorem for our purpose, but only find the reference \cite{NW}. 
As the Newlander--Nirenberg type theorem in \cite{NW} losses some regularity, we make use of the above fact obtained from the universality of our moduli space. 

\begin{prop}
\label{homeo}
The continuous map $\mathcal{I}_k$ is a homeomorphism. 
\end{prop}

\begin{proof}
Take two elements $J, J' \in (\mathcal{S}^\intg_\xi)^{-1} (0)$ and suppose there is a $L^2_{k+1}$-regular map $\phi \in \Ham_T (M, \omega)^2_{k+1}$ such that $\phi^* J = J'$. 
Then $\phi$ is $C^\infty$-smooth by Myers-Steenrod theorem. 
This shows that $\mathcal{I}_k$ is injective. 

Next we show the surjectivity. 
It is sufficient to show that for any $J \in (\mathcal{S}^\intg_\xi)^{-1} (0)^2_k$, there is a $L^2_{k+1}$-regular map $\phi \in \Ham_T (M, \omega)^2_{k+1}$ such that $\phi^* J \in (\mathcal{S}^\intg_\xi)^{-1} (0)$. 
Take a large integer $m \ge 2$ and $l$ so that $L^2_l \subset C^{m-1, \alpha} \subset C^{m-1} \subset L^2_{k+1}$. 
Since $\mathcal{I}_{l, k}$ is a homeomorphism, we can find a $L^2_{k+1}$-regular map $\phi_0 \in \Ham_T (M, \omega)^2_{k+1}$ so that $\phi_0^* J \in (\mathcal{S}^\intg_\xi)^{-1} (0)^2_l$. 
Then it follows from \cite{NW} that there is a $C^{m, \alpha/n}$-smooth diffeomorphism $\phi_1 :M' \to M$ such that $\phi_1^* \phi^*_0 J$ is a $C^\infty$-smooth integrable complex structure, where on the other hand $\phi_1^* \phi_0^* \omega$ and $\phi^*_1 \phi_0^* g_J$ is only $C^{m-1, \alpha/n}$-regular, in particular, $L^2_{k+1}$-regular. 
We can choose a $C^\infty$-smooth diffeomorphism $\phi_2 :M \to M'$, which we can additionally suppose that it is sufficiently close to $\phi_1^{-1}$ in $C^m$-topology. 
Note that $\phi_0 \circ \phi_1 \circ \phi_2$ is sufficiently close to $\phi_0$ in $L^2_{k+2}$-topology. 
The pull-back metric $\phi^*_2 (\phi_0 \circ \phi_1)^* g_J$ is a $L^2_{k+1}$-regular metric which is a K\"ahler--Ricci soliton with respect to $C^\infty$-smooth integrable complex structure $\phi^*_2 (\phi_0 \circ \phi_1)^* J$. 
The elliptic regularity argument shows that $(\phi_0 \circ \phi_1 \circ \phi_2)^* g_J$ is in fact $C^\infty$-smooth. 
Hence $(\phi_0 \circ \phi_1 \circ \phi_2)^* \omega$ is also $C^\infty$-smooth.
Since we further assume that $(\phi_0 \circ \phi_1 \circ \phi_2)^* \omega$ is close to $\phi^*_0 \omega = \omega$ in $L^2_{k+1}$-topology, both $C^\infty$-smooth symplectic form $\omega, (\phi_0 \circ \phi_1 \circ \phi_2)^* \omega$ have the same cohomology classes and $\omega_t := t\omega + (1-t) (\phi_0 \circ \phi_1 \circ \phi_2)^* \omega$ is non-degenerate for any $t \in [0,1]$. 
From Moser's theorem, we obtain a $C^\infty$-smooth diffeomorphism $\phi_3$ satisfying $\phi_3^* (\phi_0 \circ \phi_1 \circ \phi_2)^* \omega = \omega$, which is close to $\id_M$ in $L^2_{k+1}$-topology as in the proof of Proposition \ref{local slice}. 
Now we have obtained the expected $L^2_{k+1}$-regular map $\phi := \phi_0 \circ \phi_1 \circ \phi_2 \circ \phi_3$. 
From the construction, we know $\phi$ can be taken sufficiently close to $\phi_0$ in $L^2_{k+1}$-topology. 

Finally we prove that $\mathcal{I}_k$ is actually a homeomorphism. 
Take a convergent sequence $J_n \to J_\infty \in (\mathcal{S}^\intg_\xi)^{-1} (0)^2_k$. 
It suffices to show that there are elements $\phi_n, \phi_\infty \in \Ham_T (M, \omega)^2_{k+1}$ such that $\phi^*_n J_n, \phi_\infty^* J_\infty$ belong to $(\mathcal{S}^\intg_\xi)^{-1} (0)$ and the sequence $\phi^*_n J_n$ converges to $\phi^*_\infty J_\infty$ in the $C^\infty$-topology, by taking a subsequence if necessary (thanks to the injectivity of $\mathcal{I}_k$). 
Since $\mathcal{I}_k$ is surjective, we can find an element $\phi_\infty \in \Ham_T (M, \omega)^2_{k+1}$ so that $\phi_\infty^* J_\infty$ is $C^\infty$-smooth and hence there is no loss of generality in supposing $C^\infty$-smoothness of $J_\infty$ from the beginning. 
Since $\mathcal{I}_{l, k}$ is a homeomorphism, we can find a sequence $\phi_{n,l} \in \Ham_T (M, \omega)^2_{k+1}$ so that $\phi_{n, l}^* J_n \in (\mathcal{S}^\intg_\xi)^{-1} (0)^2_l$ and $\phi_{n, l}^* J_n$ converges to $J_\infty$ in the $L^2_l$-topology. 
We define a set
\[ \Sigma_l (J) := \{ \phi \in \mathrm{Ham}_T (M, \omega)^2_{l+1} ~|~ \phi^* J \in (\mathcal{S}^\intg_\xi)^{-1} (0) \} \]
for $J \in (\mathcal{S}^\intg_\xi)^{-1} (0)^2_l$. 
Since $\mathcal{I}_l$ is surjective, $\Sigma_l (J)$ is a non-empty set. 
Moreover, the density of $\Ham_T (M, \omega) \subset \Ham_T (M, \omega)^2_{l+1}$, which we can deduce from Weinstein's tubular neighbourhood theorem, shows that $\Sigma_l (J) \subset \Ham_T (M, \omega)^2_{l+1}$ is also dense. 
Therefore we can perturb $\phi_{n, l}$, as small as $n$ goes to the infinity, so that $\phi_{n, l}^* J_n$ are $C^\infty$-smooth and preserve the $L^2_l$-convergence $\phi_{n, l}^* J_n \to J_\infty \in (\mathcal{S}^\intg_\xi)^{-1} (0)^2_l$. 
Now we can proceed to the diagonal argument with respect to $(n, l)$ and conclude that a subsequence $\phi_{n_l, l}^* J_{n_l}$ converges to $J_\infty$ in $C^\infty$-topology. 
\end{proof}

There is another topological space consisting of biholomorphism classes of Fano manifolds with K\"ahler--Ricci solitons, which is considered in \cite{PSS}. 
\[ \mathcal{KR}_{GH} (n, F) := \Big{\{} [M, J, g, \xi'] ~\Big{|}~ \begin{matrix} (M, J, g, \xi') \text{ is a Fano manifold $(M, J)$ with } \\ \text{a K\"ahler--Ricci soliton $(g, \xi')$ and } \int_M |\xi'|_g^2 \omega^n \le F. \end{matrix} \Big{\}} \]
In \cite{PSS}, a topological compactification $\overline{\mathcal{KR}_{GH} (n, F)}$ of this space is considered in regard to the `complexified' Gromov--Hausdorff convergence. 
It was open whether the space $\mathcal{KR}_{GH} (n, F)$ is stable for large $F$, which was expected in \cite{PSS}. 
As we have the equality $\mathrm{Fut} (\xi') = 2 \int_M |\xi'|_g^2 \omega^n$ for K\"ahler--Ricci soliton $(g, \xi')$, this is equivalent to say that the invariants $\mathrm{Fut} (\xi')$ are uniformly bounded from above for $n$-dimensional Fano manifolds with K\"ahler--Ricci solitons. 

We show that this invariant is actually bounded, not only for Fano manifolds with K\"ahler--Ricci solitons, but also for all $n$-dimensional Fano $T$-manifolds with the maximal $K$-optimal action. 
We also compare $\mathcal{KR}_{GH} (n, F)$ with our $K (M, \omega, T)$ and $K_{T, \chi}$, which a priori have different topologies. 

\begin{prop}
\label{GH}
Set
\[ \mathcal{KR}_{GH} (n) := \Big{\{} [M, J, g, \xi'] ~\Big{|}~ \begin{matrix} (M, J, g, \xi') \text{ is a Fano manifold $(M, J)$ } \\ \text{ with a K\"ahler--Ricci soliton $(g, \xi')$. } \end{matrix} \Big{\}}. \]
Then $\mathcal{KR}_{GH} (n) = \mathcal{KR}_{GH} (n, F)$ for large $F$ and the map
\[ K (M, \omega, T) \to \mathcal{KR}_{GH} (n) :[J] \mapsto [M, J, g_J] \]
gives a homeomorphism onto a clopen (closed and open) subset of $\mathcal{KR}_{GH} (n)$, for any $2n$-dimensional symplectic Fano manifold $(M, \omega)$ with K-optimal $T$-action. 
\end{prop}

\begin{proof}
Since Fano manifolds are bounded \cite{KMM}, we have a sufficiently large Hilbert scheme $\mathrm{Hilb}$ of $\mathbb{C}P^N$ with bounded Hilbert polynomials so that for any Fano manifold $X$ we can find a point $[X] \in \mathrm{Hilb}$, uniquely up to $\mathrm{Aut} (\mathbb{P}^N) = PGL (N+1)$-action, representing an anti-canonically embedded $X \subset \mathbb{C}P^N$. 
We denote by $\mathrm{Hilb}_\mathrm{Fano}$ the Zariski open locus parametrizing the anti-canonically embedded Fano manifolds. 
Obviously, $PGL (N+1)$ preserves the subset $\mathrm{Hilb}_\mathrm{Fano}$. 

Fix a maximal algebraic torus $T$ of $PGL (N+1)$ and consider its action on $\mathrm{Hilb}_\mathrm{Fano}$. 
Note as we have $\mathrm{Stab} ([X]. g) = g^{-1} \mathrm{Stab} ([X]) g \subset PGL (N+1)$ for $g \in PGL (N+1)$, we can find a point $[X] \in \mathrm{Hilb}_\mathrm{Fano}$ so that $\mathrm{Stab} ([X]) \cap T \subset \mathrm{Stab} ([X]) \cong \mathrm{Aut} (X)$ is a maximal torus. 
Indeed, for a maximal torus $T_X \subset \mathrm{Stab} ([X])$, pick a maximal torus $T' = g T g^{-1} \subset PGL (N+1)$ so that $T' \cap \mathrm{Stab} ([X]) = T_X$, then we have a maximal torus $\mathrm{Stab} ([X]. g) \cap T = g^{-1} T_X g \subset \mathrm{Stab} ([X]. g)$. 

Next, consider the normalization $\widetilde{\mathrm{Hilb}} \to \mathrm{Hilb}$, where $\widetilde{\mathrm{Hilb}}$ is a normal projective variety and the morphism is a finite surjective morphism. 
Then we have a $T$-equivariant embedding of $\widetilde{\mathrm{Hilb}}$ into some $\mathbb{P} (V)$, where $V$ is a $T$-representation (\cite[Corollary 1.6]{MFK}). 
Since $V$ decomposes into $1$-dimensional representations as $V \cong \mathbb{C}_{\chi_1} \oplus \dotsb \oplus \mathbb{C}_{\chi_{\dim V}}$, the stabilizer $T_x \subset T$ of any point $x \in \mathbb{P} (V)$ can be written as $\chi_{i_1}^{-1} (1) \cap \dotsb \cap \chi_{i_j}^{-1} (1)$, hence the possibilities are finite. 
It also follows that every fiber $S_T^{-1} (T')$ of the following map
\[ S_T : \mathbb{P} (V) \to \{ \text{ sub torus of } T ~\} : x \mapsto T_x \]
is a (possibly non irreducible) subvariety in $\mathbb{P} (V)$. 
Therefore, we obtain a finite stratification $\{ S_T^{-1} (T_i) \subset \mathbb{P} (V) \}$ and $\{ H_i \subset \widetilde{\mathrm{Hilb}}_\mathrm{Fano} \}$ by its restriction. 
We refine this stratification by taking connected components of each $H_i$ and continue to write $\{ H_i \subset \widetilde{\mathrm{Hilb}}_\mathrm{Fano} \}$. 
Since the restricted family $\mathcal{U}|_{H_i} \to H_i$ gives a family of Fano $T_i$-manifolds, we can consider the K-optimal vector $\xi_i \in (N_i)_\mathbb{R}$ with respect to the $T_i$-action on the Fano manifolds $X_s$ ($s \in H_i$), which is independent of $s \in H_i$. 
Let $T'_i \subset T_i$ be the sub-torus generated by $\xi_i$. 

Now from the construction, every Fano manifold $X$ with a maximal K-optimal $T'$-action finds some $H_i$ satisfying $[X] \in H_i$ and $T' =T'_i$. 
Since the Futaki invariant of $\xi'_i$ on $X_s$ is independent of the choice of $s \in H_i$, we conclude that there are only finitely many possibilities of the values of $\mathrm{Fut}_X (\xi')$ for the pairs $(X, \xi')$ of Fano manifolds with vanishing modified Futaki invariant $\mathrm{Fut}_{X, \xi'}$. 
In particular, $\mathrm{Fut}_X (\xi')$ is bounded for $(X, g, \xi') \in \mathcal{KR}_{GH} (n)$ and hence $\mathcal{KR}_{GH} (n, F) = \mathcal{KR} (n)$ for large $F$. 

Next we see that the given map $K (M, \omega, T) \to \mathcal{KR}_{GH} (n)$ is a homeomorphism by a standard argument as follows. 
The continuity of the map is obvious. 
For every $[M, J, g] \in \mathcal{KR}_{GH} (n)$ and any two representatives $(M_1, J_1, g_1), (M_2, J_2, g_2) \in [M, J, g]$, we have a diffeomorphism $\phi :M_1 \to M_2$ satisfying $\phi^* J_2 = J_1$, $\phi^* g_2 = g_1$ and $(\phi^{-1})_* \xi'_2 = \xi_1'$, where $\xi_i'$ is the unique holomorphic vector field satisfying $\Ric (g_i) -L_{\xi'_i} g_i = g_i$. 
It follows that the map $K (M, \omega, T) \to \mathcal{KR}_{GH} (n)$ is injective for K-optimal $(M, \omega, T)$, and the images of $K (M_1, \omega_1, T_1), K (M_2, \omega_2, T_2) \to \mathcal{KR}_{GH} (n)$ intersect iff there is an isomorphism $\theta :T_1 \xrightarrow{\sim} T_2$ and a $(T_1, T_2)$-equivariant symplectic diffeomorphism $(M_1, \omega_1) \xrightarrow{\sim} (M_2, \omega_2)$. 

Since the images of the maps for distinct pairs $(M_1, \omega_1, T_1)$, $(M_2, \omega_2, T_2)$ are disjoint, it suffices to prove that the maps are closed. 
Actually, if the maps are closed, then the maps are homeomorphisms onto their images and the images are open from the above finiteness of the possibilities of the K-optimal pairs $(M, \omega, T)$. 
To see the closedness, take a sequence $[J_n] \in K (M, \omega, T)$ which has the convergent images $[M, J_n, g_{J_n}] \to [M_\infty, J_\infty, g_\infty]$ in $\mathcal{KR}_{GH} (n)$. 
As remarked before Proposition 6.1 in \cite{PSS}, we have a sequence $[M, J_n, g_{J_n}] \in \mathrm{Hilb}_T^\circ$ which converges to $[M_\infty, J_\infty, g_\infty] \in \mathrm{Hilb}_T^\circ$, where $\mathrm{Hilb}^\circ_T$ denotes the open subset of $\mathrm{Hilb}_T$ parametrizing gentle Fano $T$-manifolds with bounded Hilbert polynomial. 
Now we have a canonical continuous (holomorphic) map $\mathrm{Hilb}^\circ_T \to K (M, \omega, T) \subset \mathcal{KR}_{T, \chi}$ induced by the universality of $\mathcal{K}_{T, \chi}$. 
The image of the sequence $[M, J_n, g_{J_n}]$ is nothing but the original sequence $[J_n]$, so we obtain the convergence of $[J_n]$ to the image of $[M_\infty, J_\infty, g_\infty]$ in $K (M, \omega, T)$. 
\end{proof}

\begin{rem}
The compactification $\overline{\mathcal{KR}_{GH}} (n)$ of $\mathcal{KR}_{GH} (n)$ constructed in \cite{PSS} is a compact Hausdorff space with a countable basis (cf. \cite{DSII}) and the boundary $\overline{\mathcal{KR}_{GH}} (n) \setminus \mathcal{KR}_{GH} (n)$ is closed. 
The closedness of the boundary is easily confirmed as follows. 
Suppose $[X_n, g_n]$ is a sequence in $\overline{\mathcal{KR}_{GH}} (n) \setminus \mathcal{KR}_{GH} (n)$ converging to $[X_\infty, g_\infty]$ in $\overline{\mathcal{KR}_{GH}} (n)$. 
Take a sequence $[X_{n, i}, g_{n, i}] \in \mathcal{KR}_{GH} (n)$ for each $n$ converging to $[X_n, g_n]$ in $\overline{\mathcal{KR}_{GH}} (n)$. 
We can suppose that $\mathrm{Hilb} (X_{n, i}, g_{n, i}) \to \mathrm{Hilb} (X_n, g_n)$ in $\mathrm{Hilb}_T$. 
Then we can find a subsequence of $[X_n, g_n]$ so that $\mathrm{Hilb} (X_n, g_n) \to \mathrm{Hilb} (X_\infty, g_\infty)$ in $\mathrm{Hilb}_T$ by the diagonal argument. 
Since the subset of $\mathrm{Hilb}_T$ parametrizing singular subspaces of $\mathbb{C}P^N$ forms a closed subset, the limit $[X_\infty, g_\infty]$ must be also singular, hence $[X_\infty, g_\infty] \in \overline{\mathcal{KR}_{GH}} (n) \setminus \mathcal{KR}_{GH} (n)$
\end{rem}

\subsection{The promised proof of Proposition \ref{gentle degeneration}}
\label{promised proof}

If $X$ is a gentle Fano $T$-manifold, then $R_{\xi'} (X) = 1$ for the K-optimal vector $\xi'$. 
So there exists a unique solution $\omega_t = \omega_t (\alpha)$ of the following equation 
\[ \mathrm{Ric} (\omega_t) - L_\xi \omega_t = t\omega_t + (1-t) \alpha \]
for every $t \in [0, 1)$ and any initial metric $\alpha$. 

\begin{lem}
Let $\mathcal{X} \to \bar{\Delta}$ be a family of Fano $T$-manifolds with $R_{\xi'} (\mathcal{X}_\sigma) = 1$ for the K-optimal vector $\xi$ over a compact disc $\bar{\Delta}$ and $\bm{\alpha}$ be a smooth family of $T_\mathbb{R}$-invariant K\"ahler metrics $\alpha_\sigma$ on $\mathcal{X}_\sigma$. 
Then there is a sufficiently divisible $k \in \mathbb{N}$ and a positive constant $c > 0$ which depend only on the pair $(\mathcal{X}, \bm{\alpha})$ such that for any $\sigma \in \bar{\Delta}$ and $t \in [0, 1)$ the following uniform partial $C^0$-estimate holds. 
\[ \rho_{X_\sigma, \omega_t (\alpha_\sigma), k} \ge c, \]
where $\rho_{X_\sigma, \omega_t (\alpha_\sigma), k}$ denotes the function on $X_\sigma$ defined by
\[ \rho_{X_\sigma, \omega_t (\alpha_\sigma), k} (x) := \max |s (x)|_{h_{X_\sigma, \omega_t (\alpha_\sigma), k}}, \]
where $s$ runs over $s \in H^0 (X_\sigma, \mathcal{O} (- k K_X))$ with $\int_{X_\sigma} | s |^2_{h_{X_\sigma, \omega_t (\alpha_\sigma), k}} (k \omega_t (\alpha_\sigma))^n =1$ and $h_{X_\sigma, \omega_t (\alpha_\sigma), k}$ denotes a metric on $-k K_{X_\sigma}$ whose curvature is $k\omega_t (\alpha_\sigma)$. 
\end{lem}

\begin{proof}
This follows from estimates in the proof of Lemma 5.6, Lemma 5.7 and Lemma 5.8 in \cite{F. WZ}. 
Note that we can uniformly take constants $C$ in Lemma 5.6, $B$ in Lemma 5.7 and $c_1, C$ in Lemma 5.8 independent of $\alpha_\sigma$, since the constants of Theorem A in \cite{Mab} and of Corollary 5.3 in \cite{Zhu} can be uniformly taken. 
Then it follows that any sequence $(X_{\sigma_i}, \omega_{t_i} (\alpha_{\sigma_i}))$ $(t_i \to 1)$ is a sequence of almost K\"ahler--Ricci solitons in the sense of \cite[Definition 5.1]{F. WZ}. 
Now we can deduce our estimate from \cite[Corollary 1.4]{JWZ}, \cite[Lemma 3.4]{DS} and the argument after the lemma. 
\end{proof}

Now we can apply the arguments in \cite{DS} to the metric family 
\[ \{ (X, \omega_t (\alpha_s)) \}_{(t,s) \in [0,1) \times [0,1]}, \]
under the above partial $C^0$-estimate. 
Thus we have a sufficiently divisible number $k \in \mathbb{N}$ with the following properties. 
\begin{enumerate}
\item The pair $(X_\sigma, \omega_t (\alpha_\sigma))$ defines a point $\mathrm{Hilb} (X_\sigma, \omega_t (\alpha_\sigma))$ in the compact Hausdorff topological space $\mathrm{Hilb}_T/U_T$ by embedding $X_\sigma$ into $\mathbb{C}P^N$ using a unitary basis of $H^0 (X_\sigma, \mathcal{O} (-k K_{X_\sigma}))$ with respect to the metric $\omega_t (\alpha_\sigma)$. 

\item For any sequence $(\sigma_i, t_i) \in \Delta \times [0, 1]$, we have a subsequence such that $(X_{\sigma_i}, \omega_{t_i} (\alpha_{\sigma_i}))$ converges in the `complexified' Gromov--Hausdorff topology to some $\mathbb{Q}$-Fano variety $X_\infty$ with a K\"ahler--Ricci soliton $(\omega_\infty, \xi'_\infty)$. 

\item After taking a further subsequence, the sequence $\mathrm{Hilb} (X_{\sigma_i}, \omega_{t_i} (\alpha_{\sigma_i})) \in \mathrm{Hilb}_T/U_T$ converges in $\mathrm{Hilb}_T/U_T$ to the point $\mathrm{Hilb} (X_\infty, \omega_\infty)$ which is similarly defined using a unitary embedding $X_\infty \hookrightarrow \mathbb{C}P^N$. 
\end{enumerate}

\begin{proof}[\textbf{Proof of Proposition \ref{gentle degeneration}}]
Let $\{ (X, \omega_t (\alpha)) \}_{t \in [0, 1)}$ be the family of solutions of the continuity method with an initial metric $\alpha$. 
Suppose there is a smooth Fano $T$-manifold with K\"ahler--Ricci soliton $(X_\infty, \omega_\infty)$ which is the limit of a Gromov--Hausdorff convergent subsequence $(X, \omega_{t_i} (\alpha))$. 
First we show that the limit $(X_\infty, \omega_\infty)$ is uniquely determined independent of the choice of the initial metrics $\alpha$ and the subsequences $(X, \omega_{t_i} (\alpha))$. 
Suppose $\alpha'$ is another K\"ahler metric on $X$ and $(X, \omega_{t_i'} (\alpha')) \to (X_\infty', \omega_\infty')$ be a convergent subsequence to a $\mathbb{Q}$-Fano $T$-variety with K\"ahler--Ricci soliton. 
Set $\alpha_s := s \alpha' + (1-s) \alpha$. 
As we noted right before this proof, we can find a sufficiently divisible number $k_{\bm{\alpha}} \in \mathbb{N}$ so that all $(X, \omega_t (\alpha_s))$ can be uniformly embedded using the unitary basis of $H^0 (X, \mathcal{O} (- k_{\bm{\alpha}} K_X))$ with respect to $\omega_t (\alpha_s)$, which defines a point $\mathrm{Hilb} (X, \omega_t (\alpha_s)) \in \mathrm{Hilb}_T/U_T$. 
Moreover, we can assume $(X_\infty, \omega_\infty)$ and $(X_\infty', \omega_\infty')$ also define points $\mathrm{Hilb} (X_\infty, \omega_\infty) \in \mathrm{Hilb}_T/U_T$, $\mathrm{Hilb} (X_\infty', \omega_\infty') \in \mathrm{Hilb}_T/U_T$ respectively, and $\mathrm{Hilb} (X, \omega_{t_i} (\alpha)) \to \mathrm{Hilb} (X_\infty, \omega_\infty) \in \mathrm{Hilb}_T/U_T$, $\mathrm{Hilb} (X, \omega_{t_i'} (\alpha')) \to \mathrm{Hilb} (X_\infty', \omega_\infty') \in \mathrm{Hilb}_T/U_T$. 
These embeddings clearly define a continuous map $[0,1) \times [0,1] \to \mathrm{Hilb}/U_T : (t, s) \mapsto \mathrm{Hilb} (X, \omega_t (\alpha_s))$. 

Suppose $X_\infty \ncong X_\infty'$. 
If $\overline{\mathrm{Hilb} (X_\infty, \omega_\infty) PGL_T} \cap \mathrm{Hilb} (X_\infty', \omega_\infty') PGL_T \neq \emptyset$, then we obtain a test configuration of $X_\infty$ with the central fiber $X_\infty'$ from the reductivity of the stabilizer $\mathrm{Aut}_T (X_\infty')$, which allows to apply the \'etale slice theorem \cite[Theorem 2.1]{AHR} and the Hilbert-Mumford theorem. 
Since the central fiber admits a K\"ahler--Ricci soliton, the modified algebraic Futaki invariant of this test configuration is zero. 
However, as $X_\infty$ has K\"ahler--Ricci soliton and hence K-polystable, $X_\infty'$ must be isomorphic to $X_\infty$. 
This contradicts to our assumption. 

So we have $\overline{\mathrm{Hilb} (X_\infty, \omega_\infty) PGL_T} \cap \mathrm{Hilb} (X_\infty', \omega_\infty') PGL_T = \emptyset$. 
Then in particular we can take open neighbourhoods $B_\varepsilon (\overline{\mathrm{Hilb} (X_\infty, \omega_\infty) PGL_T})$, $B_{\varepsilon'} (\mathrm{Hilb} (X_\infty', \omega_\infty') U_T)$ separating the two closed subsets $\overline{\mathrm{Hilb} (X_\infty, \omega_\infty) PGL_T}$ and $\mathrm{Hilb} (X_\infty', \omega'_\infty) U_T$. 
Here we use a $U_T$-invariant distance on $\mathrm{Hilb}_T$ to consider $B_\varepsilon$ and fix this distance. 
Take $U_T$-invariant open neighbourhoods $V \Subset V' \subset \mathrm{Hilb}_T$ of $\mathrm{Hilb} (X_\infty, \omega_\infty) U_T$ so that $\mathcal{U}|_{V'} \to V'$ parametrizes Fano $T$-manifolds appearing in the family $\varpi: \mathcal{X} \to B$ with central fiber $\mathcal{X}_0 = X_\infty$. 
We can assume $\mathrm{Hilb} (X, \omega_{t_i} (\alpha)) \in V/U_T$. 
From the finiteness of the fibers of the morphism $[B/R] \to \K_{T, \chi}$ in Proposition \ref{etale}, there are only finitely many isomorphism classes of Fano $T$-manifolds with K\"ahler--Ricci solitons in this family that can be the central fiber of some gentle degeneration of $X$. 
Putting $\omega_i (\sigma) := \omega_{\sigma t_i' + (1-\sigma)t_i} (\alpha_\sigma)$, we have a continuous curve $\mathrm{Hilb} (X, \omega_i (-)) :[0,1] \to \mathrm{Hilb}_T/U_T$. 
Furthermore, putting
\[ \sigma_i := \sup \{ \sigma \in [0,1] ~|~ \mathrm{Hilb} (X, \omega_i (-))|_{[0, \sigma)} \subset B_\varepsilon (\overline{\mathrm{Hilb} (X_\infty, \omega_\infty) PGL_T}/U_T) \}, \]
we obtain a sequence of almost K\"ahler--Ricci solitons in the sense of \cite{F. WZ}. 
So after taking a subsequence, we have a sequence $(X, \omega_i (\sigma_i))$ converging to some $\mathbb{Q}$-Fano $T$-variety admitting K\"ahler--Ricci soliton $(X_\infty'', \omega_\infty'')$ with the convergent corresponding sequence $\mathrm{Hilb} (X, \omega_i (\sigma_i)) \to \mathrm{Hilb} (X_\infty'', \omega_\infty'')$ in $\mathrm{Hilb}_T/U_T$. 
Replacing $\varepsilon$ with $\varepsilon/2^k$, we can construct $\sigma_{i, k}$ and $X_{\infty, k}''$ by the same process. 

Suppose there is infinitely many $i$ for each $k$ such that $\mathrm{Hilb} (X, \omega_i (\sigma_{i, k})) \notin VPGL_T/U_T$. 
After taking subsequence, we know that
\begin{align}
\label{gentle degeneration1}
\mathrm{Hilb} (X, \omega_i (\sigma'_{i, k})) \in \partial (\overline{V PGL_T}/U_T) \cap B_{\varepsilon/2^{k-1}} (\overline{\mathrm{Hilb} (X_\infty, \omega_\infty) PGL_T}/U_T) 
\end{align}
for
\[ \sigma_{i, k}' := \sup \{ \sigma \in [0,\sigma_{i,k}] ~|~ \mathrm{Hilb} (X, \omega_i (-))|_{[0, \sigma)} \subset V PGL_T/U_T \}. \]
Since $(X, \omega_i (\sigma_{i,k}'))$ is a sequence of almost K\"ahler--Ricci solitons for each $k$, we can assume $(X, \omega_i (\sigma_{i, k}')) \to (X_{\infty,k}''', \omega_{\infty, k}''')$ for some $\mathbb{Q}$-Fano $T$-variety with K\"ahler--Ricci soliton $(X_{\infty,k}''', \omega_{\infty, k}''')$. 
The diagonal argument shows that there is a subsequence $\{ (X, \omega_{i_k} (\sigma_{i_k, k}')) \}_{k=1}^\infty$ of $\{ (X, \omega_i (\sigma_{i, k}')) \}_{i, k}$ and a $\mathbb{Q}$-Fano $T$-variety $X_{\infty, \infty}'''$ such that $(X, \omega_{i_k} (\sigma_{i_k, k}')) \to (X_{\infty, \infty}''', \omega_{\infty, \infty}''')$ and $\mathrm{Hilb} (X, \omega_{i_k} (\sigma_{i_k, k}')) \to \mathrm{Hilb} (X_{\infty, \infty}''', \omega_{\infty, \infty}''')$. 
Now from the property (\ref{gentle degeneration1}), we conclude $\mathrm{Hilb} (X_{\infty, \infty}''', \omega_{\infty, \infty}''') \in \overline{\mathrm{Hilb} (X_\infty, \omega_\infty)PGL_T} \setminus \mathrm{Hilb} (X_\infty, \omega_\infty) PGL_T$. 
But this is absurd in the same way as we have seen before. 

Therefore we can assume that for any large $k$, $\mathrm{Hilb} (X, \omega_i (\sigma_{i, k}))$ is in the neighbourhood $VPGL_T/U_T$ except for only finitely many $i$. 
In this case, the convergent sequence $(X, \omega_i (\sigma_{i, k})) \to (X_{\infty,k}'', \omega_{\infty, k}'')$ defines a convergent sequence $\mathrm{Hilb} (X, \omega_i (\sigma_{i, k})) \to \mathrm{Hilb} (X_{\infty,k}'', \omega_{\infty, k}'')$ in $\mathrm{Hilb}_T/U_T$ that is uniformly away from $\mathrm{Hilb} (X_\infty, \omega_\infty)PGL_T$ because $\mathrm{Hilb} (X, \omega_i (\sigma_{i, k})) \in \partial B_{\varepsilon/2^k} (\overline{\mathrm{Hilb} (X_\infty, \omega_\infty) PGL_T}/U_T)$. 
It follows that $X_{\infty, k}'' \ncong X_\infty$. 
Since $\mathrm{Hilb} (X_\infty'', \omega_{\infty, k}'') \in V' PGL_T/U_T$ and each there is a gentle degeneration of $X$ with its central fiber $X_{\infty, k}''$, there is only finitely many isomorphism classes in $\{ X_{\infty, k}'' \}_{k=1}^\infty$. 
So we can assume $X_{\infty, k}''$ is all isomorphic after taking subsequence. 
From the uniqueness of K\"ahler--Ricci soliton, the sequence $(X_{\infty, k}'', \omega_{\infty,k}'')$ is constant and hence converges to the limit $(X_{\infty, \infty}'', \omega_{\infty, \infty}'') \cong (X_{\infty, k}'', \omega_{\infty,k}'')$. 
It follows that $\mathrm{Hilb} (X_{\infty, \infty}'', \omega_{\infty, \infty}'') \in \overline{\mathrm{Hilb} (X_\infty, \omega_\infty) PGL_T}$ from the fact
\[ \mathrm{Hilb} (X_{\infty, k}'', \omega_{\infty, k}'') \in \overline{B_{\varepsilon/2^k}} (\overline{\mathrm{Hilb} (X_\infty, \omega_\infty) PGL_T}). \]
This is the last contradiction in this argument, which is now familiar to us. 
Finally, we conclude $X_\infty' \cong X_\infty$, so the limit is independent of the choice of the initial metrics $\alpha$ and the subsequences $t_i$. 

Now we proceed to prove the uniqueness of the central fibers of gentle degenerations of $X$. 
Let $\mathcal{X} \to \Delta$ be a gentle degeneration. 
We have a smooth family of K\"ahler metrics $\alpha_s$ on $\mathcal{X}_s$ which extends the K\"ahler--Ricci soliton $\alpha_0$ on the central fiber $\mathcal{X}_0$, thanks to the stability argument of the K\"ahler condition in any sufficiently small deformation (see the last chapter of \cite{KodMor}). 
The uniqueness of the continuity path, proved in \cite{TZ1}, shows that $\omega_t (\alpha_0) = \alpha_0$, so we can find a sequence $t_i \to 1$ and $s_i \to 0 \in \Delta$ so that $(X, \omega_{t_i} (\alpha_{s_i}))$ converges to $(\mathcal{X}_0, \alpha_0)$. 
We can show that the sequence $(X, \omega_{t_i'} (\alpha_{s_i}))$ also converges to $(\mathcal{X}_0, \alpha_0)$ for any sequence $t_i' \to 1$ by a similar argument as above (compare \cite[Lemma 6.9. (1)]{LWX1}). 
Consider some convergent sequence $(X, \omega_{t_m} (\alpha_{s_i})) \xrightarrow{t_m \to 1} (X_{\infty,i}, \omega_{\infty, i})$ and a sequence $t_i' \to 1$ so that $d_{GH} ((X_{\infty, i}, \omega_{\infty, i}), (X, \omega_{t_i'} (\alpha_{s_i}))) < 1/i$. 
The diagonal argument shows that $(X_{\infty, i}, \omega_{\infty, i}) \to \mathcal{X}_0$. 
Since $\mathcal{X}_0$ is a smooth Fano $T$-manifold, $X_{\infty, i}$ is also smooth for large $i$. 
From what we have shown in the above argument, it follows that for any fixed K\"ahler metric $\alpha$ on $X$, we obtain $(X, \omega_t (\alpha)) \xrightarrow{t \to 1} (X_{\infty, i}, \omega_{\infty, i})$ for each $i$, so especially $(X_{\infty, i}, \omega_{\infty, i})$ are all isomorphic to each other. 
Now we conclude $(X, \omega_t (\alpha)) \xrightarrow{t \to 1} (\mathcal{X}_0, \alpha_0)$ where the limit is independent of the choice of the initial metrics $\alpha$ and the sequence is also independent of the choice of the central fibers $(\mathcal{X}_0, \alpha_0)$ of the gentle degenerations. 
So for another central fiber $(\mathcal{X}_0', \alpha_0')$ of another gentle degeneration $\mathcal{X}' \to \Delta$ of $X$, we also have $(X, \omega_t (\alpha)) \xrightarrow{t \to 1} (\mathcal{X}_0', \alpha_0')$. 
It follows that $(\mathcal{X}_0', \alpha_0')$ is isomorphic to $(\mathcal{X}_0, \alpha_0)$ from the uniqueness of the limit. 
This is what we expected. 
\end{proof}

\section{Discussions}

\subsection{On some examples}

Here we observe step by step some known examples of Fano manifolds admitting K\"ahler--Ricci solitons. 
Although the existence is known, as far as the author knows, even the associated holomorphic vector fields $\xi'$ are not explicitly gievn in almost all examples. 

\begin{exam}
The blowing-up of $\mathbb{C}P^2$ at one point is a typical example of Fano manifold admitting non-Einstein K\"ahler--Ricci solitons. 
This seems the first example of a compact complex manifold proved to admit K\"ahler--Ricci solitons, which was found by Koiso \cite{Koi} and Cao \cite{Cao}, independently. 
\end{exam}

\begin{exam}[toric Fano manifolds]
It is shown in \cite{X-J. WZ} and reproved by \cite{DatSze} from the K-stability viewpoint that every toric Fano manifold admit K\"ahler--Ricci soliton and it is K\"ahler--Einstein if and only if the barycenter of the canonical polytope coincides with the origin. 
Note that the maximal torus action on a toric Fano manifold is not necessarily K-optimal. 

Every toric Fano manifold is rigid, i.e. $H^1 (X, \Theta_X) = 0$, where $\Theta_X$ denotes the tangent sheaf (\cite[Proposition 4.2.]{BieBri}). 
It follows that toric Fano manifolds give discrete points in the moduli space $\mathcal{KR}_{GH} (n)$. 
\end{exam}

\begin{exam}[Fano homogeneous toric bundles]
It is shown in \cite{PS} and recovered in \cite{Hua} that Fano homogeneous toric bundles have K\"ahler--Ricci solitons. 
This is a generalization of the main result in \cite{X-J. WZ}. 
It is again proved in \cite[Proposition 4.2.]{BieBri} that Fano homogeneous toric bundles are rigid (see also \cite[Proposition 2.2.1.]{BieBri}, \cite[Example 3.10.]{Del}). 
\end{exam}

\begin{exam}[horospherical Fano manifolds]
\label{horo-example}
It is shown in \cite{Del} from the K-stability viewpoint and reproved by \cite{Fra} that every horospherical Fano manifold admits K\"ahler--Ricci soliton. 
This is a generalization of one of the main results in \cite{PS}. 

Horospherical Fano manifolds with Picard number one ($b^2 = 1$) are classified in \cite{Pas}. 
There is a unique horospherical Fano manifold (with an action of the complex $G_2$ group) in this classified class which admits a non-trivial small deformation. 
We can see as follows (or just by checking the criterion in \cite{Del}) that the K\"ahler--Ricci soliton on this horospherical Fano manifold $X_0$ is not K\"ahler--Einstein. 
It is shown in \cite{PP} that the Kuranishi family of this horospherical Fano manifold $X_0$ is given by an iso-trivial degeneration $\mathcal{X} \to \mathbb{C}$ of the orthogonal Grassmanian $Gr_q (2, 7)$. 
As the Grassmanian $Gr_q (2, 7)$ is homogeneous, it admits K\"ahler--Einstein metric (\cite{Mat2}). 
If $X_0$ admits K\"ahler--Einstein metric, then we get a contradiction by the separation property of K\"ahler--Einstein Fano manifolds (\cite{SSY, LWX1}) as the deformation $\mathcal{X} \to \mathbb{C}$ is iso-trivial and the general fibre admits K\"ahler--Einstein metric. 
Thus we conclude that $X_0$ cannot admit K\"ahler--Einstein metrics, while it admits K\"ahler--Ricci soliton explained as above. 

This example shows that the family $\mathcal{X} \to \mathbb{C}$ is not in the category $\K (n)$, though any fibers in the family, which are isomorphic to either $Gr_q (2,7)$ or $X_0$, admit K\"ahler--Ricci solitons. 
We have to separate them into two pieces $\mathcal{X}^* \to \mathbb{C}^*$ and $X_0 \to \{ 0 \}$ as the associated holomorphic vector fields jump at the origin. 

It seems interesting to study whether any horospherical Fano manifolds are K-rigid, which means $H^1_T (X, \Theta_X) = 0$ for a K-optimal action $X \curvearrowleft T$. 
\end{exam}

\begin{exam}[Fano manifolds with complexity one]
It is shown in \cite{IS} and \cite{CabSus} that complexity one Fano threefolds of type 2.30, 2.31, 3.8*, 3.18, 3.21, 3.22, 3.23, 3.24, 4.5* and 4.8 from Mori and Mukai's classification \cite{MM} admit non-Einstein K\"ahler--Ricci soliton. 

Especially 3.8 and 4.5 admit deformations, so $H^1_T (X, \Theta) \sslash \mathrm{Aut}_T (X)$ might be not mere a point. 
\end{exam}

The product $X \times Y$ of two Fano manifolds $X, Y$ with K\"ahler--Ricci solitons admits K\"ahler--Ricci solitons. 
So for instance, suppose $X$ is a Del Pezzo surface of degree $1 \le d \le 4$ and $Y$ is the blowing-up of $\mathbb{C}P^2$ at one point, then $X \times Y$ admits non-Einstein K\"ahler--Ricci solitons. 
By deforming $X$ while fixing $Y$, we get a $T$-equivariant deformation of $X \times Y$ where $X \times Y \curvearrowleft T$ is induced from the K-optimal action $Y \curvearrowleft T$. 
So $X \times Y$ provides a non discrete point in the moduli space $\mathcal{KR}_{GH} (n)$ outside of the subset $\mathcal{K}_{0, GH} (n)$ consisting of K\"ahler--Einstein Fano manifolds. 

Dancer--Wang's examples \cite{DW} may also provide non discrete points in the moduli space. 

\subsection{Future studies}

\subsubsection{Questions on the structure of the moduli space}

\begin{quest}
Study explicit examples of $(T, \chi)$ or $(M, \omega, T)$ with non-trivial $T$ whose moduli space $\mathcal{K}_{T, \chi}$ or $\mathcal{K} (M, \omega, T)$ has positive dimension and has a concrete description on its structure. 
\end{quest}

The author does not have any concrete description of positive dimensional moduli spaces $\mathcal{K} (M, \omega, T)$ so far. 
Related studies in the K\"ahler--Einstein case (i.e. $T=0$) are explored by \cite{OSS, SS, LiuXu}. 

\begin{quest}
Is the complex analytic space $\mathcal{K}_{T, \chi}$ actually quasi-projective?
\end{quest}

This question is related to the result in \cite{LWX2} where the quasi-projectivity of the moduli space of Fano manifolds with K\"ahler--Einstein metrics is proved. 

When $T$ is non-trivial, even the finiteness of the number of the connected components of $\mathcal{K}_{T, \chi}$ is still unknown, even though it has a natural topological compactification as a moduli space. 

\begin{quest}
Is there a canonical complex analytic structure on the compact topological space $\overline{\mathcal{KR}_{GH}} (n)$? 
How about on the space $\overline{\mathcal{KR}_{GH}} (n) \setminus \mathcal{KR}_{GH} (n)$? 
Can we identify them with algebraic spaces, or moreover with projective schemes?
\end{quest}

This is related to the work of \cite{Oda2, LWX1}. 
The techniques in this paper do not work, at least directly, in the singular setting. 

\begin{quest}
\label{singular moduli}
Is there a canonical complex analytic (or algebraic) moduli space of $\mathbb{Q}$-Fano varieties with K\"ahler--Ricci solitons? 
\end{quest}

In all questions, it seems better to investigate modified K-stability from more algebro-geometric perspectives, possibly with some help of differential geometry. 

\subsubsection{Questions related to the extent of the moduli space}

\begin{quest}
Are there any non-gentle/modified K-unstable examples of Fano manifolds with Picard number one? 
How about birationally rigid Fano manifolds with Picard number one? 
\end{quest}

This is a refined question related to the Odaka-Okada conjecture \cite{OO}. 
Two modified K-unstable examples are given in \cite{Del}, but both have the Picard number greater than one. 
The following is an optimistic conjecture towards a framework for classification of K-unstable $\mathbb{Q}$-Fano varieties. 

\begin{conj}
\label{conj}
Let $X$ be a $\mathbb{Q}$-Fano variety. 

\begin{enumerate}
\item If $X$ is not modified K-semistable, there is a (non-equivariant) $\mathbb{R}$-degeneration (cf. \cite{DerSze, CSW}) of $X$ whose central fiber is a modified K-semistable $\mathbb{Q}$-Fano variety whose H-invariant attains the infimum of the H-invariants over all $\mathbb{R}$-degenerations. 
Moreover, these degenerations are unique up to isomorphisms. 

\item If $X$ is modified K-semistable with respect to a torus action $T$, then there is a $T$-equivariant degeneration $\mathcal{X}' \to \Delta$ of $X$ whose central fiber $X_0'$ is a K-polystable $\mathbb{Q}$-Fano $T$-variety (modified K-polystable with respect to the $T$-action). 
Moreover, any two such $T$-equivariant degenerations $\mathcal{X}_1' \to \Delta, \mathcal{X}_2' \to \Delta$ are equivalent up to scaling in the sense of the $T$-equivariant version of \cite[Definition 6.1.]{BHJ}, not only they have isomorphic central fibres. 
\end{enumerate}
\end{conj}

This conjecture is related to \cite[Conjecture 3.7.]{CSW} and is an analogy of the Harder-Narasimhan filtration for torsion-free coherent sheaves and the Jordan-H\"older filtration for semistable coherent sheaves (see \cite{HL}) as already observed in \cite[Remark 3.6.]{DerSze}. 
We include the singular case for the future application to the Question \ref{singular moduli}. 

For the first item, \cite{CSW} shows that every smooth Fano manifold $X$ has an $\mathbb{R}$-degeneration with the $\mathbb{Q}$-Fano central fiber $X_0$ and there is an another degeneration $\mathcal{X}' \to \Delta$ of $X_0$ with the modified K-polystable $\mathbb{Q}$-Fano central fiber $X_0'$ with the K-optimal vector $\xi'$, which can be extended to $X_0$ with the vanishing modified Futaki invariant (see also \cite{DerSze}). 
So as for the existence, it suffices to prove the modified K-semistability of $X_0$. 
Since $(X_0', \xi')$ is K-polystable, the problem is reduced to the `stability of K-semistability in small deformations', which is related to the Artinianity of the $\Cpx$-stack consisting of K-semistable $\mathbb{Q}$-Fano $T$-varieties, as in Proposition \ref{Artin stack}. 
It is remarkable that if $X$ is K-unstable (with respect to the trivial torus action), then $X_0'$ must be endowed with non-Einstein K\"ahler--Ricci solitons (\cite[p. 17]{CSW}). 

The existence part of the second item is confirmed in \cite{DatSze} for smooth modified K-semistable Fano $T$-manifolds. 
The uniqueness of the central fiber in this case could be demonstrated by the same methods in \cite{LWX1}, which is a role model of our proof of Proposition \ref{gentle degeneration}. 
(We worked with the smooth central fiber because the author felt that it makes arguments clear. ) 

The uniqueness assertion in the second item is stronger than the uniqueness of the central fiber.
This stronger uniqueness (for every smooth gentle Fano $T$-manifold $X$) has the following application. 

\begin{cor}[of the uniqueness statement of Conjecture \ref{conj} (2)]
\label{expected}
The moduli space $\K_{T, \chi} \to \mathcal{K}_{T, \chi}$ we constructed in Theorem \ref{main theorem} is good in the sense of Alper \cite{Alp1}. (In our case, the cohomological affineness should be defined as the exactness of the push-forward functor $\mathit{Coh} (\K_{T, \chi}) \to \mathit{Mod} (\mathcal{K}_{T, \chi})$. ) 
\end{cor}

Actually, using the uniqueness of the degeneration in the sense of \cite{BHJ}, we can show that the \'etale morphism $[B/R] \to \K_{T, \chi}$ is an open embedding. 
Then the corollary follows from the fact that $[BK^c /K^c] \to BK^c \sslash K^c$ is a good moduli space. 
Recall that we have already shown the central fiber of the degeneration is unique, which we used to prove that the morphism $[B/R] \to BK^c \sslash K^c$ factors through $\mathrm{Im} [B/R] \subset \K_{T, \chi}$. 
There may be other ways to show this naturally expected corollary. 

\section{Appendix}

\subsection*{A. $\Cpx$-stack}

In this Appendix A, we briefly review some general notions and examples of stacks which we used in section 4. 
As we work only over the category (or more precisely, the site) $\Cpx$ of complex spaces, we do not introduce stacks in full generality, which actually work over any site such as the \'etale sites of schemes or algebraic spaces, the site of $C^\infty$-manifolds and so on. 
The interested readers should also refer to \cite{SPA, FGA} for stacks in full generality. 

\subsubsection*{\textbf{A-1.} Fibred category}

Recall that we denote by $\Cpx$ the category of complex spaces, which are not assumed to be reduced nor irreducible. 
The set of holomorphic morphisms between complex spaces $U$ and $V$ is denoted by $\Holo (U, V)$. 

\begin{defin}[fibred category]
Let $\bm{\F}$ be a category and $p :\bm{\F} \to \Cpx$ be a functor to the category of complex spaces. 
The functor $p :\bm{\F} \to \Cpx$ is called a \textit{fibred category} over $\Cpx$ if it satisfies the following properties. 
For any holomorphic morphism $f: X \to Y$ between complex spaces and any object $\eta \in \Obj (\bm{\F})$, there exists an object $\xi \in \Obj (\bm{\F})$ and a strongly cartesian morphism $\phi :\xi \to \eta \in \Mor (\bm{\F})$ over $f$. 
\end{defin}

Here the morphism $\phi: \xi \to \eta$ is called \textit{strongly cartesian} if it enjoys the following universal property: for any complex space $X'$, any holomorphic morphism $g :X' \to X$, any object $\xi' \in \Obj (\bm{\F})$ with $p (\xi') = X'$ and any morphism $\phi' :\xi' \to \eta \in \Mor (\bm{\F})$ with $p (\phi') = f \circ g$, there exists a unique morphism $\chi :\xi' \to \xi$ such that $\phi' = \phi \circ \chi$ and $p (\chi) = g$. 
\[ \begin{tikzcd} \bm{\F} \ar{d}[description]{p} & \xi' \ar[bend left=20]{rr}{\phi'} \ar[dashed]{r}[description]{\chi} & \xi \ar{r}[description]{\phi} & \eta \\ \Cpx & X' \ar{r}[description]{g} \ar[bend left=20]{rr}{f \circ g} & X \ar{r}[description]{f} & Y \end{tikzcd} \]

Let $X$ be a complex space and $p :\bm{\F} \to \Cpx$ be a fibred category. 
We denote by $\bm{\F} (X)$ the subcategory of $\bm{\F}$ consisting of objects $\xi \in \Obj (\bm{\F})$ with $p (\xi) = X$ and morphisms $\phi$ with $p (\phi) = \id_X$. 
We call $\bm{\F}$ (or more precisely $\bm{\F} \to \Cpx$) a \textit{category fibred in groupoids} if morphisms in $\bm{\F} (X)$ are all invertible for any complex space $X$. 

A functor $f :\bm{\F} \to \bm{\G}$ between two fibred categories is called a \textit{morphism} of fibred categories if $p_{\bm{\F}} = p_{\bm{\G}} \circ f$ (strictly) and $f$ maps strongly cartesian morphisms in $\bm{\F}$ to strongly cartesian morphisms in $\bm{\G}$. 
We can also consider \textit{$2$-morphisms} between two ($1$-)morphisms $f, g: \bm{\F} \to \bm{\G}$ which are just natural transformations $t :f \to g$ satisfying $p_{\bm{\G}} (t_\xi: f (\xi) \to g (\xi)) = \id_{p_{\bm{\F}} (\xi)}$ for all $\xi \in \Obj (\bm{\F})$. 

The functor $\Cpx_X \to \Cpx : (\xi: S \to X) \mapsto S$, where $\Cpx_X$ denotes the category of holomorphic morphisms $\xi :S \to X$, is a typical example of category fibred in groupoids (actually in sets). 
A holomorphic morphism of complex spaces $f :X \to Y$ gives the morphism $\Cpx_X \to \Cpx_Y$ which maps an object $\xi: S \to X$ to the object $f \circ \xi: S \to Y$. 
On the other hand, a morphism $f :\Cpx_X \to \Cpx_Y$ as fibred categories gives a holomorphic morphism $f (\id_X) :X \to Y$. 
Therefore, we have a canonical fully faithful embedding of $\Cpx$ to the ($2$-)category of fibred categories. 
So we often abbreviate $\Cpx_X$ as $X$. 

\begin{exam}
Let $a :X \times G \to X$ be a holomorphic action of a complex Lie group $G$ to a complex space $X$. 
We denote by $[X/G]_p$\footnote{The symbol $_p$ means that this fibred category is not a stack in general; it is just a pre-stack. } the fibred category (in groupoids) defined as follows. 
\begin{enumerate}
\item Its objects are holomorphic morphisms $\xi :S \to X$ from some complex spaces $S$. 

\item Its morphisms $\xi_S \to \eta_T$ are the pairs $(f, \phi)$ of holomorphic morphisms $f :S \to T$ and $\phi :S \to X \times G$ satisfying $p_1 \circ \phi = \xi$ and $a \circ \phi = \eta \circ f$. 

\item Its functor $[X/G]_p \to \Cpx$ maps objects $\xi_S$ to $S$ and morphisms $(f, \phi) :\xi_S \to \eta_T$ to $f :S \to T$. 
\end{enumerate}
\end{exam}

Objects in the fibred category $[X/G]_p$ coincide with $X = \Cpx_X$, but morphisms are different. 
For instance, two objects $x, y :\mathrm{pt} \to X$ in $[X/G]_p$ are isomorphic if and only if there exists an element $g \in G$ with $xg = y$. 
We have the morphism $X \to [X/G]_p$ of fibred categories defined by $\xi_S \mapsto \xi_S$. 

There is another related fibred category $\llbracket X/G \rrbracket$ with a good geometric feature. 

\begin{exam}
\label{quotient stack}
We denote by $\llbracket X/G \rrbracket$ the fibred category (in groupoids) defined as follows. 
\begin{enumerate}
\item An object consists of a triple $(S, P, \xi)$ where $S$ is a complex space, $P$ is a principal $G$-holomorphic bundle over $S$ and $\xi :P \to X$ is a $G$-equivariant holomorphic morphism. 

\item A morphism $(S, P, \xi) \to (T, Q, \eta)$ is a pair $(f, \phi)$ where $f :S \to T$ is a holomorphic morphism and $\phi: P \to Q$ is a $G$-equivariant holomorphic morphism over $f$ which induces an biholomorphism $P \cong S \times_T Q$, and satisfies $\xi = \xi' \circ \phi$. 

\item Its functor $\llbracket X/G \rrbracket \to \Cpx$ maps objects $(S, P, \xi)$ to $S$ and morphisms $(f, \phi) :(S, P, \xi) \to (T, Q, \eta)$ to $f :S \to T$. 
\end{enumerate}
\end{exam}

We have the morphism $[X/G]_p \to \llbracket X/G \rrbracket$ of fibred categories which maps an object $\xi :S \to X$ to the object $(S, S \times G, a \circ (\xi \times \id_G))$ and a morphism $(f, \phi): \xi_S \to \eta_T$ to the morphism $(f, f \times \phi)$. 
This is a typical example of the `stackification' we treat in the next subsection. 
The fibred category $\llbracket X/G \rrbracket$ is called a \textit{quotient stack}. 

When the action is proper free, then there exists a complex space $X / G$, a holomorphic morphism $X \to X/G$ and an isomorphism $X/G \cong \llbracket X/G \rrbracket$ of fibred categories, which is compatible with $X \to X/G$ and $X \to \llbracket X/G \rrbracket$. 

Let us see another example generalizing $[X/G]_p$. 
A \textit{holomorphic groupoid} consists of the following data $(X, R, s, t, c)$: 
\begin{enumerate}
\item $X$ and $R$ are complex spaces. 

\item $s$ and $T$ are holomorphic morphisms from $R$ to $X$. 

\item $c : R \times_{s, X, t} R \to R$ is a holomorphic morphism. 
\end{enumerate}
These data are to satisfy the following rules for any complex space $S$: 
\begin{enumerate}
\item For every holomorphic morphism $\xi \in \Holo (S, X)$, there exists a holomorphic morphism $e_\xi \in \Holo (S, R)$ such that $c \circ (e_\xi \times \phi) = \phi$ and $c \circ (\psi \times e_\xi) = \psi$ for any pairs $(e_\xi, \phi)$, $(\psi, e_\xi)$ with $s \circ e_\xi = t \circ \phi$ and $s \circ \psi = t \circ e_\xi$. 

\item The equality $c \circ ((c \circ (\phi \times \psi)) \times \chi) = c \circ (\phi \times (c \circ (\psi \times \chi)))$ holds for any $\phi, \psi, \chi \in \Holo (S, R)$ with $s \circ \phi = t \circ \psi$ and $s \circ \psi = t \circ \chi$. 

\item For any $\phi \in \Holo (S, R)$, there exists a $\psi \in \Holo (S, R)$ such that $s \circ \phi = t \circ \psi = \xi$, $s \circ \psi = t \circ \phi = \eta$ and $c \circ (\phi \times \psi) = e_\eta$, $c \circ (\psi \times \phi) = e_\xi$. 
\end{enumerate}
This condition is equivalent to say that $\Holo (S, X)$ forms an abstract groupoid whose morphisms $\xi \to \eta$ are $\phi \in \Holo (S, R)$ with $s \circ \phi = \xi$ and $t \circ \phi = \eta$, and composition is given by $c$. 

A holomorphic group action $a :X \times G \to X$ gives an example of holomorphic groupoid with $R = X \times G$, $s = p_1$, $t = a$ and $c = \id \times \mu :X \times G \times G \to X \times G$. 
If $u: U \to X$ is a holomorphic morphism, then we can consider the pull-back holomorphic groupoid $(U, (U \times U) \times_{u \times u, X \times X, s \times t} R, s', t', c')$. 

\begin{exam}
We denote by $[X/R]_p$ the fibred category (in groupoids) defined as follows. 
\begin{enumerate}
\item Its objects are holomorphic morphisms $\xi :S \to X$ from some complex spaces $S$. 

\item Its morphisms $\xi_S \to \eta_T$ are the pairs $(f, \phi)$ of holomorphic morphisms $f :S \to T$ and $\phi :S \to R$ satisfying $s \circ \phi = \xi$ and $t \circ \phi = \eta \circ f$. 

\item Its functor $[X/R]_p \to \Cpx$ maps objects $\xi_S$ to $S$ and morphisms $(f, \phi): \xi_S \to \eta_T$ to $f: S \to T$. 
\end{enumerate}
\end{exam}

Here is our interested fibred category from Definition \ref{moduli stack}. 

\begin{lem}
\label{moduli category}
The category $\K_{T, \chi}$ and $\K^s_{T, \chi}$ forms a fibred category by the functor $\K_{T, \chi}^{(s)} \to \Cpx: (\pi: \M \to S, \alpha) \mapsto S$. 
\end{lem}

This is just because the following cartesian diagram gives a cartesian morphism for any holomorphic morphism $f :X \to Y$ between complex spaces and any object $(\pi :\M \to Y, \alpha) \in \K^{(s)}_{T, \chi}$. 
\[ \begin{tikzcd} X \times_Y \M \ar{d} \ar{r} & \M \ar{d}{\pi} \\ X \ar{r}{f} & Y \end{tikzcd} \]

The correspondence $(\pi :\M \to Y, \alpha) \mapsto (f^* \pi: X \times_Y \M \to X, f^* \alpha)$ gives a functor $\K_{T, \chi} (Y) \to \K_{T, \chi} (X)$. 
It looks like that this provides a functor $X \mapsto \K_{T, \chi} (X)$ from the category $\Cpx$ to the ``category'' of groupoids, but actually does not. 
This nuisance comes from the set theoretical fact that $X \times_{f, Y} (Y \times_{g, Z} \M) \neq X \times_{g \circ f, Z} \M$; they are not exactly the same objects but just naturally isomorphic. 
This is the reason why we should formulate things in terms of fibred category. 

\subsubsection*{\textbf{A-2.} Descent data}

We introduce descent data of a fibred category over $\Cpx$. 

\begin{defin}[descent data]
Let $p :\bm{\F} \to \Cpx$ be a fibred category, $X$ be a complex space, $\U := \{ i_\alpha :U_\alpha \hookrightarrow X \}_{\alpha \in A}$ be an open cover of $X$ (in the real topology). 
We denote by $u_{\alpha, \beta} :U_\alpha \cap U_\beta \hookrightarrow U_\alpha$ the inclusion morphism to the first factor and by $u_{\alpha \beta, \gamma} :U_\alpha \cap U_\beta \cap U_\gamma \hookrightarrow U_\alpha \cap U_\beta$ the inclusion morphism to the intersection of the first and second factor ($u_{\alpha \beta, \gamma}= u_{\beta \alpha, \gamma}$). 
Put $A_2 := A \times A/\mathfrak{S}_2$, $A_3 := A \times A \times A/\mathfrak{S}_3$. 

A \textit{descent datum} of $\bm{\F}$ over $(X, \U)$ consists of the following data $\mathfrak{D} = (\Xi_1, \Xi_2, \Xi_3, \Theta_2, \Theta_3)$: 
\[ \begin{split} \Xi_1 &:= \{ \xi_\alpha \in \bm{\F} (U_\alpha) \}_{\alpha \in A} \\ \Xi_2 &:= \{ \xi_{\alpha \beta} \in \bm{\F} (U_\alpha \cap U_\beta) \}_{\{ \alpha, \beta \} \in A_2} \\ \Xi_3 &:= \{ \xi_{\alpha \beta \gamma} \in \bm{\F} (U_\alpha \cap U_\beta \cap U_\gamma) \}_{\{ \alpha, \beta, \gamma \} \in A_3} \end{split} \]
are sets of objects in $\bm{\F}$
and
\[ \begin{split} \Theta_2 &:= \{ \theta_{\alpha, \beta} : \xi_{\alpha \beta} \to \xi_\alpha ~|~ \theta_{\alpha, \beta} \text{ is cartesian over } u_{\alpha, \beta} \}_{(\alpha, \beta) \in A^2} \\ \Theta_3 &:= \{ \theta_{\alpha \beta, \gamma} :\xi_{\alpha \beta \gamma} \to \xi_{\alpha \beta} ~|~ \theta_{\alpha \beta, \gamma} \text{ is cartesian over } u_{\alpha \beta, \gamma} \}_{(\{ \alpha, \beta \}, \gamma) \in A_2 \times A} \end{split} \]
are sets of cartesian morphisms in $\bm{\F}$. 
These data must satisfy
\[ \theta_{\alpha, \beta} \circ \theta_{\alpha \beta, \gamma} = \theta_{\alpha, \gamma} \circ \theta_{\gamma \alpha, \beta} \]
for any $\alpha, \beta, \gamma \in A$. 
\[ \begin{tikzcd} ~ & \bm{\F} \ar{rrrr}{p} & ~ & ~ & ~ & \Cpx & ~ & ~ \\ 
\xi_{\alpha \beta \gamma} \ar{rr}[description]{\theta_{\beta \gamma, \alpha}} \ar{dd}[description]{\theta_{\gamma \alpha, \beta}} \ar{dr}[description]{\theta_{\alpha \beta, \gamma}} & ~ & \xi_{\beta \gamma} \ar[gray]{dd}{\theta_{\gamma, \beta}} \ar{dr}{\theta_{\beta, \gamma}} & ~ & U_{\alpha \beta \gamma} \ar{rr}[description]{u_{\beta \gamma, \alpha}} \ar{dd}[description]{u_{\gamma \alpha, \beta}} \ar{dr}[description]{u_{\alpha \beta, \gamma}} & ~ & U_{\beta \gamma} \ar[gray]{dd}{u_{\gamma, \beta}} \ar{dr}[description]{u_{\beta, \gamma}} & ~ \\ 
~ & \xi_{\alpha \beta} \ar[crossing over]{rr}[description]{\theta_{\beta, \alpha}} & ~ & \xi_\beta &~ & U_{\alpha \beta} \ar[crossing over]{rr}[description]{u_{\beta, \alpha}} & & U_\beta \ar[gray!30]{dd}[description]{i_\beta}  \\ 
\xi_{\gamma \alpha} \ar[gray]{rr}{\theta_{\gamma, \alpha}} \ar{dr}[description]{\theta_{\alpha, \gamma}} & & {\color[gray]{0.5} \xi_\gamma} & ~ & U_{\gamma \alpha} \ar[gray]{rr}{u_{\gamma, \alpha}} \ar{dr}[description]{u_{\alpha, \gamma}} & ~ & {\color[gray]{0.5} U_\gamma} \ar[gray!30]{dr}[description]{i_\gamma} & ~ \\ 
~ & \xi_\alpha \ar[<-, crossing over]{uu}[description]{\theta_{\alpha, \beta}} & ~ & ~ &~ & U_\alpha \ar[<-, crossing over]{uu}[description]{u_{\alpha, \beta}} \ar[gray!30]{rr}[description]{i_\alpha} & ~ & {\color[gray]{0.7} X} \end{tikzcd} \]
\end{defin}

A descent datum $\mathfrak{D} = (\Xi_1, \Xi_2, \Xi_3, \Theta_2, \Theta_3)$ is called \textit{effective} if there exists an object $\xi \in \bm{\F} (X)$ and a set of morphisms 
\[ \Theta_1 := \{ \theta_\alpha : \xi_\alpha \to \xi ~|~ \theta_\alpha \text{ is cartesian over } i_\alpha \}_{\alpha \in A} \]
satisfying
\[ \theta_\alpha \circ \theta_{\alpha, \beta} = \theta_\beta \circ \theta_{\beta, \alpha} \]
for any $\alpha, \beta \in A$. 
We define an \textit{effective descent datum} of $\bm{\F}$ over $(X, \U)$ to be an object consisting of data $\mathfrak{D}_+ = (\mathfrak{D}, \xi, \Theta_1) = (\xi, \Xi_1, \Xi_2, \Xi_3, \Theta_1, \Theta_2, \Theta_3)$ as above. 

\begin{rem}
Note that
\begin{itemize}
\item Every descent datum of $\Cpx_X$ is effective. 

\item There are descent data of $[X/G]_p$ which are not effective, in general. 

\item Every descent datum of $\llbracket X/G \rrbracket$ is effective.  
\end{itemize}

As for the second item, consider the fibred category $[(\mathbb{C}^2 \setminus \{ 0 \})/\mathbb{C}^*]_p$ for example. 
More explicitly, let $\U$ be an open cover of $\mathbb{C}P^1$ defined by two open subsets $U_\alpha := \{ (z_1: z_2) ~|~ z_2 \neq 0 \}, U_\beta := \{ (z_1: z_2) ~|~ z_1 \neq 0 \}$ and let $\xi_\alpha: U_\alpha \to \mathbb{C}^2 \setminus \{ 0 \}, \xi_\beta: U_\beta \to \mathbb{C}^2 \setminus \{ 0 \}$ be morphisms defined by $\xi_\alpha (z_1: z_2) := (z_1/z_2, 1), \xi_\beta (z_1: z_2) := (1, z_2/z_1)$, respectively. 
Consider a descent datum over $(\mathbb{C}P^1, \U)$ with $\Xi_1 := \{ \xi_\alpha, \xi_\beta \}$ given by an obvious way. 
In order to be effective, this descent datum should define a non-constant morphism $\mathbb{C}P^1 \to \mathbb{C}^2 \setminus \{ 0 \}$ which is \textit{isomorphic} (not equal) to $\xi_\alpha, \xi_\beta$ when restricted to each open set, but this is impossible as every holomorphic map $\mathbb{C}P^1 \to \mathbb{C}^2 \setminus \{ 0 \}$ is constant. 
So this descent datum is not effective in this fibred category. 

On the other hand, the corresponding descent datum in the fibred category $\llbracket (\mathbb{C}^2 \setminus \{ 0 \})/\mathbb{C}^* \rrbracket$ becomes effective, completed by the object $\mathbb{C}P^1 \xleftarrow{\mathbb{C}^*} (\mathbb{C}^2 \setminus \{ 0 \}) \xrightarrow{\id} (\mathbb{C}^2 \setminus \{ 0 \})$ in $\llbracket (\mathbb{C}^2 \setminus \{ 0 \})/\mathbb{C}^* \rrbracket$. 
Actually, $\llbracket (\mathbb{C}^2 \setminus \{ 0 \}) / \mathbb{C}^* \rrbracket$ is isomorphic to $\mathbb{C}P^1$ as fibred categories. 
\end{rem}

\begin{defin}[morphism of descent data]
Let $\mathfrak{D} = (\Xi_1, \Xi_2, \Xi_3, \Theta_2, \Theta_3)$, $\mathfrak{D}' = (\Xi_1', \Xi_2', \Xi_3', \Theta_2', \Theta_3')$ be two descent data of $\bm{\F}$ over $(X, \U)$. 
A \textit{morphism} from $\mathfrak{D}$ to $\mathfrak{D}'$ is a triple $\Phi = (\Phi_1, \Phi_2, \Phi_3)$ of sets of morphisms
\[ \begin{split} \Phi_1 &:= \{ \phi_\alpha : \xi_\alpha \to \xi'_\alpha \in \bm{\F} (U_\alpha) ~|~ \xi_\alpha \in \Xi_1, \xi_\alpha' \in \Xi'_1 \}_{\alpha \in A} \\ \Phi_2 &:= \{ \phi_{\alpha \beta} :\xi_{\alpha \beta} \to \xi'_{\alpha \beta} \in \bm{\F} (U_\alpha \cap U_\beta) ~|~ \xi_{\alpha \beta} \in \Xi_2, \xi_{\alpha \beta}' \in \Xi_2' \}_{\{ \alpha, \beta \} \in A_2} \\ \Phi_3 &:= \{ \phi_{\alpha \beta \gamma} : \xi_{\alpha \beta \gamma} \to \xi_{\alpha \beta \gamma}' \in \bm{\F} (U_\alpha \cap U_\beta \cap U_\gamma) ~|~ \xi_{\alpha \beta \gamma} \in \Xi_3, \xi'_{\alpha \beta \gamma} \in \Xi_3' \}_{ \{ \alpha, \beta, \gamma \} \in A_3} \end{split} \]
in $\bm{\F}$ satisfying
\[ \phi_\alpha \circ \theta_{\alpha, \beta} = \theta'_{\alpha, \beta} \circ \phi_{\alpha \beta} ~\text{ and }~ \phi_{\alpha \beta} \circ \theta_{\alpha \beta, \gamma} = \theta'_{\alpha, \beta, \gamma} \circ \phi_{\alpha \beta \gamma} \]
for all $\alpha, \beta, \gamma \in A$. 
Descent data naturally form a category with these morphisms. 
We denote by $\bm{\F}_\des (X, \U)$ the category of descent data. 
\end{defin}

Let $\mathfrak{D}_+ = (\mathfrak{D}, \xi, \Theta_1)$, $\mathfrak{D}'_+ = (\mathfrak{D}', \xi', \Theta_1')$ be two effective descent data of $\bm{\F}$ over $(X, \U)$. 
A \textit{morphism} from $\mathfrak{D}_+$ to $\mathfrak{D}_+'$ is a quadruple $\Phi_+ = (\phi, \Phi_1, \Phi_2, \Phi_3)$ where $(\Phi_1, \Phi_2, \Phi_3)$ gives a morphism of corresponding descent data and $\phi :\xi \to \xi'$ is a morphism in $\bm{\F} (X)$ satisfying 
\[ \phi \circ \theta_\alpha = \theta'_\alpha \circ \phi_\alpha \]
for any $\alpha \in A$. 
We denote by $\bm{\F}_\eff (X, \U)$ the category of effective descent data. 

\subsubsection*{\textbf{A-3.} $\Cpx$-stack}

We can consider the forgetful functors $\bm{\F}_\eff (X, \U) \to \bm{\bm{\F}}_\des (X, \U)$ defined by $\mathfrak{D}_+ = (\mathfrak{D}, \xi, \Theta_1) \mapsto \mathfrak{D}$ and $\bm{\F}_\eff (X, \U) \to \bm{\F} (X)$ defined by $\mathfrak{D}_+ = (\mathfrak{D}, \xi, \Theta_1) \mapsto \xi$. 
The latter functor $\bm{\F}_\eff (X, \U) \to \bm{\F} (X)$ is fully faithful and essentially surjective. 
Therefore there is an inverse functor $\bm{\F} (X) \to \bm{\F}_\eff (X, \U)$ (assuming the axiom of global choice). 
As for our fibred category $\K_{T, \chi}$, there is a canonical choice\footnote{This is well-defined because $U_\alpha \cap U_\beta = U_\beta \cap U_\alpha$ as complex spaces, in particular as sets. On the other hand, $U_\alpha \times_X U_\beta \neq U_\beta \times_X U_\alpha$ as sets, though they are canonically isomorphic, because of the set theoretical fact $(a, b) = \{ \{ a \}, \{ a, b \} \} \neq \{ \{ b \}, \{ b, a \} \} = (b, a)$. } of the inverse functor defined by 
\[ (\pi: \M \to S) \mapsto (\M \to S, \{ \pi^{-1} (U_\alpha) \to U_\alpha \}_\alpha, \{ \pi^{-1} (U_\alpha \cap U_\beta) \to U_\alpha \cap U_\beta \}_{\alpha, \beta}, \ldots). \]
However, in general there is no canonical choice of this inverse functor; there needs an additional choice of $(\mathfrak{D}, \Theta_1)$ compatible to $\xi$, which is not unique as object but unique only up to isomorphisms. 

\begin{defin}[$\Cpx$-stack]
\label{cpx-stack}
A fibred category $p :\bm{\F} \to \Cpx$ is called a \textit{stack over $\Cpx$} or just \textit{$\Cpx$-stack} if it satisfies the following two conditions for any complex space $X$ and any open cover $\U$ of $X$. 
\begin{enumerate}
\item The functor $\bm{\F}_\eff (X, \U) \to \bm{\F}_\des (X, \U)$ is fully faithful. 

\item The functor $\bm{\F}_\eff (X, \U) \to \bm{\F}_\des (X, \U)$ is essentially surjective. 
\end{enumerate}
\end{defin}

\begin{rem}
If we have a choice of pull back $f \mapsto f^* \xi$ with a morphism $f^* \xi \to \xi$ so that it is cartesian over $f :S \to X$, we can consider a contravariant functor defined by
\[ \mathit{Mor}_X (\xi, \eta) :\Cpx_X^\op \to \Sets :(f: S \to X) \mapsto \Hom_{\bm{\F} (S)} (f^* \xi, f^* \eta), \]
where $\Cpx_X^\op$ stands for the opposite category of $\Cpx_X$. 
Then the first condition of the above definition is equivalent to say that the functor $\mathit{Mor}_X (\xi, \eta)$ is a sheaf on the site $\Cpx_X$. 

It is customary to denote by $\mathit{Isom}_X (\xi, \eta)$ the functor $\mathit{Mor}_X (\xi, \eta)$ when $\bm{\F}$ is a category fibred in groupoids, as every morphism in $\bm{\F} (X)$ is an isomorphism. 

We can consider a related fibred category $\bm{\mathcal{M}or}_X (\xi, \eta)$ (in setoids) without a choice of pull back. 
The category consists of objects $(f_S, \phi_\xi, \phi_\eta)$ where $f :S \to X$ is a holomorphic morphism of complex spaces and $\phi_\xi: \xi_S \to \xi, \phi_\eta :\eta_S \to \eta \in \Mor (\bm{\F})$ are cartesian arrows over $f$. 
Its morphisms $(f_S, \phi_\xi, \phi_\eta) \to (f_{S'}, \phi_\xi', \phi_\eta')$ are triples $(g, \psi_\xi, \psi_\eta)$ where $g :S \to S'$ is a holomorphic morphism of complex spaces and $\psi :\xi_S \to \xi_{S'}, \psi_\eta :\eta_S \to \eta_{S'}$ are cartesian arrows over $g$ satisfying $f_S = f_{S'} \circ g$, $\phi_\xi = \phi_\xi' \circ \psi_\xi$ and $\phi_\eta = \phi_\eta' \circ \psi_\eta$. 
\end{rem}

For any fibred category, we can always associate a stack in a canonical way. 
Here is the fact from \cite[TAG 02ZN, 0435]{SPA}. 

\begin{prop}
Let $p :\bm{\F} \to \Cpx$ be a fibred category. 
Suppose we have a choice of pull back $(f, \xi) \mapsto f^* \xi$ (just for simplicity). 
Then there exists a $\Cpx$-stack $\bm{\F}'$ (with a choice of pull back) and a morphism $s: \bm{\F} \to \bm{\F}'$ of fibred categories with the following properties. 
\begin{enumerate}
\item For every complex space $X$ and any $\xi, \eta \in \Obj (\bm{\F} (X))$, the morphism of presheaf $\mathit{Mor}_X (\xi, \eta) \to \mathit{Mor}_X (s (\xi), s (\eta))$ is a sheafification of $\mathit{Mor}_X (\xi, \eta)$. 

\item For every complex space $X$ and any $\xi' \in \Obj (\bm{\F}' (X))$, there exists an open cover $\U = \{ i_\alpha :U_\alpha \to X \}_\alpha$ of $X$ such that $i_\alpha^* \xi'$ is isomorphic to $s (\xi_\alpha)$ for some $\xi_\alpha \in \Obj (\bm{\F})$ for every $\alpha$. 

\item Given a $\Cpx$-stack $\bm{\G}$ and a morphism $g :\bm{\F} \to \bm{\G}$ of fibred categories, there exists a morphism $g': \bm{\F}' \to \bm{\G}$ of fibred categories such that there exists a $2$-isomorphism between $g$ and $g' \circ s$. 
\end{enumerate}
\end{prop}

The last property actually follows from the first two properties.  
A stack $\bm{\F}'$ with the last property is called a \textit{stackification} of $\bm{\F}$ and the stack $\bm{\F}'$ constructed in the proof of this proposition as \textit{the stackification} of $\bm{\F}$ (a fixed construction is in mind). 
We denote by $[X/G]$, $[X/R]$ the stackification of the fibred category $[X/G]_p$, $[X/R]_p$ respectively and call them \textit{the quotient stack}. 
The stack $\llbracket X/G \rrbracket$ is a stackification of the fibred category $[X/G]_p$, so it is (canonically) isomorphic to the stackification $[X/G]$. 

A 2-fibre product \cite[003Q]{SPA} of fibred categories can be calculated as follows. 
We refer to this construction as \textit{the 2-fibre product of fibred categories}. 

\begin{prop}
Let $\bm{\F}, \bm{\G}, \bm{\mathcal{H}}$ be fibred categories over $\Cpx$ and $f :\bm{\F} \to \bm{\mathcal{H}}$, $g: \bm{\G} \to \bm{\mathcal{H}}$. 
The fibred category $\bm{\mathcal{E}}$ defined as follows enjoys the universal property of $2$-fibre product. 
\begin{enumerate}
\item An object of $\bm{\mathcal{E}}$ is a quadruple $(X, \xi, \eta, \phi)$ where $X$ is a complex space, $\xi$ is an object in $\bm{\F} (X)$, $\eta$ is an object in $\bm{\G} (X)$ and $\phi : f (\xi) \to g (\eta)$ is an isomorphism in $\bm{\mathcal{H}} (X)$. 

\item A morphism $(X, \xi, \eta, \phi) \to (Y, \xi', \eta', \phi')$ is a pair $(\sigma, \tau)$ where $\sigma: \xi \to \xi'$ is a morphism in $\bm{\F}$ and $\tau :\eta \to \eta'$ is a morphism in $\bm{\G}$ satisfying $p (\sigma) = p (\tau) :X \to Y$ and $g (\tau) \circ \phi = \phi' \circ f (\sigma)$. 
\end{enumerate}
\end{prop}

When $\bm{\F}, \bm{\G}, \bm{\mathcal{H}}$ are all $\Cpx$-stacks, the stackification of the fibred category $\bm{\mathcal{E}}$ is denoted by $\bm{\F} \times_{f, \bm{\mathcal{H}}, g} \bm{\G}$. 
The $\Cpx$-stack $\bm{\F} \times_{f, \bm{\mathcal{H}}, g} \bm{\G}$ satisfies the universal property of $2$-fibre product in the $2$-category of $\Cpx$-stacks. 

The following verifying process might help the readers' better understanding of the notion of descent. 
See Definition \ref{moduli stack} for the definition of $\K_{T, \chi}, \K_{T, \chi}^s$. 

\begin{lem}
\label{moduli can stack}
The fibred categories $\K_{T, \chi}$, $\K_{T, \chi}^s$ are $\Cpx$-stacks. 
\end{lem}

\begin{proof}
For abbreviation, we let $\M$ stand for $(\pi :\M \to S, \alpha) \in \K_{T, \chi}$. 
Let $S$ be a complex space, $\U = \{ U_\alpha \}_\alpha$ be an open cover of $S$ and $\mathfrak{D} = (\Xi_1, \Xi_2, \Xi_3, \Theta_2, \Theta_3)$ be a descent datum of $\K_{T, \chi}$ over $(X, \U)$. 
Since $\theta_{\alpha, \beta} \in \Theta_2$ is cartesian, it induces an isomorphism
\[ \tilde{\theta}_{\alpha, \beta} : \M_{\alpha \beta} \xrightarrow{\sim} \M_\alpha|_{U_\alpha \cap U_\beta}. \]
So we obtain an isomorphism
\[ \theta'_{\beta \alpha} := \tilde{\theta}_{\beta, \alpha} \circ \tilde{\theta}_{\alpha, \beta}^{-1} : \M_\alpha|_{U_\alpha \cap U_\beta} \xrightarrow{\sim} \M_\beta|_{U_\alpha \cap U_\beta}. \]

Similarly, we obtain an isomorphism
\[(\tilde{\theta}_{\beta, \alpha}|_{U_{\alpha \beta \gamma}} \circ \tilde{\theta}_{\beta \alpha, \gamma}) \circ (\tilde{\theta}_{\alpha, \beta}|_{U_{\alpha \beta \gamma}} \circ \tilde{\theta}_{\alpha \beta, \gamma})^{-1} : \M_\alpha|_{U_{\alpha \beta \gamma}} \xrightarrow{\sim} \M_\beta|_{U_{\alpha \beta \gamma}} \]
which we denote by $\theta'_{\beta \alpha, \gamma}$, from the cartesian arrow $\theta_{\alpha \beta, \gamma} \in \Theta_3$. 

From the condition $\theta_{\alpha, \beta} \circ \theta_{\alpha \beta, \gamma} = \theta_{\alpha, \gamma} \circ \theta_{\gamma \alpha, \beta}$, we obtain $\theta'_{\gamma \beta, \alpha} \circ \theta'_{\beta \alpha, \gamma} = \theta'_{\gamma \alpha, \beta}$ and $\theta'_{\beta \alpha}|_{U_\alpha \cap U_\beta \cap U_\gamma} = \theta'_{\beta \alpha, \gamma}$. 
So we can glue $\M_\alpha$ together by gluing maps $\theta'_{\beta \alpha}$ and obtain a complex space $\M$ with a natural set of morphisms $\Theta_1 := \{ \theta_\alpha :\M_\alpha \to \M \}_{\alpha \in A}$ such that $(\M, \Xi_1, \Xi_2, \Xi_3, \Theta_1, \Theta_2, \Theta_3)$ is an effective descent datum. 
Therefore the forgetful functor $\F_\eff (S, \U) \to \F_\des (S, \U)$ is essentially surjective. 

It is easy to see that $\mathit{Mor}_S (\M, \M')$ is a sheaf on the site $\Cpx_X$. 
\end{proof}

\begin{exam}
This example is cited from \cite[Example 8.2]{Alp1} and must help the readers to understand that gluing good moduli spaces is a non-trivial task. 
Consider the $\mathbb{C}^*$-action on $\mathbb{C}^2$ by the scalar multiplication. 
The quotient stack $[(\mathbb{C}^2 \setminus \{ 0 \}) /\mathbb{C}^*]$ is naturally an open sub-stack of the quotient stack $[\mathbb{C}^2/\mathbb{C}^*]$. 
Both stacks admit good moduli spaces $[(\mathbb{C}^2 \setminus \{ 0 \}) /\mathbb{C}^*] \to \mathbb{C}P^1$, $[\mathbb{C}^2/\mathbb{C}^*] \to \mathbb{C}^2 \sslash \mathbb{C}^* = \mathrm{pt}$ respectively. 
In spite of the openness of the morphism $[(\mathbb{C}^2 \setminus \{ 0 \})/\mathbb{C}^*] \to [\mathbb{C}^2 / \mathbb{C}^*]$, the induced morphism $\mathbb{C}P^1 \to \mathrm{pt}$ of good moduli spaces is not open. 
\end{exam}

\subsection*{B. Fundamental equalities for Proposition \ref{moment}}

\label{Calculation}

Let $\omega$ be a symplectic form, $J$ be a $\omega$-compatible almost complex structure, $g = \omega (-, J-)$ be the Riemannian metric associated to $(\omega, J)$ on a manifold $M$. 
Let $A$ be an endomorphism of the tangent bundle $TM$ satisfying $JA + AJ = 0$ and $\omega (AX, Y) = \omega (AY, X)$. 

On a local coordinate, we denote by $\omega^{kj}$ the tensor satisfying $\omega^{kj} \omega_{ij} = \delta^k_i$ and by $g^{kj}$ the tensor satisfying $g^{kj} g_{ij} = \delta^k_i$. 
The following rules are basic for our calculations. 
\begin{enumerate}
\item[A.] (a) $\omega_{ij} = - \omega_{ji}$, (b) $J_i^j J_j^k = - \delta^k_i$, (c) $g_{ij} = g_{ji}$. 

\item[B.] (a) $\omega_{ij} = g_{pj} J^p_i = - g_{ip} J^p_j$, (b) $g_{ij} = \omega_{iq} J^q_j = - \omega_{qj} J^q_i$. 

\item[C.] (a) $\omega^{kj} = - g^{qj} J_q^k = g^{kq} J^j_q = - \omega^{jk}$, (b) $g^{kj} = \omega^{pj} J^k_p = - \omega^{kp} J^j_p = g^{jk}$. 

\item[D.] (a) $\omega^{kj} \omega_{ij} = \omega^{jk} \omega_{ji} = \delta^k_i$, (b) $g^{kj} g_{ij} = g^{jk} g_{ji} = \delta^k_i$. 

\item[E.] $f_j = - X_f^i \omega_{ij} = X_f^i g_{pi} J^p_j = - X_f^i J^p_i g_{pj}$. 

\item[F.] $X_f^k = -f_j \omega^{kj} = f_j g^{qj} J^k_q = - f_j J^j_q g^{qk}$. 

\item[G.] (a) $(JA)^k_i = J^k_p A^p_i = - J^p_i A^k_p$, (b) $\omega_{kj} A^k_i = \omega_{ki} A^k_j$. 

\item[H.] $g^{ij} g_{kl} (JA)^l_j = (JA)^i_k$. 
\end{enumerate}

\end{document}